\documentclass[11pt,a4paper]{article}
\usepackage[T1]{fontenc}
\usepackage[utf8]{inputenc}
\usepackage{amsmath,amssymb,amsthm,mathtools}
\usepackage{newtxtext,newtxmath}
\usepackage{bm}
\usepackage{microtype}
\usepackage{geometry}
\usepackage[authoryear,round]{natbib}
\usepackage{enumitem}
\usepackage[hidelinks]{hyperref}
\allowdisplaybreaks[3]
\setlist[enumerate]{itemsep=2pt,topsep=4pt}

\newtheorem{theorem}{Theorem}[section]
\newtheorem{lemma}[theorem]{Lemma}
\newtheorem{proposition}[theorem]{Proposition}
\newtheorem{corollary}[theorem]{Corollary}
\theoremstyle{remark}
\newtheorem{remark}[theorem]{Remark}

\numberwithin{equation}{section}
\hypersetup{
  pdftitle={Exact Sign Recovery for PCA Connectivity Analysis in Mixture Models},
  pdfauthor={Kohei Kawamoto},
  pdfkeywords={principal component analysis, PCA connectivity, high-dimensional clustering, exact recovery, anisotropic sub-Gaussian noise}
}

\title{Exact Sign Recovery for PCA Connectivity Analysis\\
in Mixture Models}
\author{Kohei Kawamoto\\
\small Joint Graduate School of Mathematics for Innovation, Kyushu University\\
\small 744 Motooka, Nishi-ku, Fukuoka 819-0395, Japan\\
\small Corresponding author: \texttt{kawamoto.kohei.532@s.kyushu-u.ac.jp}\\
\small ORCID: 0009-0001-9424-0018}
\date{}

\begin{document}
\maketitle

\begin{abstract}
We study exact sign recovery for PCA connectivity analysis in mixture models with $K$ components, using the projection onto the first $(K-1)$ principal components.
For any fixed number of components with affinely independent means, we derive the oracle signal projection explicitly.
Under anisotropic sub-Gaussian noise, with a random number of observations in each component, a rowwise eigenspace bound yields a sufficient condition, governed by the weakest signal direction, for recovering all connectivity signs without requiring a bounded condition number of the nonzero signal spectrum.
Under isotropic Gaussian noise, when the configuration of the component means is fixed up to scaling and orthogonal transformations, we identify an asymptotically sharp threshold for the signal strength: exact sign recovery succeeds above the threshold and fails below the threshold, with probability tending to one when the signal strength stays a fixed relative distance from the threshold.
For an equally weighted regular simplex, we give an explicit sharp cutoff coefficient and show that optimizing the common entrywise threshold improves the coefficient's dependence on the number of components from linear growth to inverse proportionality.

\end{abstract}

\noindent\textbf{Keywords:} principal component analysis; high-dimensional clustering \\
\noindent\textbf{MSC 2020:} 62H30, 62H25

\section{Introduction}
\subsection{Overview}
In a mixture model with $K$ components, clustering aims to recover the partition of observations induced by their latent component labels.
A widely used approach is $K$-means clustering, which minimizes the within-cluster sum of squared distances to empirical group means.
Spectral relaxations of the $K$-means objective establish a connection with principal component analysis (PCA) \citep{DingHe2004,ZhaHeDingGuSimon2001}.
PCA is also used to reduce the dimension of the data before clustering.
By retaining a small number of principal component scores, it provides a lower-dimensional space in which clustering procedures can be applied.

The effectiveness of this dimension reduction depends on whether the resulting representation retains information about component membership.
Theoretical analyses of PCA-based clustering therefore seek to determine how the separation between component means, the noise distribution, the dimension, and the sample size affect clustering accuracy.
Such analyses provide bounds on misclassification rates and conditions for exact recovery of the underlying partition.
In this paper, we investigate how these factors affect exact sign recovery for PCA connectivity analysis based on the orthogonal projection onto the leading $(K-1)$ principal components.

\subsection{Related work and our contributions}
\citet{CaiZhang2018} and \citet{Vershynin2018} used first principal component score signs without sample centering for two-component Gaussian models with common isotropic covariance.
\citet{CaiZhang2018} gave bounds on the expected misclassification rate and minimax lower bounds for fixed labels and a dimension at least as large as the sample size.
\citet{Vershynin2018} bounded the misclassification rate with high probability for independent equally likely labels under signal and sample-size conditions.
\citet{CaiHanZhang2022} extended the analysis to common diagonal covariance matrices under fixed labels, deriving an upper bound on the expected misclassification rate for the same PCA rule and a minimax lower bound.
For two-component Gaussian mixtures with independent equally likely labels, \citet{KawamotoGotoTsukuda2025} studied the same uncentered rule under the allometric extension model, allowing distinct positive definite component covariances with simple largest eigenvalues and leading eigenvectors parallel to the mean difference.
\citet{KawamotoGotoTsukuda2026} allowed general positive definite component covariances without this alignment.
Both studies bound individual misclassification probabilities and give high-dimensional sufficient conditions for exact recovery of the label vector up to sign under signal-to-noise and sample-size conditions.
Using sample-centered PCA, \citet{AzizyanSinghWasserman2013} studied first principal component score signs in equally weighted two-component Gaussian mixtures with unknown means and common isotropic covariance.
Their upper and minimax lower bounds concern the expected disagreement with the Bayes-optimal clustering rule, rather than errors relative to the latent component labels.

For $K\ge3$, centered indicators of $K$ nonempty components span $K-1$ dimensions.
When the component means are affinely independent, the column space of the noise-free centered Gram matrix coincides with the span of the centered cluster indicators.
\citet{LofflerZhangZhou2021} and \citet{ZhangZhou2024} apply $K$-means to the rows of the leading right singular-vector matrix of the uncentered data, scaled by the corresponding singular values, under fixed but unknown labels.
\citet{LofflerZhangZhou2021} obtained high-probability, optimal exponential misclassification bounds for sufficiently separated isotropic Gaussian mixtures without spectral gap assumptions.
\citet{ZhangZhou2024} used a leave-one-out analysis to derive exponential bounds on the expected misclassification rate under independent sub-Gaussian noise and conditions on group sizes and signal strength, without assuming bounded signal condition numbers.
For isotropic Gaussian mixtures, they obtained the minimax-optimal exponential misclassification rate under a signal-strength condition with improved dependence on $K$ compared with \citet{LofflerZhangZhou2021}.
\citet{AbbeFanWang2022} derived perturbation bounds for eigenvectors and spectral embeddings obtained from a hollowed Gram matrix of uncentered data.
Their clustering guarantees allow anisotropic, heteroscedastic sub-Gaussian noise under separation and regularity conditions, including a bounded condition number for the nonzero signal eigenvalues.
\citet{ChenZhang2024} established minimax-optimal misclassification rates for covariance-adjusted Lloyd algorithms under Gaussian mixture models with unknown common or component-specific anisotropic covariance matrices, assuming suitable initialization, bounded covariance eigenvalue ratios, and dimension restrictions.

Alongside spectral relaxations, semidefinite programming (SDP) provides a tighter relaxation of the $K$-means partition formulation \citep{PengWei2007}.
\citet{MixonVillarWard2017} bounded the squared Frobenius distance between the SDP solution and the normalized matrix corresponding to the true partition for sub-Gaussian mixtures under separation, group-size, and sample-size conditions.
\citet{AwasthiEtAl2015} established exact recovery for the SDP relaxation of $K$-means and the linear programming relaxation of $K$-median under a stochastic ball model.
\citet{FeiChen2018} derived exponential bounds on estimation and misclassification errors for an SDP relaxation of $K$-means under isotropic sub-Gaussian mixtures with equal-sized groups and i.i.d. noise entries, with exact recovery under stronger separation conditions.
\citet{Royer2017} obtained exact partition recovery under anisotropic sub-Gaussian noise using a bias-corrected SDP adapting to effective dimension and allowing unknown $K$.

\citet{Ndaoud2022} established a sharp threshold for exact recovery in a two-component Gaussian model with isotropic noise and a fixed but unknown label vector, based on a minimax analysis of the expected misclassification rate up to sign.
For Gaussian mixtures with common spherical covariance, \citet{ChenYang2021} obtained an SDP exact-label-recovery cutoff for equal-sized groups and a minimax lower bound on the probability of any label error, up to relabeling, over approximately balanced groups.

\citet{LittleXieSun2023} studied classical multidimensional scaling using the Gram matrix of observations centered by the sample mean.
Under sub-Gaussian noise and conditions including a bounded condition number for the retained signal eigenvalues, they established entrywise error bounds for the eigenvectors and the eigenvalue-weighted embedding.
Their embedding bounds yield exact label recovery by suitable distance-based clustering under sufficient separation.

For the same Gram matrix, we study the unnormalized orthogonal projection onto its leading $(K-1)$-dimensional eigenspace; see \eqref{eq:kmeans-spectral-solution}.
This projection uses no eigenvalue weights and is independent of the choice of orthonormal basis for that eigenspace.
We investigate exact entrywise sign recovery, requiring positive entries within components and negative entries between components.
For $K\ge3$, exact label recovery from an embedding need not imply this sign pattern.

For fixed $K$, fixed positive mixing proportions, and affinely independent component means, we derive the oracle projection conditional on the labels, allowing random component sizes.
Theorems~\ref{thm:rowwise} and~\ref{thm:sign-recovery} give rowwise eigenspace and projection bounds under anisotropic sub-Gaussian noise, yielding a sufficient condition for exact sign recovery governed by the weakest signal direction without requiring a bounded condition number for the nonzero signal spectrum.
Under isotropic Gaussian noise with the mean configuration fixed up to scaling and orthogonal transformations, Theorem~\ref{thm:gaussian-cutoff} gives a sharp cutoff: the probability of exact sign recovery tends to one above and zero below the cutoff when the signal stays a fixed positive relative distance away.

For equally weighted regular simplices, Corollary~\ref{cor:gaussian-simplex} gives the cutoff coefficient at zero threshold, while for $K\ge3$, Theorem~\ref{thm:simplex-threshold} identifies the optimal positive threshold on the projection scaled by the sample size, reducing the coefficient from order $K$ to order $K^{-1}$.
The orders concern separate results for each fixed $K$ and do not guarantee recovery for $K=K_n\to\infty$.
Optimality among fixed common scalar thresholds does not imply minimax optimality for label recovery (Subsection~\ref{sec:cutoff-comparison}).

\subsection{Notation}
We use $\|\cdot\|_2$, $\|\cdot\|_F$, and $\|\cdot\|_{\mathrm{op}}$ for the Euclidean, Frobenius, and operator norms, respectively.
For a matrix $\bm{B}$, $B_{i*}$ denotes its $i$th row, and
\begin{equation*}
 \|\bm{B}\|_{\max}=\max_{i,j}|B_{ij}|,\qquad
 \|\bm{B}\|_{2\to\infty}=\max_i\|B_{i*}\|_2.
\end{equation*}
For a real-valued random variable $X$, the sub-Gaussian norm $\|\cdot\|_{\psi_2}$ is defined by
\begin{equation*}
 \|X\|_{\psi_2}=\inf\{t>0:\mathbb{E}[\exp(X^2/t^2)]\le2\}.
\end{equation*}
The symbols $\mathbf{1}_d$ and $\bm{I}_d$ denote the vector of ones and identity matrix, while $\mathbb I_{\mathcal E}$ and $\mathcal E^{\mathrm c}$ denote the indicator and complement of an event $\mathcal E$.
We write $\mathrm{rank}$, $\mathrm{Ker}$, $\mathrm{Im}$, and $\mathrm{span}$ for rank, kernel, image, and linear span, respectively, and $\mathcal{O}(d)=\{\bm{R}\in\mathbb{R}^{d\times d}:\bm{R}^{\mathsf{T}}\bm{R}=\bm{I}_d\}$ for the orthogonal group.
For a symmetric $d\times d$ matrix $\bm{B}$, $\lambda_j(\bm{B})$ denotes its $j$th largest eigenvalue $(j\in\{1,\ldots,d\})$.
For symmetric matrices, $\bm{B}\preceq\bm{C}$, equivalently $\bm{C}\succeq\bm{B}$, means that $\bm{C}-\bm{B}$ is positive semidefinite; $\bm{B}\prec\bm{C}$, equivalently $\bm{C}\succ\bm{B}$, means that $\bm{C}-\bm{B}$ is positive definite.
Square roots of positive semidefinite matrices are taken to be symmetric positive semidefinite.

The notation $X_n\xrightarrow{p}X$ means that $\mathbb{P}(\|X_n-X\|>\delta)\to0$ for every $\delta>0$, with the relevant norm for vectors or matrices and absolute value for scalars.
For a positive deterministic sequence $a_n$, $x_n=O(a_n)$ means $|x_n|\le Ca_n$ eventually for some $C<\infty$, and $x_n=o(a_n)$ means $x_n/a_n\to0$.
The relation $X_n=O_p(a_n)$ means that $X_n/a_n$ is bounded in probability, and $X_n=o_p(a_n)$ means $X_n/a_n\xrightarrow{p}0$.
For vectors and matrices, deterministic and stochastic orders refer to the relevant norm.
For positive deterministic sequences $a_n,b_n$, we use
\begin{equation*}
\begin{aligned}
 &a_n\asymp b_n\quad\Longleftrightarrow\quad cb_n\le a_n\le Cb_n\ \text{eventually for some }0<c\le C<\infty,\\
 &a_n\gtrsim b_n\quad\Longleftrightarrow\quad a_n\ge cb_n\ \text{eventually for some }c>0,\\
 &a_n\sim b_n\quad\Longleftrightarrow\quad a_n/b_n\to1.
\end{aligned}
\end{equation*}
Implicit constants do not depend on $n$; dependence on fixed model parameters is specified with the relevant results.
The symbol $\sim$ also denotes a distributional law when followed by a distribution.

\section{Model and PCA connectivity analysis}
\label{sec:model}
\subsection{Mixture model and PCA connectivity}
Fix an integer $K\ge2$, put $r=K-1$, and consider sample sizes $n\ge K$ and dimensions $p=p_n\ge r$.
We study the asymptotic behavior as $n\to\infty$, allowing $p_n\to\infty$.
Let $\pi_1,\ldots,\pi_K>0$ be fixed mixing proportions with $\sum_{k=1}^K\pi_k=1$.
Write $\bm{\pi}=(\pi_1,\ldots,\pi_K)^{\mathsf{T}}$ and $\pi_{\min}=\min_{k\in\{1,\ldots,K\}}\pi_k$, so $\pi_{\min}>0$.
Let the observations $\bm{x}_1,\ldots,\bm{x}_n\in\mathbb{R}^p$ satisfy
\begin{equation*}
 \bm{x}_i=\widetilde{\bm{\mu}}_{Z_i}+\bm{g}_i,
 \qquad i\in\{1,\ldots,n\},
\end{equation*}
where $Z_1,\ldots,Z_n$ are independent labels with $\mathbb{P}(Z_i=k)=\pi_k$ for $k\in\{1,\ldots,K\}$.
The vectors $\bm{g}_1,\ldots,\bm{g}_n$ are mean-zero noise vectors.
Under the noise assumptions below, $\widetilde{\bm{\mu}}_k=\mathbb{E}[\bm{x}_i\mid Z_i=k]$ is the mean of component $k$ for each $k\in\{1,\ldots,K\}$.
The population mean is
\begin{equation*}
 \overline{\bm{\mu}}=\mathbb{E}[\bm{x}_i]=\sum_{k=1}^K\pi_k\widetilde{\bm{\mu}}_k.
\end{equation*}
For our theoretical analysis, define $\bm{\mu}_k=\widetilde{\bm{\mu}}_k-\overline{\bm{\mu}}$ for $k\in\{1,\ldots,K\}$, so that $\sum_k\pi_k\bm{\mu}_k=\bm0$.
For each $k\in\{1,\ldots,K\}$, denote by $\mathcal G_k$, $n_k$, and $\widehat\pi_k$ the index set, size, and empirical proportion of component $k$, respectively:
\begin{equation*}
 \mathcal G_k=\{i\in\{1,\ldots,n\}:Z_i=k\},\qquad
 n_k=|\mathcal G_k|,\qquad \widehat\pi_k=n_k/n.
\end{equation*}
We center the observations by the sample mean:
\begin{equation*}
  \bm{y}_i=\bm{x}_i-\overline{\bm{x}},\qquad
  \overline{\bm{x}}=\frac1n\sum_{i=1}^n\bm{x}_i.
\end{equation*}
For the matrix representation of sample centering, define
\begin{equation}
 \bm{e}=n^{-1/2}\mathbf{1}_n,\qquad \bm{J}_n=\bm{I}_n-\bm{e}\bm{e}^{\mathsf{T}}.
 \label{eq:centering-matrix}
\end{equation}
Writing $\overline{\bm{\mu}}_Z=\sum_k\widehat\pi_k\bm{\mu}_k$ and $\overline{\bm{g}}=n^{-1}\sum_i\bm{g}_i$, we have
\begin{equation*}
 \bm{y}_i=\bm{\mu}_{Z_i}-\overline{\bm{\mu}}_Z+\bm{g}_i-\overline{\bm{g}},\qquad i\in\{1,\ldots,n\}.
\end{equation*}

We assume that the centered component means form a matrix $\bm{M}_{\bm{\mu}}=(\bm{\mu}_1,\ldots,\bm{\mu}_K)\in\mathbb{R}^{p\times K}$ satisfying
\begin{equation}
 \mathrm{rank}(\bm{M}_{\bm{\mu}})=r.
 \label{eq:full-signal-rank}
\end{equation}
Since $\bm{M}_{\bm{\mu}}\bm{\pi}=\bm0$, \eqref{eq:full-signal-rank} is equivalent to
\begin{equation}
 \mathrm{Ker}(\bm{M}_{\bm{\mu}})=\mathrm{span}\{\bm{\pi}\},
 \label{eq:mean-kernel}
\end{equation}
and to affine independence of $\widetilde{\bm{\mu}}_1,\ldots,\widetilde{\bm{\mu}}_K$, meaning that $\widetilde{\bm{\mu}}_1-\widetilde{\bm{\mu}}_K,\ldots,\widetilde{\bm{\mu}}_{K-1}-\widetilde{\bm{\mu}}_K$ are linearly independent.
For $K=3$, the assumption is equivalent to linear independence of $\bm{\mu}_1$ and $\bm{\mu}_2$.

The noise vectors $\bm{g}_1,\ldots,\bm{g}_n$ are independent of $Z_1,\ldots,Z_n$ and have the representation
\begin{equation}
 \bm{g}_i=\bm{\Sigma}^{1/2}\bm{\xi}_i,\qquad i\in\{1,\ldots,n\},
 \label{eq:noise-model}
\end{equation}
where $\bm{\Sigma}$ is a symmetric positive semidefinite matrix.
Write $\bm{\xi}_i=(\xi_{i1},\ldots,\xi_{ip})^{\mathsf{T}}\in\mathbb{R}^p$.
All coordinates in $\{\xi_{ij}:i\in\{1,\ldots,n\},\ j\in\{1,\ldots,p\}\}$ are independent, and there is a constant $\kappa>0$, independent of $n$, such that
\begin{equation}
 \mathbb{E}[\xi_{ij}]=0,\qquad \mathbb{E}[\xi_{ij}^2]=1,\qquad
 \|\xi_{ij}\|_{\psi_2}\le\kappa,
 \qquad i\in\{1,\ldots,n\},\ j\in\{1,\ldots,p\}.
 \label{eq:subgaussian-assumption}
\end{equation}
Thus $\mathbb{E}[\bm{g}_i]=\bm0$ and $\mathbb{E}[\bm{g}_i\bm{g}_i^{\mathsf{T}}]=\bm{\Sigma}$.
Apart from satisfying \eqref{eq:subgaussian-assumption}, the coordinate distributions need not be identical.

Let $\bm{G}=(\bm{g}_1,\ldots,\bm{g}_n)$ and define
\begin{equation}
 \begin{aligned}
 &\bm{Y}_{\bm{\mu}}=(\bm{\mu}_{Z_1},\ldots,\bm{\mu}_{Z_n})\bm{J}_n,\\
 &\bm{Y}=(\bm{y}_1,\ldots,\bm{y}_n)=\bm{Y}_{\bm{\mu}}+\bm{G}\bm{J}_n,\\
 &\bm{A}=\bm{Y}_{\bm{\mu}}^{\mathsf{T}}\bm{Y}_{\bm{\mu}}.\end{aligned}
 \label{eq:signal-data-Gram}
\end{equation}
We call $\bm{A}$ the oracle signal Gram matrix after sample centering.
On the event $\{\min_{k\in\{1,\ldots,K\}}n_k>0\}$, $\mathrm{rank}(\bm{A})=r$.
Denote the positive eigenvalues of $\bm{A}$ by $\lambda_1\ge\cdots\ge\lambda_r>0$, and let $\bm{V}\in\mathbb{R}^{n\times r}$ have orthonormal columns spanning the corresponding eigenspace.
The oracle signal projection is
\begin{equation*}
 \bm{P}=\bm{V}\bm{V}^{\mathsf{T}}.
\end{equation*}
The positive eigenspace of $\bm{A}$ and the projection $\bm{P}$ are unique, although the basis need not be: $(\bm{V}\bm{R})(\bm{V}\bm{R})^{\mathsf{T}}=\bm{P}$ for every orthogonal $\bm{R}\in\mathbb{R}^{r\times r}$.
Here oracle refers to the noise-free centered signal conditional on the labels, not to the unconditional expectation of the sample Gram matrix.

Let $\widehat{\bm{V}}\in\mathbb{R}^{n\times r}$ have orthonormal columns spanning a leading $r$-dimensional eigenspace of the sample Gram matrix $\bm{Y}^{\mathsf{T}}\bm{Y}$.
Since $\bm{Y}\bm{e}=\bm0$, eigenvectors with positive eigenvalues lie in $\bm{e}^\perp$.
If fewer than $r$ eigenvalues are positive, choose the remaining basis vectors from $\mathrm{Ker}(\bm{Y})\cap\bm{e}^\perp$.
Define
\begin{equation}
 \widehat{\bm{P}}=\widehat{\bm{V}}\widehat{\bm{V}}^{\mathsf{T}}.
 \label{eq:sample-connectivity}
\end{equation}
Write $\widehat P_{ij}$ for the $(i,j)$ entry of $\widehat{\bm{P}}$.
For distinct indices $i,j\in\{1,\ldots,n\}$, the sign rule classifies the two observations as belonging to the same component if $\widehat P_{ij}>0$ and to different components otherwise.
The sign rule requires only $K$, not $\overline{\bm{\mu}}$, $\widetilde{\bm{\mu}}_1,\ldots,\widetilde{\bm{\mu}}_K$, $\bm{\pi}$, or $\bm{\Sigma}$, and involves no additional $K$-means step.

On the event $\{\min_{1\le k\le K}n_k=0\}$, choose $\bm{V}$ measurably as an $n\times r$ matrix with orthonormal columns in $\bm{e}^\perp$ and set $\bm{P}=\bm{V}\bm{V}^{\mathsf{T}}$.
A union bound gives $\mathbb{P}\{\min_{1\le k\le K}n_k=0\}\le\sum_{k=1}^K(1-\pi_k)^n$.
If $\lambda_r(\bm{Y}^{\mathsf{T}}\bm{Y})=\lambda_{r+1}(\bm{Y}^{\mathsf{T}}\bm{Y})$, choose the leading $r$-dimensional sample eigenspace measurably, following the convention above.

The random vector $\bm{\mu}_{Z_1}$ has mean $\mathbb{E}[\bm{\mu}_{Z_1}]=\bm0$ and covariance matrix
\begin{equation*}
 \bm{B}_{\bm{\mu}}=\operatorname{Cov}(\bm{\mu}_{Z_1})
 =\sum_{k=1}^K\pi_k\bm{\mu}_k\bm{\mu}_k^{\mathsf{T}},
\end{equation*}
where $L_K=\lambda_r(\bm{B}_{\bm{\mu}})>0$ denotes its smallest positive eigenvalue and measures the signal strength.
The noise scales are
\begin{equation*}
 \rho=\|\bm{\Sigma}\|_{\mathrm{op}},\qquad
 \nu=\sqrt{\mathrm{tr}(\bm{\Sigma}^2)},\qquad
 a_n=\nu\sqrt{\log n}+\rho\log n.
\end{equation*}
The dimension $p=p_n$, covariance $\bm{\Sigma}=\bm{\Sigma}_n$, and centered means $\bm{\mu}_k=\bm{\mu}_{k,n}$ may depend on $n$.
Unless a conditional distribution is explicitly specified, probabilities and stochastic orders are considered under the joint distribution of $Z_1,\ldots,Z_n$ and $\bm{g}_1,\ldots,\bm{g}_n$.

\begin{remark}[Sample centering]
The sample Gram matrix $\bm{Y}^{\mathsf{T}}\bm{Y}$ is formed after centering the observations by their sample mean.
\citet{LittleXieSun2023} use the same Gram matrix for classical multidimensional scaling.
\end{remark}

\begin{remark}[Normalization and thresholding]
At threshold zero, our sign rule uses the rank-$(K-1)$ projection $\widehat{\bm{P}}$.
\citet{DingHe2004} normalize the connectivity entries and apply a positive threshold.
When $\widehat P_{ii},\widehat P_{jj}>0$, normalization preserves the sign of $\widehat P_{ij}$ because its denominator $\sqrt{\widehat P_{ii}\widehat P_{jj}}$ is positive.
Applying a positive threshold to the normalized entry gives a different recovery criterion (Subsection~\ref{sec:cutoff-comparison}).
\end{remark}

\subsection{Relation to relaxations of \texorpdfstring{$K$}{K}-means}
\label{sec:kmeans-relaxations}
For a partition $\mathcal C=(C_1,\ldots,C_K)$ of $\{1,\ldots,n\}$ into nonempty sets, let $\mathbf{1}_{C_k}\in\mathbb{R}^n$ denote the membership indicator vector and put
\begin{equation*}
 \bm{B}_{\mathcal C}=\sum_{k=1}^K\frac{\mathbf{1}_{C_k}\mathbf{1}_{C_k}^{\mathsf{T}}}{|C_k|},
 \qquad
 \overline{\bm{y}}_{C_k}=\frac1{|C_k|}\sum_{i\in C_k}\bm{y}_i.
\end{equation*}
The $K$-means objective has the matrix form \citep{PengWei2007}
\begin{equation}
 \sum_{k=1}^K\sum_{i\in C_k}\|\bm{y}_i-\overline{\bm{y}}_{C_k}\|_2^2
 =\mathrm{tr}(\bm{Y}^{\mathsf{T}}\bm{Y})-\mathrm{tr}(\bm{Y}^{\mathsf{T}}\bm{Y}\bm{B}_{\mathcal C}).
 \label{eq:kmeans-objective}
\end{equation}
Every $\bm{B}_{\mathcal C}$ is a symmetric, entrywise nonnegative orthogonal projection with $\bm{B}_{\mathcal C}\mathbf{1}_n=\mathbf{1}_n$ and $\mathrm{tr}(\bm{B}_{\mathcal C})=K$.
Define $\mathcal F_{\mathrm{part}}=\{\bm{B}_{\mathcal C}:\mathcal C\text{ is a partition into }K\text{ nonempty sets}\}$.
Let $\mathbb{S}^n=\{\bm{B}\in\mathbb{R}^{n\times n}:\bm{B}=\bm{B}^{\mathsf{T}}\}$ be the space of real symmetric matrices.
Define two subsets of $\mathbb{S}^n$ by
\begin{align}
 &\mathcal F_{\mathrm{SDP}}
 =\{\bm{B}\in\mathbb{S}^n:\bm{B}\succeq\bm0,\ \bm{B}\mathbf{1}_n=\mathbf{1}_n,\ \mathrm{tr}(\bm{B})=K, B_{ij}\ge0\text{ for every }i,j\in\{1,\ldots,n\}\},\label{eq:kmeans-feasible-sets1}\\
 &\mathcal F_{\mathrm{sp}}
 =\{\bm{B}\in\mathbb{S}^n:\bm0\preceq\bm{B}\preceq\bm{I}_n,
       \ \bm{B}\mathbf{1}_n=\mathbf{1}_n,\ \mathrm{tr}(\bm{B})=K\} .\end{align}
For $\bm{B}\in\mathcal F_{\mathrm{SDP}}$, the entrywise nonnegativity and unit row sums in \eqref{eq:kmeans-feasible-sets1} imply $|\lambda_j(\bm{B})|\le1$ for every $j\in\{1,\ldots,n\}$.
Together with $\bm{B}\succeq\bm0$, the eigenvalue bound gives $\bm0\preceq\bm{B}\preceq\bm{I}_n$.
Consequently,
\begin{equation}
 \mathcal F_{\mathrm{part}}\subseteq\mathcal F_{\mathrm{SDP}}
 \subseteq\mathcal F_{\mathrm{sp}}.
 \label{eq:kmeans-inclusions}
\end{equation}
For the nested feasible sets in \eqref{eq:kmeans-inclusions}, maximizing $\mathrm{tr}(\bm{Y}^{\mathsf{T}}\bm{Y}\bm{B})$ over $\mathcal F_{\mathrm{SDP}}$ gives the Peng--Wei relaxation.
Maximizing over $\mathcal F_{\mathrm{sp}}$ gives a spectral relaxation \citep{DingHe2004,PengWei2007}.
Since $\bm{Y}\bm{e}=\bm0$, the spectral relaxation admits the optimizer
\begin{equation}
 \widehat{\bm{B}}_{\mathrm{sp}}=\bm{e}\bm{e}^{\mathsf{T}}+\widehat{\bm{P}},
 \qquad
 \widehat P_{ij}=(\widehat{\bm{B}}_{\mathrm{sp}})_{ij}-\frac1n.
 \label{eq:kmeans-spectral-solution}
\end{equation}
For $\bm{B}\in\mathcal F_{\mathrm{sp}}$, the restriction of $\bm{B}-\bm{e}\bm{e}^{\mathsf{T}}$ to $\bm{e}^\perp$ has eigenvalues in $[0,1]$ and trace $K-1$; the spectral maximum is attained by the projection onto the leading $K-1$ directions.
Thus the sign rule based on \eqref{eq:sample-connectivity} tests whether an entry of $\widehat{\bm{B}}_{\mathrm{sp}}$ exceeds $1/n$.
For the true partition $\mathcal C=(\mathcal G_1,\ldots,\mathcal G_K)$ with nonempty components, Proposition~\ref{prop:population-structure} gives $\bm{B}_{\mathcal C}-\bm{e}\bm{e}^{\mathsf{T}}=\bm{P}$.
Exact recovery of the signs of $\widehat{\bm{P}}$ requires neither $\widehat{\bm{B}}_{\mathrm{sp}}=\bm{B}_{\mathcal C}$ nor entrywise nonnegativity of $\widehat{\bm{B}}_{\mathrm{sp}}$.
Theorem~\ref{thm:sign-recovery} therefore concerns the discrete information retained by a spectral relaxation, rather than exact recovery of the normalized partition matrix or an optimality certificate for \eqref{eq:kmeans-objective}.

\section{Main results for the sub-Gaussian model}
In the sub-Gaussian results, constants $c,C,C_B$ may depend on $K$, $\kappa$, and $\pi_1,\ldots,\pi_K$, and $C_B$ may also depend on $B$.
The constants $c,C,C_B$ do not depend on $n$, $p$, $\bm{\Sigma}$, or $\bm{\mu}_1,\ldots,\bm{\mu}_K$, and may change from line to line.
Constants in the Gaussian results may additionally depend on the fixed vectors $\bm{m}_1,\ldots,\bm{m}_K$ introduced in Section~\ref{sec:gaussian-main}.

\subsection{Oracle signal projection}
\begin{proposition}[Structure of the oracle signal projection]
\label{prop:population-structure}
Whenever all groups are nonempty, for $i\in\mathcal G_k$ and $j\in\mathcal G_\ell$,
\begin{equation}
 P_{ij}=\frac{\mathbb I_{\{k=\ell\}}}{n_k}-\frac1n.
  \label{eq:oracle-projection-entries}
\end{equation}
Consequently, $P_{ij}>0$ if $Z_i=Z_j$ and $P_{ij}<0$ if $Z_i\ne Z_j$.
Moreover, for every fixed $B>0$, with probability $1-O(n^{-B})$,
\begin{equation*}
 \lambda_r\asymp nL_K,\qquad
 \frac{c}{n}\le |P_{ij}|\le\frac{C}{n}\quad(i,j\in\{1,\ldots,n\}),
 \qquad \|\bm{V}\|_{2\to\infty}\le Cn^{-1/2}.
\end{equation*}
\end{proposition}

\begin{proposition}[A necessary rank condition]
\label{prop:rank-necessity}
Without imposing \eqref{eq:full-signal-rank}, suppose that $K\ge2$, all $K$ groups are nonempty, $\pi_k>0$, and $\bm{M}_{\bm{\mu}}\bm{\pi}=\bm0$.
Let $\bm{P}_{\mathrm{sig}}$ be the orthogonal projection onto the entire positive eigenspace of $\bm{Y}_{\bm{\mu}}^{\mathsf{T}}\bm{Y}_{\bm{\mu}}$.
If $(P_{\mathrm{sig}})_{ij}<0$ for every pair from distinct groups, then $\mathrm{rank}(\bm{M}_{\bm{\mu}})=K-1$.
\end{proposition}

\begin{remark}[Distinct means are not enough]
Proposition~\ref{prop:rank-necessity} does not require \eqref{eq:full-signal-rank}.
For $K=4$ and equal mixing proportions, take $\bm{\mu}_1=(1,0,0)^{\mathsf{T}}$, $\bm{\mu}_2=(-1,0,0)^{\mathsf{T}}$, $\bm{\mu}_3=(0,1,0)^{\mathsf{T}}$, and $\bm{\mu}_4=(0,-1,0)^{\mathsf{T}}$.
Here $p=3=K-1$, but the mean matrix $\bm{M}_{\bm{\mu}}$ has rank $2$ (which is less than $3$).
If each group contains $m$ observations, the nonzero rows of $\bm{Y}_{\bm{\mu}}$ are orthogonal with squared norm $2m$.
Thus
\begin{equation*}
        (P_{\mathrm{sig}})_{ij}
        =\frac{\bm{\mu}_k^{\mathsf{T}}\bm{\mu}_\ell}{2m},
        \qquad i\in\mathcal G_k,\ j\in\mathcal G_\ell.
\end{equation*}
The entry between groups 1 and 3 is zero.
A third leading eigenvector corresponds to a zero eigenvalue and is not determined by the signal.
Proposition~\ref{prop:rank-necessity} concerns the strict negativity of $\bm{P}_{\mathrm{sig}}$ between all distinct groups, not impossibility for other methods.
\end{remark}

\subsection{Rowwise consistency and exact sign recovery}
Consider the signal-to-noise condition
\begin{equation}
 \frac{\rho\log n+\nu\sqrt{(\log n)/n}}{L_K}\to0.
 \label{eq:main-condition}
\end{equation}
For $\bm{\Sigma}=\bm0$, the left-hand side of \eqref{eq:main-condition} is zero.
Since $a_n=\nu\sqrt{\log n}+\rho\log n$, the condition is equivalent to $(\rho\log n+a_n/\sqrt n)/L_K\to0$.
Define
\begin{equation}
 \epsilon_n
 =
 \sqrt{\frac{\rho\log n}{L_K}}
 +
 \frac{a_n}{\sqrt n\,L_K}.
 \label{eq:epsilon-main}
\end{equation}
Under \eqref{eq:main-condition}, $\epsilon_n\to0$.

\begin{theorem}[Rowwise consistency of the leading eigenspace]
\label{thm:rowwise}
Assume \eqref{eq:main-condition}.
Set $\bm{O}=\bm{V}^{\mathsf{T}}\widehat{\bm{V}}$ and let $\bm{R}\in\mathcal{O}(r)$ be a measurable orthogonal polar factor of $\bm{O}$, so that $\bm{O}=\bm{R}(\bm{O}^{\mathsf{T}}\bm{O})^{1/2}$.
For every fixed $B>0$, there is a deterministic integer $N_B$ such that, for every integer $n\ge N_B$, with probability $1-O(n^{-B})$,
\begin{equation}
 \|\widehat{\bm{V}}-\bm{V}\bm{R}\|_{2\to\infty}
 \le C_B\frac{\epsilon_n}{\sqrt n}.
 \label{eq:rowwise-rate}
\end{equation}
On the event where \eqref{eq:rowwise-rate} holds, $\bm{O}$ is nonsingular and the orthogonal polar factor of $\bm{O}$ is unique.
In particular, $\|\widehat{\bm{V}}-\bm{V}\bm{R}\|_{2\to\infty}=o_p(n^{-1/2})$.
\end{theorem}

\begin{theorem}[Exact sign recovery]
\label{thm:sign-recovery}
Assume \eqref{eq:main-condition}.
For every fixed $B>0$, there is a deterministic integer $N_B$ such that, for every integer $n\ge N_B$, with probability $1-O(n^{-B})$,
\begin{equation*}
 \|\widehat{\bm{P}}-\bm{P}\|_{\max}\le C_B\frac{\epsilon_n}{n}.
\end{equation*}
Consequently, $\|\widehat{\bm{P}}-\bm{P}\|_{\max}=o_p(n^{-1})$, and, for every fixed $B>0$,
\begin{equation*}
 \mathbb{P}\{\mathrm{sign}(\widehat P_{ij})=\mathrm{sign}(P_{ij})
                  \text{ for all }i,j\in\{1,\ldots,n\}\}=1-O(n^{-B}).
\end{equation*}
\end{theorem}

The next two corollaries express the sufficient condition in terms of the noise covariance and, for three components, the geometry of the means.

\begin{corollary}[A sufficient condition in terms of stable rank]
\label{cor:effective-rank}
Suppose $\bm{\Sigma}\ne\bm0$ (equivalently, $\rho>0$) and set $r_2(\bm{\Sigma})=\mathrm{tr}(\bm{\Sigma}^2)/\rho^2$.
If
\begin{equation*}
 \frac{\rho\{\log n+\sqrt{r_2(\bm{\Sigma})(\log n)/n}\}}{L_K}
                  \to0,
\end{equation*}
then the conclusions of Theorems~\ref{thm:rowwise} and~\ref{thm:sign-recovery} hold.
When $\bm{\Sigma}=\sigma^2\bm{I}_p$, a sufficient condition is
\begin{equation*}
 \frac{\sigma^2\{\log n+\sqrt{p(\log n)/n}\}}{L_K}\to0.
\end{equation*}
\end{corollary}

\begin{corollary}[Angle condition for three components]
\label{cor:three-angle}
Let $K=3$ and let $\theta\in(0,\pi)$ be the angle between $\bm{\mu}_1$ and $\bm{\mu}_2$.
Put $L_{\mathrm{ang}}=\min\{\|\bm{\mu}_1\|_2^2,\|\bm{\mu}_2\|_2^2\}\sin^2\theta$.
Then $L_3\asymp L_{\mathrm{ang}}$.
In particular, the conclusions of Theorems~\ref{thm:rowwise} and~\ref{thm:sign-recovery} hold if
\begin{equation*}
 \frac{\rho\log n+\nu\sqrt{(\log n)/n}}{L_{\mathrm{ang}}}\to0.
\end{equation*}
\end{corollary}

\begin{remark}[Interpretation of exact sign recovery]
Theorem~\ref{thm:sign-recovery} recovers all relations within and between components, and hence the partition up to relabeling, with probability tending to one.
Under \eqref{eq:main-condition}, Lemma~\ref{lem:eig-lower} guarantees a gap between the $r$th and $(r+1)$th sample eigenvalues with probability tending to one.
With probability tending to one, the uniform error in Theorem~\ref{thm:sign-recovery} is smaller than the minimum absolute oracle entry $\min_{i,j}|P_{ij}|$, which is bounded below by $c/n$ in Proposition~\ref{prop:population-structure}.
The sufficient condition is therefore governed by the weakest signal direction $L_K$, even when the other signal directions are strong.
For $K=3$, near collinearity weakens the second signal direction through $L_{\mathrm{ang}}$.
More generally, strict negativity between components of the projection onto the positive signal space requires rank $K-1$; see Proposition~\ref{prop:rank-necessity}.
\end{remark}

\subsection{Comparison with spectral embedding results}
\label{sec:spectral-comparison}
\citet{LittleXieSun2023} analyze the same sample-centered Gram matrix $\bm{Y}^{\mathsf{T}}\bm{Y}$, bounding both its unweighted eigenvectors and its eigenvalue-weighted embedding.
Their general signal-plus-noise model allows approximate low rank and independent sub-Gaussian noise vectors, including heterogeneous distributions.
To compare eigenvector bounds, we specialize the sub-Gaussian noise model in Section~\ref{sec:model} to $\bm{\Sigma}=\sigma^2\bm{I}_p$, where $\sigma>0$, and assume fixed $K$, $\mathrm{rank}(\bm{M}_{\bm{\mu}})=K-1$, and a bounded ratio $\lambda_1(\bm{B}_{\bm{\mu}})/L_K$.
Conditional on labels with $n_k\asymp n$, their centered signal matrix is $\bm{Y}_{\bm{\mu}}^{\mathsf{T}}$, and Proposition~\ref{prop:population-structure} and Lemma~\ref{lem:signal-eigenvalues} give $\lambda_r\asymp nL_K$, $\lambda_1/\lambda_r=O(1)$, and bounded eigenvector incoherence.
The count concentration in Lemma~\ref{lem:count-concentration} makes the conditional comparison valid under our random-label model.
The noise variance proxy is comparable to $\sigma^2$ for fixed $\kappa$.
Their entrywise eigenvector error bound and our rowwise bound concern the same eigenspaces, although the orthogonal alignments may differ.
For their orthogonal alignment $\bm{Q}$, put $\bm{E}=\widehat{\bm{V}}\bm{Q}-\bm{V}$ and $\delta=\|\bm{E}\|_{2\to\infty}$.
Norm equivalence and the Frobenius optimality of the polar factor in Theorem~\ref{thm:rowwise} give
\begin{equation*}
\begin{aligned}
 &\|\bm{E}\|_{\max}\le\delta\le\sqrt r\,\|\bm{E}\|_{\max},\\
 &\|\widehat{\bm{V}}-\bm{V}\bm{R}\|_{2\to\infty}
 \le(1+2\sqrt n\,\|\bm{V}\|_{2\to\infty})\delta.
\end{aligned}
\end{equation*}
Thus their entrywise eigenvector bound also controls our polar-aligned rowwise error up to a constant in the common model; the choice of norm alone does not yield an improved rate.

We now compare the sufficient signal conditions and the resulting eigenvector error bounds under these common assumptions.
Write $s_n=L_K/\sigma^2$, $b_{n,p}=\sqrt p\log n+p/n$, and $g_{n,p}=\log n+\sqrt{p(\log n)/n}$.
In the sufficient signal conditions below, the constants implicit in $\gtrsim$ are taken sufficiently large.
The sufficient signal condition for their eigenvector bound becomes $s_n\gtrsim b_{n,p}$, whereas \eqref{eq:main-condition} becomes $g_{n,p}/s_n\to0$.
In particular, for $p\asymp n$, the former requires $L_K\gtrsim\sigma^2\sqrt n\log n$, while the latter requires only $\sigma^2\log n/L_K\to0$.
The corresponding upper-bound rates for the eigenvector errors, measured by $\|\cdot\|_{2\to\infty}$ after polar alignment, are
\begin{equation*}
\begin{aligned}
 &e_{\mathrm{LXS},n}
 =\sqrt{\frac{\log n}{ns_n}}
   +\frac{\sqrt p\log n+p/n}{\sqrt n\,s_n},\\
 &e_{\mathrm{PCA},n}
 =\sqrt{\frac{\log n}{ns_n}}
   +\frac{\sqrt{p\log n}+\log n}{n s_n}.
\end{aligned}
\end{equation*}
The first terms coincide.
The second term in each bound is the contribution quadratic in the noise scale $\sigma$, since $1/s_n=\sigma^2/L_K$.
The first bound holds with probability $1-O(n^{-1})$; Theorem~\ref{thm:rowwise} gives the second with probability $1-O(n^{-B})$ for every fixed $B>0$.
Whenever both sets of assumptions hold,
\begin{equation*}
 e_{\mathrm{PCA},n}\asymp\sqrt{\frac{\log n}{ns_n}},
 \qquad
 \frac{e_{\mathrm{LXS},n}}{e_{\mathrm{PCA},n}}
 \asymp1+\frac{b_{n,p}}{\sqrt{s_n\log n}}.
\end{equation*}
Hence the upper-bound order improves when $s_n\log n=o(b_{n,p}^2)$ within the common domain, but is unchanged when $s_n\log n\gtrsim b_{n,p}^2$.
For example, $p\asymp n$ and $s_n=n$ satisfy both sets of conditions and give orders $(\log n)/n$ and $\sqrt{\log n}/n$, respectively.
For fixed $p$, the two bounds have the same order on their common domain.
These statements compare upper bounds, not the actual estimation error or necessary signal strengths.

For label recovery, put $\Delta^2=\min_{k\ne\ell}\|\bm{\mu}_k-\bm{\mu}_\ell\|_2^2$.
Under the common assumptions, including \eqref{eq:full-signal-rank} and $\lambda_1(\bm{B}_{\bm{\mu}})/L_K=O(1)$, $\Delta^2\asymp L_K$; without a shape restriction, no uniform two-sided comparison is available.
The general sub-Gaussian embedding result of \citet{LittleXieSun2023} gives exact labels after distance-based clustering when $s_n\gtrsim b_{n,p}$.
With a common covariance, convex concentration, and a sufficiently flat retained signal spectrum, their alternative result gives the sufficient condition $s_n\gtrsim n+\sqrt p$.
Isotropic Gaussian noise and an equally weighted regular simplex satisfy the additional assumptions on a count event of probability tending to one.
Both label-recovery results hold with probability $1-O(n^{-1})$.
Our sufficient signal scale must therefore be compared with $\min\{b_{n,p},n+\sqrt p\}$ when the additional assumptions hold.
For $p\to\infty$, both $b_{n,p}/g_{n,p}\to\infty$ and $(n+\sqrt p)/g_{n,p}\to\infty$, whereas for fixed $p$ the general label-recovery scale and $g_{n,p}$ are both of order $\log n$.
Moreover, their eigenvector bound itself can imply exact projection signs: with $\delta=\|\widehat{\bm{V}}\bm{Q}-\bm{V}\|_{2\to\infty}$,
\begin{equation*}
 \|\widehat{\bm{P}}-\bm{P}\|_{\max}
 \le2\|\bm{V}\|_{2\to\infty}\delta+\delta^2.
\end{equation*}
Together with \eqref{eq:oracle-projection-entries}, a sufficiently large constant in $s_n\gtrsim b_{n,p}$ makes this bound smaller than the oracle sign margin.
Consequently, their eigenvector bound also yields exact sign recovery under this sufficiently strong signal condition.
Theorem~\ref{thm:sign-recovery} supplies the weaker sufficient scale $g_{n,p}=o(s_n)$ in growing dimensions and extends beyond the common model to anisotropic noise and unbounded signal condition numbers.

\citet{AbbeFanWang2022} control eigenspaces and eigenvalue-weighted embeddings of a hollowed, uncentered Gram matrix and obtain exact labels through approximate $K$-means.
Their mixture recovery theorem assumes a bounded ratio of the nonzero mean-Gram eigenvalues but allows Hilbert space observations and heteroscedastic noise.
\citet{ZhangZhou2024} analyze $K$-means on a singular-value-weighted embedding of uncentered data under independent sub-Gaussian noise, without a bounded signal condition number.
When the embedding dimension equals the signal rank, their result requires a sufficiently large smallest nonzero signal singular value; adaptive dimension selection removes that requirement under additional separation and tuning conditions.
Both works allow mean configurations that are not affinely independent.
Theorems~\ref{thm:rowwise} and~\ref{thm:sign-recovery} instead control the leading eigenspace of the sample-centered Gram matrix and the entrywise signs of the corresponding orthogonal projection under \eqref{eq:noise-model} and \eqref{eq:subgaussian-assumption}, with affinely independent component means and no additional clustering step.
Exact signs determine the labels, whereas correct labels from an embedding need not imply correct projection signs; the comparisons assert neither uniform statistical superiority nor minimax optimality.

\section{A sharp cutoff under isotropic Gaussian noise}
\label{sec:gaussian-main}
To derive a sharp cutoff for exact sign recovery, we specialize to isotropic Gaussian noise and a fixed mean shape.
The Gaussian assumptions \eqref{eq:gauss-shape} and \eqref{eq:gauss-model} apply only to Sections~\ref{sec:gaussian-main} and~\ref{sec:optimal-thresholds} and Subsection~\ref{sec:gaussian-proofs}; the sufficient condition \eqref{eq:main-condition} is not imposed.
Fix $\bm{m}_1,\ldots,\bm{m}_K\in\mathbb{R}^r$ and $\pi_1,\ldots,\pi_K$ independent of $n$, such that
\begin{equation}
 \sum_{k=1}^K\pi_k\bm{m}_k=\bm0,
 \qquad
 \bm{D}_0:=\sum_{k=1}^K\pi_k\bm{m}_k\bm{m}_k^{\mathsf{T}}\succ\bm0,
 \qquad \lambda_{\min}(\bm{D}_0)=1.
 \label{eq:gauss-shape}
\end{equation}
The normalization fixes the scale of the mean configuration.
For each $n$, let $p=p_n\ge r$, let $\bm{U}_n\in\mathbb{R}^{p\times r}$ satisfy $\bm{U}_n^{\mathsf{T}}\bm{U}_n=\bm{I}_r$, and suppose that
\begin{equation}
 \bm{x}_i=\overline{\bm{\mu}}+\sqrt{L_n}\,\bm{U}_n\bm{m}_{Z_i}+\sigma_n\bm{\xi}_i,
 \qquad i\in\{1,\ldots,n\},
 \qquad \bm{\xi}_i\sim N(\bm0,\bm{I}_p),
 \label{eq:gauss-model}
\end{equation}
where $L_n,\sigma_n>0$, $\bm{\xi}_1,\ldots,\bm{\xi}_n$ are independent and are also independent of $Z_1,\ldots,Z_n$.
The population mean $\overline{\bm{\mu}}$ is arbitrary and disappears upon sample centering.
In particular, $\bm{\Sigma}=\sigma_n^2\bm{I}_p$ and $\bm{B}_{\bm{\mu}}=L_n\bm{U}_n\bm{D}_0\bm{U}_n^{\mathsf{T}}$, so $L_K=L_n$.
The dimension $p_n$, signal scale $L_n$, noise scale $\sigma_n$, and $\bm{U}_n$ may vary with $n$.
We refer to the configuration specified by $\bm{m}_1,\ldots,\bm{m}_K$, up to scaling and orthogonal transformations, as the mean shape.
The fixed mean configuration fixes all ratios between the positive population signal eigenvalues, but does not require equal proportions or means forming a regular simplex.

Define
\begin{equation*}
 \bm{q}_k=\bm{D}_0^{-1/2}\bm{m}_k\quad(k\in\{1,\ldots,K\}),
 \qquad t_n=\frac{L_K}{\sigma_n^2},\qquad \gamma_n=\frac pn.
\end{equation*}
The vectors $\bm{q}_1,\ldots,\bm{q}_K$ satisfy
\begin{align}
 &\sum_k\pi_k\bm{q}_k =\bm0,\label{eq:gauss-q-mean}\\
 &\sum_k\pi_k\bm{q}_k\bm{q}_k^{\mathsf{T}} =\bm{I}_r,\label{eq:gauss-q-covariance}\\
 &\bm{q}_k^{\mathsf{T}}\bm{q}_\ell =\frac{\mathbb I_{\{k=\ell\}}}{\pi_k}-1.\label{eq:gauss-q-geometry}
\end{align}
The vectors $\bm{q}_k$ are used only in the theoretical analysis; PCA connectivity is still estimated by $\widehat{\bm{P}}$ defined in \eqref{eq:sample-connectivity}.
For $t>0$, put
\begin{equation}
 \bm{\Omega}_n(t)=\frac1t\bm{D}_0^{-1}
                  +\frac{\gamma_n}{t^2}\bm{D}_0^{-2},
 \qquad
 s_{k\ell}=\begin{cases}1,&k=\ell,\\-1,&k\ne\ell.\end{cases}
 \label{eq:gauss-Omega}
\end{equation}
For $k,\ell\in\{1,\ldots,K\}$, let
\begin{equation*}
 \mathcal D_{k\ell}
 =\{(\bm{x},\bm{y})\in\mathbb{R}^r\times\mathbb{R}^r:
 s_{k\ell}(\bm{q}_k+\bm{x})^{\mathsf{T}}(\bm{q}_\ell+\bm{y})\le0\}.
\end{equation*}
For every positive definite $r\times r$ matrix $\bm{W}$, define
\begin{equation}
 \mathfrak d(\bm{W})
 =\min_{1\le k\le\ell\le K}\;
  \inf_{(\bm{x},\bm{y})\in\mathcal D_{k\ell}}
  \max\{\bm{x}^{\mathsf{T}}\bm{W}^{-1}\bm{x},\ \bm{y}^{\mathsf{T}}\bm{W}^{-1}\bm{y}\}.
 \label{eq:gauss-d-functional}
\end{equation}
Thus $\mathfrak d(\bm{W})$ is the smallest squared ellipsoidal radius permitting a zero or incorrect sign.
The case $k=\ell$ concerns two distinct observations, not a diagonal entry.
The maximum, rather than the sum, of the two quadratic costs reflects that both perturbations must lie in the two corresponding limiting point sets.

For $n\ge K$, let $t_{\mathrm{crit},n}>0$ denote the solution of
\begin{equation}
 \mathfrak d\!\left(\bm{\Omega}_n(t_{\mathrm{crit},n})\right)=2\log n.
 \label{eq:gauss-critical-equation}
\end{equation}
The existence and uniqueness of this solution are established in Theorem~\ref{thm:gaussian-cutoff}.
Set $L_{\mathrm{crit},n}=\sigma_n^2t_{\mathrm{crit},n}$ and define
\begin{equation*}
 \mathcal S_n
 =\{s_{Z_iZ_j}\widehat P_{ij}>0
                   \text{ for all }1\le i<j\le n\}.
\end{equation*}

\begin{theorem}[Sharp Gaussian cutoff for exact sign recovery]
\label{thm:gaussian-cutoff}
Assume \eqref{eq:gauss-shape} and \eqref{eq:gauss-model}.
For every $n\ge K$, \eqref{eq:gauss-critical-equation} has a unique positive solution.
Fix $\varepsilon\in(0,1)$.
If there is an integer $N$ such that
\begin{equation*}
 L_K\ge(1+\varepsilon)L_{\mathrm{crit},n}
 \qquad\text{for every integer }n\ge N,
\end{equation*}
then $\mathbb{P}(\mathcal S_n)\to1$.
If there is an integer $N$ such that
\begin{equation*}
 L_K\le(1-\varepsilon)L_{\mathrm{crit},n}
 \qquad\text{for every integer }n\ge N,
\end{equation*}
then $\mathbb{P}(\mathcal S_n^{\mathrm{c}})\to1$.
More generally, if
\begin{equation*}
 \liminf_{n\to\infty}
 \frac{\mathfrak d(\bm{\Omega}_n(t_n))}{2\log n}>1,
\end{equation*}
then $\mathbb{P}(\mathcal S_n)\to1$; if
\begin{equation*}
 \limsup_{n\to\infty}
 \frac{\mathfrak d(\bm{\Omega}_n(t_n))}{2\log n}<1,
\end{equation*}
then $\mathbb{P}(\mathcal S_n^{\mathrm{c}})\to1$.
\end{theorem}

\begin{corollary}[Explicit cutoff for an equally weighted regular simplex]
\label{cor:gaussian-simplex}
Suppose that \eqref{eq:gauss-model} holds, $\pi_k=1/K$, and $\bm{m}_k^{\mathsf{T}}\bm{m}_\ell=K\mathbb I_{\{k=\ell\}}-1$.
Then $\bm{D}_0=\bm{I}_r$ and
\begin{equation*}
 L_{\mathrm{crit},n}
 =\frac{C_K\sigma_n^2\log n}{2}
       \left(1+\sqrt{1+\frac{4p}{C_Kn\log n}}\right),
 \qquad
 C_K=\begin{cases}
       2,&K=2,\\
       4\{K-1+\sqrt{K(K-2)}\},&K\ge3.
      \end{cases}
\end{equation*}
Equivalently, the critical signal strength satisfies
\begin{equation}
 \frac{L_{\mathrm{crit},n}^{\,2}}
      {\sigma_n^2L_{\mathrm{crit},n}+\sigma_n^4p/n}=C_K\log n.
 \label{eq:gauss-simplex-effective}
\end{equation}
\end{corollary}

\begin{remark}[Interpretation of the zero-threshold cutoff]
Theorem~\ref{thm:gaussian-cutoff} imposes no explicit restriction on the growth of $p_n$, apart from $p_n\ge r$.
The cutoff concerns the signal strength required to recover all signs of the orthogonal projection onto the leading $r$-dimensional eigenspace of the sample-centered, nonhollowed Gram matrix, with the entrywise threshold fixed at zero.
Theorem~\ref{thm:gaussian-cutoff} does not assert a minimax threshold for clustering, and failure of exact sign recovery need not imply failure of label recovery by other procedures.
The theorem does not cover the regime $L_K/L_{\mathrm{crit},n}\to 1$.
The terms $t^{-1}\bm{D}_0^{-1}$ and $\gamma_n t^{-2}\bm{D}_0^{-2}$  in \eqref{eq:gauss-Omega} stem from the signal--noise cross term and the quadratic noise term, respectively.
The cutoff occurs when extreme embedded observations can first cross an inner-product sign boundary.
The conditional Gaussian representation underlying the extreme-observation interpretation is established in Lemma~\ref{lem:gauss-Haar}.
Lemma~\ref{lem:gauss-extreme-cloud} shows that, within each component, the set of independent standard Gaussian vectors divided by $\sqrt{\log n}$ converges in probability, in Hausdorff distance, to the ball of radius $\sqrt2$; this yields the scale $2\log n$ in \eqref{eq:gauss-critical-equation}.
\end{remark}

\begin{remark}[Order of the regular-simplex cutoff]
\label{rem:simplex-cutoff-order}
For fixed $K$, the regular-simplex cutoff is of order $\sigma_n^2\{\log n+\sqrt{p(\log n)/n}\}$.
The leading term of $L_{\mathrm{crit},n}$ is $C_K\sigma_n^2\log n$ when $p=o(n\log n)$, and $\sigma_n^2\sqrt{C_Kp(\log n)/n}$ when $p/(n\log n)\to\infty$.
For three components, $C_3=8+4\sqrt3$.
\end{remark}

\section{Optimal thresholds for a regular simplex}
\label{sec:optimal-thresholds}
\subsection{Sharp cutoff at the optimal threshold}
Fix an integer $K\ge3$.
Consider the isotropic Gaussian model \eqref{eq:gauss-model} with equal mixing proportions $\pi_k=1/K$ and fixed regular-simplex vertices satisfying $\bm{m}_k^{\mathsf{T}}\bm{m}_\ell=K\mathbb I_{\{k=\ell\}}-1$ for all $k,\ell\in\{1,\ldots,K\}$.
For a fixed common threshold $\tau\in\mathbb{R}$, define
\begin{equation*}
 \mathcal S_n(\tau)
 =\bigl\{\mathbb I_{\{n\widehat P_{ij}>\tau\}}
       =\mathbb I_{\{Z_i=Z_j\}}\text{ for all }i<j\bigr\}.
\end{equation*}
Put $a_K=\sqrt{(K-2)/2}$, $b_K=\sqrt{K/2}$, and $c_K=\sqrt{K-1}$, and, for $a_K^2\le\tau<c_K^2$, define
\begin{equation*}
\begin{aligned}
 &R_{\mathrm w,K}^2(\tau)=(c_K-\sqrt\tau)^2,\\
 &R_{\mathrm b,K}^2(\tau)
 =\min_{0\le y\le b_K}
 \bigl\{(\sqrt{\tau+y^2}-a_K)^2+(b_K-y)^2\bigr\}.
\end{aligned}
\end{equation*}

For $-1<\tau<r$, where $r=K-1$, the thresholded version of \eqref{eq:gauss-d-functional} with $\bm{W}=\bm{I}_r$ is
\begin{equation}
 R_K^2(\tau)
 =\min_{1\le k\le\ell\le K}\;
 \inf_{\substack{\bm{x},\bm{y}\in\mathbb{R}^r\\
 s_{k\ell}\{(\bm{q}_k+\bm{x})^{\mathsf{T}}(\bm{q}_\ell+\bm{y})-\tau\}\le0}}
 \max\{\|\bm{x}\|_2^2,\|\bm{y}\|_2^2\}.
 \label{eq:simplex-threshold-radius}
\end{equation}
For each fixed common threshold $\tau\in(-1,r)$, the cutoff coefficient is $C_K^{\mathrm{thr}}(\tau)=2/R_K^2(\tau)$.

\begin{theorem}[Optimal threshold and sharp simplex cutoff]
\label{thm:simplex-threshold}
There is a unique $\tau_K^\star\in(a_K^2,c_K^2)$ satisfying
\begin{equation*}
 R_{\mathrm w,K}^2(\tau_K^\star)=R_{\mathrm b,K}^2(\tau_K^\star),
 \qquad
 C_{K,\mathrm{opt}}^{\mathrm{thr}}
 :=\frac{2}{(c_K-\sqrt{\tau_K^\star})^2}.
\end{equation*}
The threshold $\tau_K^\star$ uniquely minimizes $C_K^{\mathrm{thr}}(\tau)$ among fixed common scalar thresholds $\tau\in(-1,r)$, and $C_K^{\mathrm{thr}}(\tau_K^\star)=C_{K,\mathrm{opt}}^{\mathrm{thr}}$.
For every fixed common threshold $\tau\notin(-1,r)$, $\mathbb{P}(\mathcal S_n(\tau))\to0$.
Define the effective signal-to-noise ratio appearing in \eqref{eq:gauss-simplex-effective} by
\begin{equation*}
 \iota_n=\frac{L_K^2}{\sigma_n^2L_K+\sigma_n^4p/n}
        =\frac{t_n^2}{t_n+\gamma_n}.
\end{equation*}
Then
\begin{equation*}
\begin{aligned}
 &\mathbb{P}(\mathcal S_n(\tau_K^\star))\to1
 \quad\text{if }\liminf_{n\to\infty}
 \frac{\iota_n}{C_{K,\mathrm{opt}}^{\mathrm{thr}}\log n}>1,\\
 &\mathbb{P}(\mathcal S_n(\tau_K^\star))\to0
 \quad\text{if }\limsup_{n\to\infty}
 \frac{\iota_n}{C_{K,\mathrm{opt}}^{\mathrm{thr}}\log n}<1.
\end{aligned}
\end{equation*}

Put $h_K=b_K/(c_K+a_K)$ and let $w_K\in(0,h_K)$ be the unique solution of
\begin{equation}
 w_K^4+h_Kw_K^3+h_Kw_K-h_K^2=0.
 \label{eq:simplex-threshold-quartic}
\end{equation}
Then
\begin{equation*}
 \tau_K^\star=(K-1)
 \left\{\frac{(1-w_K^2)^2(w_K+h_K)}{2(1+h_K^2)w_K}\right\}^{\!2}.
\end{equation*}
Let $h_\infty=\sqrt2-1$, let $w_\infty\in(0,h_\infty)$ solve \eqref{eq:simplex-threshold-quartic} with $h_K$ replaced by $h_\infty$, and put
\begin{equation*}
 \alpha_\infty=
 \left\{\frac{(1-w_\infty^2)^2(w_\infty+h_\infty)}
 {2(1+h_\infty^2)w_\infty}\right\}^{\!2}.
\end{equation*}
As $K\to\infty$,
\begin{equation*}
 \frac{\tau_K^\star}{K}\to\alpha_\infty,
 \qquad
 K C_{K,\mathrm{opt}}^{\mathrm{thr}}
 \to\frac{2}{(1-\sqrt{\alpha_\infty})^2}
 =28.8966509395\ldots .
\end{equation*}
\end{theorem}

\begin{remark}[Effect of a common positive threshold]
For Theorem~\ref{thm:simplex-threshold}, the critical signal strength is the expression in Corollary~\ref{cor:gaussian-simplex} with $C_K$ replaced by $C_{K,\mathrm{opt}}^{\mathrm{thr}}$.
The two probability conclusions do not cover the regime $\iota_n/(C_{K,\mathrm{opt}}^{\mathrm{thr}}\log n)\to1$.
For comparison, the zero-threshold coefficient in Corollary~\ref{cor:gaussian-simplex} satisfies $C_K\sim8K$.
The limits as $K\to\infty$ in Theorem~\ref{thm:simplex-threshold} describe the constants obtained for each fixed $K$; they do not assert recovery when $K=K_n\to\infty$.
\end{remark}

\subsection{Comparison with other recovery thresholds}
\label{sec:cutoff-comparison}
For $K=2$ with equal mixing proportions, let $\Delta=\|\bm{\mu}_1-\bm{\mu}_2\|_2$.
Then $L_2=\Delta^2/4$, and \eqref{eq:gauss-simplex-effective} becomes $t_{\mathrm{crit},n}^{\,2}/(t_{\mathrm{crit},n}+\gamma_n)=2\log n$.
The identity agrees with the sharp recovery threshold of \citet{Ndaoud2022}.
For the hollowed, uncentered Gram matrix in the model with opposite component means, \citet{AbbeFanWang2022} already prove exact label recovery by the signs of the leading eigenvector of the hollowed matrix when $t_n^2/(t_n+\gamma_n)$ exceeds $(2+\varepsilon)\log n$ for some fixed $\varepsilon>0$, without subsequent iterations that alternate between assigning observations to the nearest estimated center and updating the estimated centers.
Thus $C_2=2$ is not a new cutoff coefficient for exact label recovery; our result concerns the orthogonal projection onto the leading eigenspace of the sample-centered, nonhollowed Gram matrix and extends the analysis of exact sign recovery to multiple components.

For $\Delta^2=\min_{k\ne\ell}\|\bm{\mu}_k-\bm{\mu}_\ell\|_2^2$, \citet{ChenYang2021} identify the exact-label-recovery cutoff
\begin{equation*}
 \Delta_{\mathrm{CY},n}^2
 =4\sigma_n^2\left(1+\sqrt{1+\frac{Kp}{n\log n}}\right)\log n.
\end{equation*}
Their SDP recovery guarantee holds a fixed relative distance above this cutoff for equal deterministic group sizes and $K\le c\log n/\log\log n$.
Their minimax converse holds a fixed relative distance below the cutoff over a parameter space with approximately balanced deterministic partitions when $K\le\log n$.
Neither statement imposes a separate restriction on $p$.
Their general SDP recovery theorem also allows unequal group sizes under explicit size and separation conditions.
Here $K$ is fixed and $n_k/n\xrightarrow{p}1/K$ under equal mixing proportions.
On such asymptotically balanced configurations, the size conditions hold for suitable fixed auxiliary constants, and the remainder in the separation bound is negligible relative to the leading term, uniformly in $p$.
Applying the general theorem conditionally on the labels and then averaging gives exact label recovery with probability tending to one whenever $\Delta^2\ge(1+\varepsilon)\Delta_{\mathrm{CY},n}^2$ eventually for some fixed $\varepsilon>0$.
Theorem~\ref{thm:simplex-threshold} concerns all pairwise decisions based on $n\widehat P_{ij}-\tau_K^\star$ and optimizes only over fixed common scalar thresholds; correct labels from another procedure need not make these signs correct.

In the regular-simplex model, $\Delta^2=2KL_K$.
Writing $C=C_{K,\mathrm{opt}}^{\mathrm{thr}}$, the cutoff in the same squared-distance scale is
\begin{equation*}
 \Delta_{\mathrm{PCA},n}^2
 =KC\sigma_n^2\left(1+\sqrt{1+\frac{4p}{Cn\log n}}\right)\log n.
\end{equation*}
For fixed $K$ and $p=o(n\log n)$,
\begin{equation*}
\begin{aligned}
 &\Delta_{\mathrm{CY},n}^2\sim8\sigma_n^2\log n,
 \qquad \Delta_{\mathrm{PCA},n}^2\sim2KC\sigma_n^2\log n,\\
 &\frac{\Delta_{\mathrm{PCA},n}^2}{\Delta_{\mathrm{CY},n}^2}
 \to\frac{KC}{4}.
\end{aligned}
\end{equation*}
Taking $K\to\infty$ only after the fixed-$K$ limit gives
\begin{equation*}
 C_{K,\mathrm{opt}}^{\mathrm{thr}}\sim\frac{28.8966509395\ldots}{K},
 \qquad
 \lim_{K\to\infty}\lim_{n\to\infty}
 \frac{\Delta_{\mathrm{PCA},n}^2}{\Delta_{\mathrm{CY},n}^2}
 =\frac{28.8966509395\ldots}{4}\approx7.224.
\end{equation*}
This comparison does not assert a cutoff for growing $K$.
The minimax class of \citet{ChenYang2021} does not require the component means to form a regular simplex; its lower bound therefore does not determine the minimax cutoff for the fixed regular-simplex model.

Finally, \citet{DingHe2004} threshold the normalized connectivity $\widehat P_{ij}/\sqrt{\widehat P_{ii}\widehat P_{jj}}$ at a positive constant.
The threshold $\tau_K^\star/n$ applies instead to the unnormalized projection, so Theorem~\ref{thm:simplex-threshold} does not optimize their normalization-dependent rule.

\section{Proofs}

\subsection{Proofs for the sub-Gaussian model}
\label{sec:proofs}
Throughout Subsection~\ref{sec:proofs}, concentration bounds involving the noise are understood conditionally on the labels, except for the unconditional inequalities in Lemma~\ref{lem:subg-tools}.
Combining the conditional noise bounds with the unconditional concentration bounds for the component counts yields the stated probability bounds under the joint law.
Define
\begin{equation*}
 \widehat{\bm{A}}=\bm{Y}^{\mathsf{T}}\bm{Y}-\mathrm{tr}(\bm{\Sigma})\bm{J}_n,
 \qquad
 \bm{M}_{\bm{\Sigma}}=\bm{G}^{\mathsf{T}}\bm{G}-\mathrm{tr}(\bm{\Sigma})\bm{I}_n.
\end{equation*}
On $\bm{e}^\perp$, subtracting $\mathrm{tr}(\bm{\Sigma})\bm{J}_n$ from $\bm{Y}^{\mathsf{T}}\bm{Y}$ subtracts $\mathrm{tr}(\bm{\Sigma})$ from every eigenvalue and preserves eigenspaces.
The matrices $\widehat{\bm{A}}$ and $\bm{M}_{\bm{\Sigma}}$ are introduced only for the proofs.
Expanding $\bm{Y}^{\mathsf{T}}\bm{Y}$ using \eqref{eq:signal-data-Gram} gives
\begin{equation}
 \begin{split}
 \widehat{\bm{A}}&=\bm{A}+\bm{E},\quad
 \bm{E}=\bm{Y}_{\bm{\mu}}^{\mathsf{T}}\bm{G}\bm{J}_n
       +\bm{J}_n\bm{G}^{\mathsf{T}}\bm{Y}_{\bm{\mu}}+\bm{J}_n\bm{M}_{\bm{\Sigma}}\bm{J}_n.
 \end{split}
 \label{eq:gram-decomposition}
\end{equation}
In particular, $\mathbb{E}[\bm{Y}^{\mathsf{T}}\bm{Y}\mid Z_1,\ldots,Z_n]=\bm{A}+\mathrm{tr}(\bm{\Sigma})\bm{J}_n$.
Since the columns of $\bm{G}\bm{J}_n$ are dependent, we apply independent-coordinate concentration bounds to $\bm{G}$, not $\bm{G}\bm{J}_n$.

\subsubsection{Population signal geometry}
\begin{lemma}[Concentration of the component counts]
\label{lem:count-concentration}
For every fixed $B>0$, there is $C_B>0$ such that
\begin{equation*}
        \mathcal E_B^{\mathrm{cnt}}
        :=\left\{
        \max_{1\le k\le K}|\widehat\pi_k-\pi_k|
        \le C_B\sqrt{\frac{\log n}{n}}
        \right\}
\end{equation*}
satisfies $\mathbb{P}\bigl((\mathcal E_B^{\mathrm{cnt}})^{\mathrm{c}}\bigr)\le Cn^{-B}$.
There is a deterministic integer $N_B$ such that, for every $n\ge N_B$, on $\mathcal E_B^{\mathrm{cnt}}$,
\begin{equation*}
        \frac{\pi_{\min}}2\le\widehat\pi_k\le1,
        \qquad k\in\{1,\ldots,K\}.
\end{equation*}
\end{lemma}

\begin{proof}
For each $k$, $n_k\sim\operatorname{Bin}(n,\pi_k)$, so the Bernstein inequality gives
\begin{equation*}
 \mathbb{P}(|\widehat\pi_k-\pi_k|>t)
 \le2\exp\left\{-\frac{nt^2}{2\pi_k(1-\pi_k)+2t/3}\right\}.
\end{equation*}
Take $t=C_B\sqrt{(\log n)/n}$, with $C_B$ chosen large enough that a union bound over the fixed $K$ components gives the probability bound in Lemma~\ref{lem:count-concentration}.
Since $C_B\sqrt{\log n/n}\to0$, choose a deterministic integer $N_B$ such that
\begin{equation*}
 C_B\sqrt{\log n/n}\le\pi_{\min}/2
 \qquad\text{for every integer }n\ge N_B.
\end{equation*}
For every integer $n\ge N_B$, on $\mathcal E_B^{\mathrm{cnt}}$, we have $\widehat\pi_k\ge\pi_k-\pi_{\min}/2\ge\pi_{\min}/2$ for every $k\in\{1,\ldots,K\}$, and $\widehat\pi_k\le1$ by definition.
\end{proof}

To verify the equivalence of \eqref{eq:full-signal-rank} and affine independence, let $\mathbf{1}_K^{\mathsf{T}}\bm{b}=0$.
Then $\sum_k b_k\widetilde{\bm{\mu}}_k=\bm{M}_{\bm{\mu}}\bm{b}$, which vanishes only for $\bm{b}=\bm0$ by \eqref{eq:mean-kernel} and $\mathbf{1}_K^{\mathsf{T}}\bm{\pi}=1$.
Conversely, suppose that $\bm{a}\in\mathrm{Ker}(\bm{M}_{\bm{\mu}})$ is not proportional to $\bm{\pi}$.
Define
\begin{equation*}
\bm{b}=\bm{a}-(\mathbf{1}_K^{\mathsf{T}}\bm{a})\bm{\pi}.
\end{equation*}
Since $\mathbf{1}_K^{\mathsf{T}}\bm{\pi}=1$ and $\bm{M}_{\bm{\mu}}\bm{\pi}=\bm0$, we have
\begin{equation*}
\mathbf{1}_K^{\mathsf{T}}\bm{b}=0,
\qquad
\bm{M}_{\bm{\mu}}\bm{b}=\bm0.
\end{equation*}
Moreover, $\bm{b}\neq\bm0$; otherwise, $\bm{a}=(\mathbf{1}_K^{\mathsf{T}}\bm{a})\bm{\pi}$, contradicting the assumption that $\bm{a}$ is not proportional to $\bm{\pi}$.
Thus $\bm{b}$ gives a nontrivial affine relation among the component means.

\begin{lemma}[The nonzero signal eigenvalues]
\label{lem:signal-eigenvalues}
On $\mathcal E_B^{\mathrm{cnt}}$, set $\tau_n=\max_{1\le k\le K}|\widehat\pi_k/\pi_k-1|$.
For every $j\in\{1,\ldots,r\}$,
\begin{equation*}
        (1-\tau_n)n\lambda_j(\bm{B}_{\bm{\mu}})
        \le\lambda_j(\bm{A})
        \le(1+\tau_n)n\lambda_j(\bm{B}_{\bm{\mu}}).
\end{equation*}
In particular, $\lambda_r\asymp nL_K$.
\end{lemma}

\begin{proof}
Put
\begin{equation*}
 \widehat{\bm{B}}_{\mathrm c}
 =\frac1n\bm{Y}_{\bm{\mu}}\bm{Y}_{\bm{\mu}}^{\mathsf{T}}
 =\sum_k\widehat\pi_k(\bm{\mu}_k-\overline{\bm{\mu}}_Z)
                         (\bm{\mu}_k-\overline{\bm{\mu}}_Z)^{\mathsf{T}}.
\end{equation*}
For every $\bm{u}\in\mathbb{R}^p$,
\begin{equation*}
 \bm{u}^{\mathsf{T}}\widehat{\bm{B}}_{\mathrm c}\bm{u}
   =\min_{b\in\mathbb{R}}\sum_k\widehat\pi_k(\bm{u}^{\mathsf{T}}\bm{\mu}_k-b)^2,
 \qquad
 \bm{u}^{\mathsf{T}}\bm{B}_{\bm{\mu}}\bm{u}
   =\min_{b\in\mathbb{R}}\sum_k\pi_k(\bm{u}^{\mathsf{T}}\bm{\mu}_k-b)^2.
\end{equation*}
Indeed, for nonnegative weights $w_k$ summing to one, put $a_k=\bm{u}^{\mathsf{T}}\bm{\mu}_k$ and $\bar a_w=\sum_k w_ka_k$.
Then
\begin{equation}
 \sum_k w_k(a_k-b)^2
 =\sum_k w_k(a_k-\bar a_w)^2+(b-\bar a_w)^2.
 \label{eq:weighted-variance-decomposition}
\end{equation}
The minimizing $b$ is therefore $\bar a_w$, equal to $\bm{u}^{\mathsf{T}}\overline{\bm{\mu}}_Z$ for $w_k=\widehat\pi_k$ and to zero for $w_k=\pi_k$.
Equation~\eqref{eq:weighted-variance-decomposition} therefore expresses sample and population centering as scalar minimization over $b$.
For $\tau_n<1$, comparing the weights before taking the minima gives
\begin{equation*}
 (1-\tau_n)\bm{B}_{\bm{\mu}}\preceq\widehat{\bm{B}}_{\mathrm c}
                         \preceq(1+\tau_n)\bm{B}_{\bm{\mu}}.
\end{equation*}
If $\tau_n\ge1$, the lower bound follows from $\widehat{\bm{B}}_{\mathrm c}\succeq\bm0$, and the upper bound follows from the same comparison of minima.
On the event $\mathcal E_B^{\mathrm{cnt}}$, $\tau_n\le C_B\pi_{\min}^{-1}\sqrt{\log n/n}=o(1)$.
The Courant--Fischer variational characterization and the common positive eigenvalues of $\bm{Y}_{\bm{\mu}}^{\mathsf{T}}\bm{Y}_{\bm{\mu}}$ and $\bm{Y}_{\bm{\mu}}\bm{Y}_{\bm{\mu}}^{\mathsf{T}}$ give the bounds on $\lambda_j(\bm{A})$.
Taking $j=r$ yields $\lambda_r\asymp nL_K$.
\end{proof}

For calculations with individual eigenvalues, choose the columns of $\bm{V}$ as orthonormal eigenvectors corresponding to the positive eigenvalues of $\bm{A}$, and write
\begin{equation*}
        \bm{A}\bm{V}=\bm{V}\bm{\Lambda},
        \qquad
        \bm{\Lambda}=\mathrm{diag}(\lambda_1,\ldots,\lambda_r),
        \qquad
        \bm{P}=\bm{V}\bm{V}^{\mathsf{T}},
        \qquad
        \bm{Q}=\bm{J}_n-\bm{P}.
\end{equation*}

\begin{proof}[Proof of Proposition~\ref{prop:population-structure}]
Assume that all groups are nonempty.
Let $\bm{H}\in\mathbb{R}^{n\times K}$ be the matrix whose entries satisfy $H_{ik}=\mathbb I_{\{Z_i=k\}}$, and put
\begin{equation*}
 \bm{N}=\bm{H}^{\mathsf{T}}\bm{H}=\mathrm{diag}(n_1,\ldots,n_K),\qquad
 \bm{W}=\bm{H}\bm{N}^{-1/2},\qquad
 \bm{\Pi}=\bm{W}\bm{W}^{\mathsf{T}}.
\end{equation*}
Then $\bm{Y}_{\bm{\mu}}=\bm{M}_{\bm{\mu}}\bm{H}^{\mathsf{T}}\bm{J}_n$ and $\mathrm{Im}(\bm{Y}_{\bm{\mu}}^{\mathsf{T}})\subseteq\mathrm{Im}(\bm{H})\cap\bm{e}^\perp$.
The map $\bm{H}^{\mathsf{T}}\bm{J}_n$ has range $\mathbf{1}_K^\perp$ and rank $K-1$.
To verify the range assertion, $\mathbf{1}_K^{\mathsf{T}}\bm{H}^{\mathsf{T}}\bm{J}_n=\mathbf{1}_n^{\mathsf{T}}\bm{J}_n=\bm0^{\mathsf{T}}$ gives one inclusion.
Conversely, for $\bm{b}\in\mathbf{1}_K^\perp$, set $\bm{a}=\bm{H}\bm{N}^{-1}\bm{b}$.
Then
\begin{equation*}
 \mathbf{1}_n^{\mathsf{T}}\bm{a}=\mathbf{1}_K^{\mathsf{T}}\bm{b}=0,\qquad
 \bm{H}^{\mathsf{T}}\bm{J}_n\bm{a}=\bm{H}^{\mathsf{T}}\bm{a}=\bm{b}.
\end{equation*}
By \eqref{eq:mean-kernel} and $\mathbf{1}_K^{\mathsf{T}}\bm{\pi}=1$, $\bm{M}_{\bm{\mu}}$ is injective on $\mathbf{1}_K^\perp$: any vector in $\mathrm{Ker}(\bm{M}_{\bm{\mu}})\cap\mathbf{1}_K^\perp$ has the form $a\bm{\pi}$, and $\mathbf{1}_K^{\mathsf{T}}(a\bm{\pi})=a$ forces $a=0$.
Since $\bm{H}$ has rank $K$ and $\bm{e}\in\mathrm{Im}(\bm{H})$, the space $\mathrm{Im}(\bm{H})\cap\bm{e}^\perp$ has dimension $K-1$.
Thus $\bm{A}$ has rank $r$ and
\begin{equation*}
 \mathrm{Im}(\bm{A})=\mathrm{Im}(\bm{H})\cap\bm{e}^\perp,\qquad
 \bm{P}=\bm{\Pi}-\bm{e}\bm{e}^{\mathsf{T}},\qquad
 \bm{Q}=\bm{I}_n-\bm{\Pi}.
\end{equation*}
Indeed, $\bm{W}^{\mathsf{T}}\bm{W}=\bm{I}_K$ and $\bm{\Pi}\bm{e}=\bm{e}$, so $(\bm{\Pi}-\bm{e}\bm{e}^{\mathsf{T}})^2=\bm{\Pi}-\bm{e}\bm{e}^{\mathsf{T}}$.
The symmetric idempotent matrix $\bm{\Pi}-\bm{e}\bm{e}^{\mathsf{T}}$ has image $\mathrm{Im}(\bm{H})\cap\bm{e}^\perp$, identifying $\bm{P}$ without choosing individual eigenvectors.
The decomposition $\bm{I}_n=\bm{e}\bm{e}^{\mathsf{T}}+\bm{P}+\bm{Q}$ separates the overall mean, between-group contrasts, and within-group residuals.
Since the columns of $\bm{Y}_{\bm{\mu}}$ are constant within each group and sum to zero, $\bm{Y}_{\bm{\mu}}\bm{P}=\bm{Y}_{\bm{\mu}}$ and $\bm{Y}_{\bm{\mu}}\bm{Q}=\bm0$.

Taking entries of $\bm{P}=\bm{\Pi}-\bm{e}\bm{e}^{\mathsf{T}}$ yields \eqref{eq:oracle-projection-entries}.
For $k\ne\ell$, $P_{ij}=-1/n$.
For $k=\ell$, $P_{ij}=(1-\widehat\pi_k)/(n\widehat\pi_k)>0$ since $K\ge2$ and all groups are nonempty.
On the event $\mathcal E_B^{\mathrm{cnt}}$, $\widehat\pi_k\ge\pi_{\min}/2$ and $1-\widehat\pi_k=\sum_{\ell\ne k}\widehat\pi_\ell\ge(K-1)\pi_{\min}/2$, giving the margin.
The bound for $\|\bm{V}\|_{2\to\infty}$ follows from \eqref{eq:oracle-projection-entries} and $P_{ii}=\|V_{i*}\|_2^2$, and $\lambda_r\asymp nL_K$ follows from Lemma~\ref{lem:signal-eigenvalues}.
Lemma~\ref{lem:count-concentration} gives probability $1-O(n^{-B})$.
\end{proof}

\begin{proof}[Proof of Corollary~\ref{cor:three-angle}]
Here $K=3$ and $r=2$.
Put $\bm{U}_0=(\bm{\mu}_1,\bm{\mu}_2)\in\mathbb{R}^{p\times2}$.
By centering, $\bm{\mu}_3=-(\pi_1/\pi_3)\bm{\mu}_1-(\pi_2/\pi_3)\bm{\mu}_2$, so
\begin{equation*}
        \bm{B}_{\bm{\mu}}=\bm{U}_0\bm{T}_{\pi}\bm{U}_0^{\mathsf{T}},
        \qquad
        \bm{T}_{\pi}=
        \begin{pmatrix}
        \pi_1+\pi_1^2/\pi_3&\pi_1\pi_2/\pi_3\\
        \pi_1\pi_2/\pi_3&\pi_2+\pi_2^2/\pi_3
        \end{pmatrix}.
\end{equation*}
For $\bm{z}=(a,b)^{\mathsf{T}}$,
\begin{equation*}
        \bm{z}^{\mathsf{T}}\bm{T}_{\pi}\bm{z}
        =\pi_1a^2+\pi_2b^2
          +\frac{(\pi_1a+\pi_2b)^2}{\pi_3}.
\end{equation*}
Thus $c_{\pi}\bm{I}_2\preceq\bm{T}_{\pi}\preceq C_{\pi}\bm{I}_2$.
Multiplying by $\bm{U}_0$ and $\bm{U}_0^{\mathsf{T}}$ and applying the Courant--Fischer variational characterization gives
\begin{equation*}
        L_3\asymp\lambda_{\min}(\bm{U}_0^{\mathsf{T}}\bm{U}_0).
\end{equation*}
Put $\bm{T}=\bm{U}_0^{\mathsf{T}}\bm{U}_0$ and $t_k=\|\bm{\mu}_k\|_2^2$ for $k\in\{1,2\}$.
Since
\begin{equation*}
 \det(\bm{T})=t_1t_2\sin^2\theta,
 \qquad
 \max(t_1,t_2)\le\lambda_{\max}(\bm{T})\le t_1+t_2,
\end{equation*}
we have
\begin{equation*}
 \frac{t_1t_2\sin^2\theta}{t_1+t_2}
 \le\lambda_{\min}(\bm{T})
     =\frac{\det(\bm{T})}{\lambda_{\max}(\bm{T})}
 \le\min(t_1,t_2)\sin^2\theta.
\end{equation*}
Hence $L_{\mathrm{ang}}/2\le\lambda_{\min}(\bm{T}) \le L_{\mathrm{ang}}$, so $L_3\asymp L_{\mathrm{ang}}$.
The condition in Corollary~\ref{cor:three-angle} implies \eqref{eq:main-condition}, and Theorems~\ref{thm:rowwise} and~\ref{thm:sign-recovery} apply.
\end{proof}

\begin{proof}[Proof of Proposition~\ref{prop:rank-necessity}]
With $\bm{H}$ and $\bm{N}$ as defined in the proof of Proposition~\ref{prop:population-structure}, set $\bm{S}=\bm{H}\bm{N}^{-1/2}$.
The columns of $\bm{S}$ are orthonormal and the signal space lies in $\mathrm{Im}(\bm{S})$, so
\begin{equation*}
        \bm{P}_{\mathrm{sig}}=\bm{S}\bm{J}\bm{S}^{\mathsf{T}},
        \qquad
        \bm{J}=\bm{S}^{\mathsf{T}}\bm{P}_{\mathrm{sig}}\bm{S},
\end{equation*}
where $\bm{J}$ is a $K\times K$ orthogonal projection satisfying $\mathrm{rank}(\bm{J})=\mathrm{rank}(\bm{M}_{\bm{\mu}})$.
Indeed, $\bm{H}^{\mathsf{T}}\bm{J}_n$ maps onto $\mathbf{1}_K^\perp$, and $\bm{M}_{\bm{\mu}}(\mathbf{1}_K^\perp)=\mathrm{Im}(\bm{M}_{\bm{\mu}})$ since $\bm{M}_{\bm{\mu}}\bm{\pi}=\bm0$ and $\mathbf{1}_K^{\mathsf{T}}\bm{\pi}=1$.
The vector $\bm{a}=\bm{N}^{1/2}\mathbf{1}_K/\sqrt n$ has strictly positive entries and $\bm{S}\bm{a}=\bm{e}$.
Since $\bm{Y}_{\bm{\mu}}\bm{e}=\bm0$, we have $\bm{J}\bm{a}=\bm0$.
For $i\in\mathcal G_k$, $j\in\mathcal G_\ell$,
\begin{equation*}
        (P_{\mathrm{sig}})_{ij}
        =\frac{J_{k\ell}}{\sqrt{n_kn_\ell}}.
\end{equation*}
Thus the assumed strict negativity of $(P_{\mathrm{sig}})_{ij}$ between distinct groups implies $J_{k\ell}<0$ whenever $k\ne\ell$.
Let $\bm{D}_a=\mathrm{diag}(a_1,\ldots,a_K)$ and $\bm{T}_a=\bm{D}_a\bm{J}\bm{D}_a$.
Then $\bm{T}_a\mathbf{1}_K=\bm0$ and all off-diagonal entries of $\bm{T}_a$ are negative.
With $w_{k\ell}=-(T_a)_{k\ell}>0$, symmetry and zero row sums give
\begin{equation*}
        \bm{x}^{\mathsf{T}}\bm{T}_a\bm{x}
        =\sum_{k<\ell}w_{k\ell}(x_k-x_\ell)^2.
\end{equation*}
Indeed, $(T_a)_{kk}=\sum_{\ell\ne k}w_{k\ell}$ by the zero row sums, and expansion gives
\begin{equation*}
 \sum_k x_k^2\sum_{\ell\ne k}w_{k\ell}
       -2\sum_{k<\ell}w_{k\ell}x_kx_\ell
 =\sum_{k<\ell}w_{k\ell}(x_k-x_\ell)^2.
\end{equation*}
Since $w_{k\ell}>0$ and $\bm{D}_a$ is invertible,
\begin{equation*}
\begin{aligned}[b]
 &\mathrm{Ker}(\bm{T}_a)
 =\{\bm{x}:x_k=x_\ell\text{ for all }k<\ell\}
   =\mathrm{span}\{\mathbf{1}_K\},\\
 &\mathrm{rank}(\bm{M}_{\bm{\mu}})=\mathrm{rank}(\bm{J})=\mathrm{rank}(\bm{T}_a)=K-1.\end{aligned}
\qedhere
\end{equation*}
\end{proof}

\subsubsection{Concentration bounds}

\begin{lemma}[Linear and quadratic concentration]
\label{lem:subg-tools}
Assume \eqref{eq:noise-model} and \eqref{eq:subgaussian-assumption}.
For $m\in\{1,\ldots,n\}$, define $\bm{\Xi}_m=(\bm{\xi}_1^{\mathsf{T}},\ldots,\bm{\xi}_m^{\mathsf{T}})^{\mathsf{T}}\in\mathbb{R}^{mp}$.
\begin{enumerate}[label=\textup{(\roman*)},leftmargin=3em]
\item For every deterministic vector $\bm{h}\in\mathbb{R}^{mp}$,
\begin{equation*}
        \|\left\langle \bm{h},\bm{\Xi}_m\right\rangle\|_{\psi_2}
        \le C_{\kappa}\|\bm{h}\|_2,
\end{equation*}
and hence, for every $t>0$,
\begin{equation*}
        \mathbb{P}\{|\left\langle \bm{h},\bm{\Xi}_m\right\rangle|>t\}
        \le2\exp\left(-c_{\kappa}\frac{t^2}{\|\bm{h}\|_2^2}\right),
\end{equation*}
where the right-hand side is interpreted as zero if $\bm{h}=\bm0$.
For $\bm{h}=\bm0$, $\left\langle \bm{h},\bm{\Xi}_m\right\rangle=0$ almost surely, so the probability on the left is also zero and equality holds for every $t>0$.
\item For every symmetric matrix $\bm{B}\in\mathbb{R}^{mp\times mp}$ and every $s\ge1$,
\begin{equation*}
\mathbb{P}\left(
\left|\bm{\Xi}_m^{\mathsf{T}}\bm{B}\bm{\Xi}_m-
\mathbb{E}[\bm{\Xi}_m^{\mathsf{T}}\bm{B}\bm{\Xi}_m]\right|
>C_{\kappa}\{\|\bm{B}\|_F\sqrt s+\|\bm{B}\|_{\mathrm{op}}s\}
\right)
\le2e^{-c_{\kappa}s}.
\end{equation*}
\end{enumerate}
\end{lemma}

\begin{proof}
Part (i) follows from \citet[Proposition~2.6.1 and Theorem~2.6.3]{Vershynin2018}.
Part (ii) follows from the Hanson--Wright inequality \citep{RudelsonVershynin2013} with threshold $C_{\kappa}\{\|\bm{B}\|_F\sqrt s+\|\bm{B}\|_{\mathrm{op}}s\}$; see also \citet[Theorem~6.2.1]{Vershynin2018}.
\end{proof}

\begin{lemma}[Net reconstruction]
\label{lem:net-reconstruction}
Let $S^{d-1}=\{\bm{x}\in\mathbb{R}^d:\|\bm{x}\|_2=1\}$ denote the unit sphere in $\mathbb{R}^d$.
Let $\mathcal N\subset S^{d-1}$ be a finite $\varepsilon$-net in Euclidean norm, with $0<\varepsilon<1/2$.
\begin{enumerate}[label=\textup{(\roman*)},leftmargin=3em]
\item If $\bm{H}\in\mathbb{R}^{d\times d}$ is symmetric, then
\begin{equation*}
        \|\bm{H}\|_{\mathrm{op}}
        \le\frac{1}{1-2\varepsilon}
        \max_{\bm{u}\in\mathcal N}|\bm{u}^{\mathsf{T}}\bm{H}\bm{u}|.
\end{equation*}
\item For every $\bm{B}\in\mathbb{R}^{d_1\times d_2}$, if $\mathcal N_1$ and $\mathcal N_2$ are finite $\varepsilon$-nets of the corresponding unit spheres, then
\begin{equation*}
        \|\bm{B}\|_{\mathrm{op}}
        \le\frac{1}{1-2\varepsilon}
        \max_{\bm{u}\in\mathcal N_1,\,\bm{v}\in\mathcal N_2}|\bm{u}^{\mathsf{T}}\bm{B}\bm{v}|.
\end{equation*}
Moreover, a $1/4$-net of $S^{d-1}$ can be chosen with at most $9^d$ points.
\end{enumerate}
\end{lemma}

\begin{proof}
For (i), choose $\bm{x}\in S^{d-1}$ satisfying $|\bm{x}^{\mathsf{T}}\bm{H}\bm{x}|=\|\bm{H}\|_{\mathrm{op}}$, and let $\bm{x}_0\in\mathcal N$ satisfy $\|\bm{x}-\bm{x}_0\|_2\le\varepsilon$.
Then
\begin{equation*}
\begin{aligned}
 &\|\bm{H}\|_{\mathrm{op}}
\le|\bm{x}_0^{\mathsf{T}}\bm{H}\bm{x}_0|
 +|\,(\bm{x}-\bm{x}_0)^{\mathsf{T}}\bm{H}\bm{x}\,|
 +|\,\bm{x}_0^{\mathsf{T}}\bm{H}(\bm{x}-\bm{x}_0)\,|\\
 &\quad \le \max_{\bm{u}\in\mathcal N}|\bm{u}^{\mathsf{T}}\bm{H}\bm{u}|+2\varepsilon\|\bm{H}\|_{\mathrm{op}}.
\end{aligned}
\end{equation*}
Rearranging proves (i).
For (ii), choose unit vectors $\bm{x},\bm{y}$ attaining $|\bm{x}^{\mathsf{T}}\bm{B}\bm{y}|=\|\bm{B}\|_{\mathrm{op}}$ and approximate $\bm{x}$ and $\bm{y}$ by points of $\mathcal N_1$ and $\mathcal N_2$, respectively.
Each approximation error has absolute value at most $\varepsilon\|\bm{B}\|_{\mathrm{op}}$, proving (ii).
The cardinality bound follows from \citet[Corollary~4.2.13]{Vershynin2018}.
\end{proof}

On the event $\{\min_{1\le k\le K}n_k>0\}$, fix a measurable compact singular value decomposition
\begin{equation*}
        \bm{Y}_{\bm{\mu}}=\bm{U}\bm{D}\bm{V}^{\mathsf{T}},
        \qquad
        \bm{D}=\bm{\Lambda}^{1/2}.
\end{equation*}
Here $\bm{U}\in\mathbb{R}^{p\times r}$ and $\bm{V}\in\mathbb{R}^{n\times r}$ have orthonormal columns and are deterministic given the labels.
All calculations involving diagonal $\bm{D}$ use the bases $\bm{U}$ and $\bm{V}$.

For the following concentration bounds, define
\begin{equation*}
 \mathcal U_{\Sigma}
 =\{\bm{a}\in\mathrm{span}(\bm{U}):\|\bm{\Sigma}^{1/2}\bm{a}\|_2\le1\}.
\end{equation*}

\begin{lemma}[A simultaneous concentration event]
\label{prop:basic-event}
Fix $B>0$ and a realization of $Z_1,\ldots,Z_n$ satisfying $\mathcal E_B^{\mathrm{cnt}}$.
Conditional on the fixed label realization, there is an event $\mathcal E_B^{\mathrm{noise}}$ with probability at least $1-O(n^{-B})$ on which all of the following inequalities hold:
\begin{align}
 &\|\bm{M}_{\bm{\Sigma}}\|_{\mathrm{op}}
\le C_B(\nu\sqrt n+\rho n),\label{eq:E-Mop}\\
 &\|\bm{\Sigma}^{1/2}\bm{G}\|_{\mathrm{op}}
\le C_B(\nu+\rho\sqrt n),\label{eq:E-SGop}\\
 &\|\bm{M}_{\bm{\Sigma}}\bm{W}\|_{2\to\infty}
\le C_Ba_n,\label{eq:E-MV-row}\\
 &\|\bm{V}^{\mathsf{T}}\bm{M}_{\bm{\Sigma}}\bm{V}\|_{\mathrm{op}}
\le C_Ba_n,\label{eq:E-VMV}\\
 &\|\bm{U}^{\mathsf{T}}\bm{G}\bm{V}\|_{\mathrm{op}}
\le C_B\sqrt{\rho\log n},\label{eq:E-UGV}\\
 &\sup_{\bm{a}\in\mathcal U_{\Sigma}}
\|\bm{G}^{\mathsf{T}}\bm{a}\|_2
\le C_B\sqrt n.\label{eq:E-GU-metric}
\end{align}
\end{lemma}

\begin{proof}
Given the labels, $\bm{U}$, $\bm{V}$, and $\bm{W}$ are deterministic.
Since $\bm{G}$ is independent of the labels, the noise bounds for fixed coefficient matrices also hold conditionally on the labels.
We prove \eqref{eq:E-Mop}--\eqref{eq:E-GU-metric} and intersect the six events.

To prove \eqref{eq:E-Mop}, let $\bm{X}=(\bm{\xi}_1,\ldots,\bm{\xi}_n)$, so $\bm{G}=\bm{\Sigma}^{1/2}\bm{X}$.
For fixed $\bm{u}\in S^{n-1}$, with $\bm{\Xi}=(\bm{\xi}_1^{\mathsf{T}},\ldots,\bm{\xi}_n^{\mathsf{T}})^{\mathsf{T}}$,
\begin{equation*}
\bm{u}^{\mathsf{T}}\bm{M}_{\bm{\Sigma}}\bm{u}
=\bm{\Xi}^{\mathsf{T}}(\bm{u}\bm{u}^{\mathsf{T}}\otimes\bm{\Sigma})\bm{\Xi}
 -\mathrm{tr}(\bm{u}\bm{u}^{\mathsf{T}}\otimes\bm{\Sigma}).
\end{equation*}
The Kronecker product identities give
\begin{equation*}
\|\bm{u}\bm{u}^{\mathsf{T}}\otimes\bm{\Sigma}\|_F=\nu,
\qquad
\|\bm{u}\bm{u}^{\mathsf{T}}\otimes\bm{\Sigma}\|_{\mathrm{op}}=\rho.
\end{equation*}
Lemma~\ref{lem:subg-tools}(ii) therefore implies
\begin{equation}
\mathbb{P}\left(
|\bm{u}^{\mathsf{T}}\bm{M}_{\bm{\Sigma}}\bm{u}|>
C_{\kappa}(\nu\sqrt t+\rho t)
\,\middle|\,Z_1,\ldots,Z_n
\right)
\le2e^{-c_{\kappa}t}.
\label{eq:gram-direction-tail}
\end{equation}
Apply \eqref{eq:gram-direction-tail} with $t=A_Bn$ on a $1/4$-net of at most $9^n$ points.
A union bound and Lemma~\ref{lem:net-reconstruction}(i) give failure probability $O(n^{-B})$ for sufficiently large $A_B$.
The union bound over the $1/4$-net is at most $2\exp\{n\log9-c_\kappa A_Bn\}$.
Choosing $c_\kappa A_B>\log9$ makes the failure probability exponentially small, hence $O(n^{-B})$ for any fixed $B$.
In contrast, the $nK$ entrywise bounds for \eqref{eq:E-MV-row} use $s=A_B\log n$ and have total failure probability at most $2K n^{1-c_\kappa A_B}$, which is $O(n^{-B})$ once $c_\kappa A_B\ge B+1$.
The net and entrywise union bounds require different tail scales.

For \eqref{eq:E-SGop}, use the argument for \eqref{eq:E-Mop} with $\bm{\Sigma}$ replaced by $\bm{\Sigma}^2$.
Since $\|\bm{\Sigma}^2\|_F\le\rho\nu$ and $\|\bm{\Sigma}^2\|_{\mathrm{op}}=\rho^2$, with conditional probability $1-O(n^{-B})$,
\begin{equation*}
 \|\bm{X}^{\mathsf{T}}\bm{\Sigma}^2\bm{X}-\nu^2\bm{I}_n\|_{\mathrm{op}}
 \le C_B(\rho\nu\sqrt n+\rho^2n).
\end{equation*}
Consequently,
\begin{equation*}
 \|\bm{\Sigma}^{1/2}\bm{G}\|_{\mathrm{op}}^2
 =\|\bm{X}^{\mathsf{T}}\bm{\Sigma}^2\bm{X}\|_{\mathrm{op}}
 \le\nu^2+C_B(\rho\nu\sqrt n+\rho^2n)
 \le C_B'(\nu+\rho\sqrt n)^2,
\end{equation*}
which gives \eqref{eq:E-SGop}.

For \eqref{eq:E-MV-row}, let $\bm{e}_i$ be the $i$th coordinate vector in $\mathbb{R}^n$.
For deterministic $\bm{v}\in\mathbb{R}^n$ with $\|\bm{v}\|_2\le1$, put
\begin{equation*}
 \bm{H}_{i,\bm{v}}=\frac{\bm{e}_i\bm{v}^{\mathsf{T}}+\bm{v}\bm{e}_i^{\mathsf{T}}}2.
\end{equation*}
Then
\begin{equation*}
 (\bm{M}_{\bm{\Sigma}}\bm{v})_i
 =\bm{\Xi}^{\mathsf{T}}(\bm{H}_{i,\bm{v}}\otimes\bm{\Sigma})\bm{\Xi}
   -\mathrm{tr}(\bm{H}_{i,\bm{v}}\otimes\bm{\Sigma}),
\end{equation*}
and
\begin{equation*}
 \|\bm{H}_{i,\bm{v}}\|_{\mathrm{op}}\le1,
 \qquad
 \|\bm{H}_{i,\bm{v}}\|_F^2=\frac{\|\bm{v}\|_2^2+v_i^2}{2}\le1.
\end{equation*}
Apply Lemma~\ref{lem:subg-tools}(ii) with $s=A_B\log n$ to the columns of $\bm{W}$, and take a union bound over $nK$ entries, using fixed $K$.

For \eqref{eq:E-VMV}, fix $\bm{b}\in S^{r-1}$ and set $\bm{v}=\bm{V}\bm{b}$.
Then
\begin{equation*}
 \bm{b}^{\mathsf{T}}\bm{V}^{\mathsf{T}}\bm{M}_{\bm{\Sigma}}\bm{V}\bm{b}
 =\bm{\Xi}^{\mathsf{T}}(\bm{v}\bm{v}^{\mathsf{T}}\otimes\bm{\Sigma})\bm{\Xi}
   -\mathrm{tr}(\bm{v}\bm{v}^{\mathsf{T}}\otimes\bm{\Sigma}).
\end{equation*}
The coefficient matrix has Frobenius norm $\nu$ and operator norm $\rho$.
Apply Lemma~\ref{lem:subg-tools}(ii) with $s=A_B\log n$ on a $1/4$-net of $S^{r-1}$ with at most $9^r$ points.
A union bound and Lemma~\ref{lem:net-reconstruction}(i) give \eqref{eq:E-VMV}.

For \eqref{eq:E-UGV}, take $\bm{a},\bm{b}\in S^{r-1}$ and let $\bm{q}=\bm{\Sigma}^{1/2}\bm{U}\bm{a}$ and $\bm{s}=\bm{V}\bm{b}$.
Then
\begin{equation*}
 \bm{a}^{\mathsf{T}}\bm{U}^{\mathsf{T}}\bm{G}\bm{V}\bm{b}
 =\sum_{j=1}^ns_j\left\langle \bm{q},\bm{\xi}_j\right\rangle
 =\left\langle \bm{h}_{\bm{a},\bm{b}},\bm{\Xi}\right\rangle,
\end{equation*}
where $\bm{h}_{\bm{a},\bm{b}}=(s_1\bm{q}^{\mathsf{T}},\ldots,s_n\bm{q}^{\mathsf{T}})^{\mathsf{T}}$.
Since $\|\bm{s}\|_2=1$,
\begin{equation*}
        \|\bm{h}_{\bm{a},\bm{b}}\|_2^2
        =\|\bm{q}\|_2^2
        =\bm{a}^{\mathsf{T}}\bm{U}^{\mathsf{T}}\bm{\Sigma}\bm{U}\bm{a}
        \le\rho.
\end{equation*}
Lemma~\ref{lem:subg-tools}(i), $1/4$-nets of $S^{r-1}$, and Lemma~\ref{lem:net-reconstruction}(ii) give $\|\bm{U}^{\mathsf{T}}\bm{G}\bm{V}\|_{\mathrm{op}}\le C_B\sqrt{\rho\log n}$ by a union bound.

For \eqref{eq:E-GU-metric}, equip $\mathrm{span}(\bm{U})$ with $\left\langle \bm{a},\bm{b}\right\rangle_{\bm{\Sigma}}=\bm{a}^{\mathsf{T}}\bm{\Sigma}\bm{b}$ and quotient out $\mathrm{span}(\bm{U})\cap\mathrm{Ker}(\bm{\Sigma}^{1/2})$, which $\bm{G}^{\mathsf{T}}$ annihilates.
The quotient has dimension $d_{\bm{\Sigma}}\le r$.
Choose a $\bm{\Sigma}$-orthonormal basis $\bm{u}_1,\ldots,\bm{u}_{d_{\bm{\Sigma}}}$ and set $\bm{q}_\ell=\bm{\Sigma}^{1/2}\bm{u}_\ell$, so $\|\bm{q}_\ell\|_2=1$.
The representatives $\bm{u}_1,\ldots,\bm{u}_{d_{\bm{\Sigma}}}$ can be constructed from the positive eigenpairs $(\tau_\ell,\bm{b}_\ell)$ of $\bm{U}^{\mathsf{T}}\bm{\Sigma}\bm{U}$, with orthonormal $\bm{b}_\ell$: set $\bm{u}_\ell=\bm{U}\bm{b}_\ell/\sqrt{\tau_\ell}$.
Then
\begin{equation*}
 \bm{u}_\ell^{\mathsf{T}}\bm{\Sigma}\bm{u}_j=\mathbb I_{\{\ell=j\}},\qquad
 \bm{a}=\sum_\ell\alpha_\ell\bm{u}_\ell+\bm{a}_0,\qquad
 \bm{\Sigma}^{1/2}\bm{a}_0=\bm0.
\end{equation*}
The kernel component $\bm{a}_0$ satisfies $\bm{G}^{\mathsf{T}}\bm{a}_0=\bm{X}^{\mathsf{T}}\bm{\Sigma}^{1/2}\bm{a}_0=\bm0$.
Thus removing the kernel loses no information relevant to $\bm{G}^{\mathsf{T}}\bm{a}$, and avoids dividing by zero eigenvalues.
If $d_{\bm{\Sigma}}=0$, the supremum in \eqref{eq:E-GU-metric} is zero.
For $d_{\bm{\Sigma}}>0$ and $\ell\in\{1,\ldots,d_{\bm{\Sigma}}\}$,
\begin{equation*}
 \|\bm{G}^{\mathsf{T}}\bm{u}_\ell\|_2^2
 =\sum_{i=1}^n\left\langle \bm{q}_\ell,\bm{\xi}_i\right\rangle^2
 =\bm{\Xi}^{\mathsf{T}}(\bm{I}_n\otimes\bm{q}_\ell\bm{q}_\ell^{\mathsf{T}})\bm{\Xi}.
\end{equation*}
The expectation is $n$, and the coefficient matrix has Frobenius and operator norms $\sqrt n$ and $1$.
Lemma~\ref{lem:subg-tools}(ii), with $t=A_B\log n$, gives $\|\bm{G}^{\mathsf{T}}\bm{u}_\ell\|_2\le C_B\sqrt n$ for each of the at most $r$ basis vectors.
For $\bm{a}=\sum_\ell\alpha_\ell\bm{u}_\ell$ modulo the kernel, $\sum_\ell\alpha_\ell^2=\|\bm{\Sigma}^{1/2}\bm{a}\|_2^2$.
The Cauchy--Schwarz inequality yields
\begin{equation*}
 \|\bm{G}^{\mathsf{T}}\bm{a}\|_2
 \le \left(\sum_\ell\alpha_\ell^2\right)^{1/2}
      \left(\sum_\ell\|\bm{G}^{\mathsf{T}}\bm{u}_\ell\|_2^2\right)^{1/2}
 \le C_B\sqrt n\,\|\bm{\Sigma}^{1/2}\bm{a}\|_2.
\end{equation*}
The Cauchy--Schwarz inequality proves \eqref{eq:E-GU-metric}.
Taking the intersection of the six events constructed for \eqref{eq:E-Mop}--\eqref{eq:E-GU-metric} proves the lemma.
\end{proof}

For $i\in\{1,\ldots,n\}$, let
\begin{equation*}
        \mathcal F_i
        =\sigma\bigl(Z_1,\ldots,Z_n,\{\bm{g}_j:j\ne i\}\bigr),
\end{equation*}
where $\sigma(\cdot)$ denotes the sigma-algebra generated by the indicated random variables.
Let $\bm{G}_{-i}$ denote $\bm{G}$ with the $i$th column removed, and put
\begin{equation*}
        \bm{M}^{(i)}
        =\bm{G}_{-i}^{\mathsf{T}}\bm{G}_{-i}-\mathrm{tr}(\bm{\Sigma})\bm{I}_{n-1}.
\end{equation*}

\begin{lemma}[Leave-one-out measurable concentration events]
\label{lem:loo-basic-event}
Fix $B>0$ and condition on a label realization in $\mathcal E_B^{\mathrm{cnt}}$.
There are events $\mathcal E_{i,B}^{\mathrm{loo}}\in\mathcal F_i$ such that
\begin{equation*}
\mathbb{P}\left(\bigcap_{i=1}^n\mathcal E_{i,B}^{\mathrm{loo}}
\,\middle|\,Z_1,\ldots,Z_n\right)
\ge1-O(n^{-B}),
\end{equation*}
and, on $\mathcal E_{i,B}^{\mathrm{loo}}$,
\begin{align}
 &\|\bm{M}^{(i)}\|_{\mathrm{op}}
\le C_B(\nu\sqrt n+\rho n),\label{eq:loo-Mop}\\
 &\|\bm{\Sigma}^{1/2}\bm{G}_{-i}\|_{\mathrm{op}}
\le C_B(\nu+\rho\sqrt n),\label{eq:loo-SGop}\\
 &\sup_{\bm{a}\in\mathcal U_{\Sigma}}
\|\bm{G}_{-i}^{\mathsf{T}}\bm{a}\|_2
\le C_B\sqrt n.\label{eq:loo-GU-metric}
\end{align}
\end{lemma}

\begin{proof}
Let $\bm{J}^{(-i)}\in\mathbb{R}^{n\times(n-1)}$ be the coordinate inclusion matrix that omits coordinate $i$.
Then
\begin{equation}
 \bm{G}_{-i}=\bm{G}\bm{J}^{(-i)},
 \qquad
 \bm{M}^{(i)}=(\bm{J}^{(-i)})^{\mathsf{T}}\bm{M}_{\bm{\Sigma}}\bm{J}^{(-i)}.
 \label{eq:coordinate-deletion}
\end{equation}
By \eqref{eq:coordinate-deletion} and $\|\bm{J}^{(-i)}\|_{\mathrm{op}}=1$,
\begin{equation*}
 \|\bm{M}^{(i)}\|_{\mathrm{op}}\le\|\bm{M}_{\bm{\Sigma}}\|_{\mathrm{op}},
 \qquad
 \|\bm{\Sigma}^{1/2}\bm{G}_{-i}\|_{\mathrm{op}}
 \le\|\bm{\Sigma}^{1/2}\bm{G}\|_{\mathrm{op}},
\end{equation*}
and
\begin{equation*}
 \|\bm{G}_{-i}^{\mathsf{T}}\bm{a}\|_2
 \le \|\bm{G}^{\mathsf{T}}\bm{a}\|_2
 \qquad\text{for every }\bm{a}.
\end{equation*}
For each $i$, let $\mathcal E_{i,B}^{\mathrm{loo}}$ be the event on which \eqref{eq:loo-Mop}, \eqref{eq:loo-SGop}, and \eqref{eq:loo-GU-metric} hold.
Since $\mathcal E_{i,B}^{\mathrm{loo}}$ depends only on $\bm{G}_{-i}$ and the labels, $\mathcal E_{i,B}^{\mathrm{loo}}$ is $\mathcal F_i$-measurable.
Since coordinate deletion does not increase the relevant norms, \eqref{eq:E-Mop}, \eqref{eq:E-SGop}, and \eqref{eq:E-GU-metric} imply \eqref{eq:loo-Mop}, \eqref{eq:loo-SGop}, and \eqref{eq:loo-GU-metric}, respectively, with the same constants.
Hence
\begin{equation*}
 \mathcal E_B^{\mathrm{noise}}
 \subseteq
 \bigcap_{i=1}^n\mathcal E_{i,B}^{\mathrm{loo}}.
\end{equation*}
Therefore, by Lemma~\ref{prop:basic-event},
\begin{equation*}
\mathbb{P}\left(
\bigcap_{i=1}^n\mathcal E_{i,B}^{\mathrm{loo}}
\,\middle|\,Z_1,\ldots,Z_n
\right)
\ge
\mathbb{P}\left(
\mathcal E_B^{\mathrm{noise}}
\,\middle|\,Z_1,\ldots,Z_n
\right)
\ge 1-O(n^{-B}).
\qedhere
\end{equation*}
\end{proof}

\begin{lemma}[Leave-one-out norms of cross products]
\label{lem:loo-ri-norm}
For every fixed $B>0$, the bounds \eqref{eq:Sigmagi} and \eqref{eq:ri-norm} hold simultaneously with probability $1-O(n^{-B})$:
\begin{align}
 &\max_{i\in\{1,\ldots,n\}}\|\bm{\Sigma}^{1/2}\bm{g}_i\|_2
    \le C_B(\nu+\rho\sqrt{\log n}),\label{eq:Sigmagi}\\
 &\max_{i\in\{1,\ldots,n\}}\|\bm{G}_{-i}^{\mathsf{T}}\bm{g}_i\|_2
    \le C_B\sqrt n(\nu+\rho\sqrt{\log n}).\label{eq:ri-norm}
\end{align}
\end{lemma}

\begin{proof}
Since $\|\bm{\Sigma}^{1/2}\bm{g}_i\|_2^2 =\bm{\xi}_i^{\mathsf{T}}\bm{\Sigma}^2\bm{\xi}_i$, Lemma~\ref{lem:subg-tools}(ii) and
\begin{equation*}
        \mathbb{E}[\bm{\xi}_i^{\mathsf{T}}\bm{\Sigma}^2\bm{\xi}_i]=\nu^2,
        \qquad
        \|\bm{\Sigma}^2\|_F\le\rho\nu,
        \qquad
        \|\bm{\Sigma}^2\|_{\mathrm{op}}=\rho^2
\end{equation*}
imply
\begin{equation*}
\|\bm{\Sigma}^{1/2}\bm{g}_i\|_2
\le C_{\kappa}(\nu+\rho\sqrt t)
\end{equation*}
except with probability $2e^{-c_{\kappa}t}$.
Take $t=A_B\log n$ and a union bound to obtain \eqref{eq:Sigmagi}.

To prove \eqref{eq:ri-norm}, condition on $\bm{g}_i$.
If $\|\bm{\Sigma}^{1/2}\bm{g}_i\|_2=0$, then $\bm{G}_{-i}^{\mathsf{T}}\bm{g}_i$ vanishes almost surely.
Otherwise, for $j\ne i$, set
\begin{equation*}
        \omega_j
        =\frac{\bm{g}_j^{\mathsf{T}}\bm{g}_i}{\|\bm{\Sigma}^{1/2}\bm{g}_i\|_2}.
\end{equation*}
The normalized coefficient vector satisfies
\begin{equation*}
\begin{aligned}
 &\bm{q}_i=\frac{\bm{\Sigma}^{1/2}\bm{g}_i}{\|\bm{\Sigma}^{1/2}\bm{g}_i\|_2},
 \qquad \|\bm{q}_i\|_2=1,\\
 &\omega_j=\langle\bm{q}_i,\bm{\xi}_j\rangle,
 \qquad j\ne i.
\end{aligned}
\end{equation*}
Given $\bm{g}_i$, the random variables $(\omega_j)_{j\in\{1,\ldots,n\}\setminus\{i\}}$ are independent, centered, and uniformly sub-Gaussian.
Since $\mathbb{E}[\bm{\xi}_j\bm{\xi}_j^{\mathsf{T}}]=\bm{I}_p$,
\begin{equation*}
        \mathbb{E}[\omega_j^2\mid\bm{g}_i]=1.
\end{equation*}
The Hanson--Wright inequality, applied conditionally to $(\omega_j)_{j\ne i}$ and the identity, with $t=A_B\log n$, gives
\begin{equation}
\mathbb{P}\left(
\sum_{j\ne i}\omega_j^2>
(n-1)+C_B\{\sqrt{n\log n}+\log n\}
\,\middle|\,\bm{g}_i
\right)
\le Cn^{-B-2}.
\label{eq:conditional-HW}
\end{equation}
For the constant $C_B$ in \eqref{eq:conditional-HW}, there exist a deterministic integer $N_B$ and a constant $C_B^{\prime}$ such that $(n-1)+C_B\{\sqrt{n\log n}+\log n\}\le C_B^{\prime}n$ for every integer $n\ge N_B$.
Absorbing the fixed factor $C_B^{\prime}$ into the subsequent constants, we obtain, for every integer $n\ge N_B$, with the conditional failure bound in \eqref{eq:conditional-HW},
\begin{equation*}
 \|\bm{G}_{-i}^{\mathsf{T}}\bm{g}_i\|_2
 \le C_B\sqrt n\,\|\bm{\Sigma}^{1/2}\bm{g}_i\|_2.
\end{equation*}
Integrating \eqref{eq:conditional-HW} over the distribution of $\bm{g}_i$, using \eqref{eq:Sigmagi}, and taking a union bound over $i$ yields \eqref{eq:ri-norm}.
\end{proof}

\subsubsection{Signal separation and auxiliary rates}

\begin{lemma}[Deterministic rate consequences]
\label{lem:rate-consequences}
Assume \eqref{eq:main-condition}.
Define
\begin{align*}
 &\eta_n
 :=\sqrt{\frac{\rho}{L_K}}+\frac{\nu}{\sqrt nL_K}+\frac{\rho}{L_K},\\
 &b_n
 :=\frac{(\nu+\rho\sqrt n)\sqrt{\log n}}{nL_K},\\
 &c_n
 :=\frac{(\nu+\rho\sqrt{\log n})^2}{n^{3/2}L_K^2}.
\end{align*}
There is a deterministic integer $N$ such that, for every $n\ge N$,
\begin{align}
 &\eta_n \le2\epsilon_n,\label{eq:rates-epsilon}\\
 &\sqrt n\,b_n \le2\epsilon_n,\\
 &\sqrt n\,c_n \le\epsilon_n^2.
\end{align}
Consequently,
\begin{align}
 &\eta_n =o(1),\\
 &b_n\eta_n+c_n \le\frac{5\epsilon_n^2}{\sqrt n}=o(n^{-1/2}).\label{eq:rates-small}
\end{align}
Moreover,
\begin{align}
 &\frac{\rho}{L_K} =o(1),\label{eq:direct-consequences}\\
 &\frac{\rho\log n}{L_K} =o(1),\\
 &\frac{\nu}{\sqrt nL_K} =o(1),\label{eq:rate-nu-small}\\
 &\frac{a_n}{\sqrt nL_K} =o(1).\label{eq:rate-an-small}
\end{align}
If $\bm{\Sigma}=\bm0$, the quantities $\eta_n,b_n,c_n$ and the four ratios in \eqref{eq:direct-consequences}--\eqref{eq:rate-an-small} vanish.
\end{lemma}

\begin{proof}
Using the definition of $\epsilon_n$ in \eqref{eq:epsilon-main}, put
\begin{equation*}
 x_n=\sqrt{\rho\log n/L_K},\qquad
 y_n=a_n/(\sqrt nL_K),\qquad \epsilon_n=x_n+y_n.
\end{equation*}
Choose a deterministic integer $N\ge3$ such that $\epsilon_n\le1$ for every integer $n\ge N$.
Then $\log n\ge1$ for every integer $n\ge N$.
Since $a_n=(\nu+\rho\sqrt{\log n})\sqrt{\log n}$,
\begin{equation*}
\begin{aligned}
 &\eta_n\le x_n+y_n+x_n^2\le2\epsilon_n,\\
 &\sqrt n\,b_n\le y_n+x_n^2\le2\epsilon_n,\\
 &\sqrt n\,c_n
 =\frac{(\nu+\rho\sqrt{\log n})^2}{nL_K^2}
   =\frac{y_n^2}{\log n}\le\epsilon_n^2.
\end{aligned}
\end{equation*}
The bounds for $\eta_n,b_n,c_n$ imply \eqref{eq:rates-epsilon}--\eqref{eq:rates-small}.
The limits in \eqref{eq:direct-consequences}--\eqref{eq:rate-an-small} follow from $\rho\log n/L_K=x_n^2\to0$ and $a_n/(\sqrt nL_K)=y_n\to0$, since $\log n\ge1$ and $a_n\ge\nu\sqrt{\log n}$.
\end{proof}

\paragraph*{Block estimates and the sample eigengap}
Since $\mathrm{Ker}(\bm{A})\cap\bm{e}^\perp=\mathrm{Im}(\bm{Q})$, $\bm{Y}_{\bm{\mu}}\bm{Q}=\bm0$ and $\bm{Q}\bm{Y}_{\bm{\mu}}^{\mathsf{T}}=\bm0$.
Since $\bm{V}^{\mathsf{T}}\bm{e}=\bm0$ and $\bm{Q}=\bm{J}_n-\bm{V}\bm{V}^{\mathsf{T}}$, \eqref{eq:centering-matrix} gives $\bm{Q}\bm{J}_n=\bm{J}_n\bm{Q}=\bm{Q}$ and $\bm{J}_n\bm{V}=\bm{V}$.
Substituting these identities into \eqref{eq:gram-decomposition} yields
\begin{equation}
 \bm{Q}\bm{E}\bm{Q}=\bm{Q}\bm{M}_{\bm{\Sigma}}\bm{Q}.
 \label{eq:QEQ}
\end{equation}

\begin{lemma}[Scaled signal block]
\label{lem:signal-block-scaled}
Assume \eqref{eq:main-condition}.
Let $\bm{F}=\bm{V}^{\mathsf{T}}\bm{E}\bm{V}$.
For every fixed $B>0$, with probability $1-O(n^{-B})$,
\begin{equation*}
 \|\bm{D}^{-1}\bm{F}\bm{D}^{-1}\|_{\mathrm{op}}
 \le C_B\frac{\epsilon_n}{\sqrt n}=o(1).
\end{equation*}
\end{lemma}

\begin{proof}
Using $\bm{Y}_{\bm{\mu}}\bm{V}=\bm{U}\bm{D}$ and $\bm{J}_n\bm{V}=\bm{V}$,
\begin{equation}
\bm{D}^{-1}\bm{F}\bm{D}^{-1}
=\bm{U}^{\mathsf{T}}\bm{G}\bm{V}\bm{D}^{-1}
 +\bm{D}^{-1}\bm{V}^{\mathsf{T}}\bm{G}^{\mathsf{T}}\bm{U}
 +\bm{D}^{-1}\bm{V}^{\mathsf{T}}\bm{M}_{\bm{\Sigma}}\bm{V}\bm{D}^{-1}.
 \label{eq:scaled-signal-block-expansion}
\end{equation}
Since $\|\bm{D}^{-1}\|_{\mathrm{op}}=\lambda_r^{-1/2}$, \eqref{eq:scaled-signal-block-expansion}, Lemma~\ref{lem:signal-eigenvalues}, \eqref{eq:E-VMV}, and \eqref{eq:E-UGV} give
\begin{equation*}
\begin{aligned}[b]
 &\|\bm{D}^{-1}\bm{F}\bm{D}^{-1}\|_{\mathrm{op}}
 \le \frac{2\|\bm{U}^{\mathsf{T}}\bm{G}\bm{V}\|_{\mathrm{op}}}{\sqrt{\lambda_r}}
       +\frac{\|\bm{V}^{\mathsf{T}}\bm{M}_{\bm{\Sigma}}\bm{V}\|_{\mathrm{op}}}{\lambda_r}\\
 &\quad \le C_B\left\{\sqrt{\frac{\rho\log n}{nL_K}}
                  +\frac{a_n}{nL_K}\right\}
  =C_B\frac{\epsilon_n}{\sqrt n}.\end{aligned}
\qedhere
\end{equation*}
\end{proof}

\begin{lemma}[Sample eigengap]
\label{lem:eig-lower}
Assume \eqref{eq:main-condition}.
For $n\ge K+1$, let $\widehat\lambda_1\ge\cdots\ge\widehat\lambda_{n-1}$ be the eigenvalues of $\widehat{\bm{A}}$ restricted to $\bm{e}^\perp$.
For every fixed $B>0$, there are fixed constants $c,c'>0$ and a deterministic sequence $\varepsilon_{n,B}\to0$ such that, with probability $1-O(n^{-B})$,
\begin{align}
 &\|\bm{Q}\bm{M}_{\bm{\Sigma}}\bm{Q}\|_{\mathrm{op}} \le \varepsilon_{n,B} nL_K,\label{eq:compressed-noise-bound}\\
 &\widehat\lambda_r \ge cnL_K,\label{eq:signal-eigenvalue-lower}\\
 &\widehat\lambda_{r+1} \le \varepsilon_{n,B} nL_K,\\
 &\widehat\lambda_r-\widehat\lambda_{r+1} \ge c'nL_K.
\end{align}
\end{lemma}

\begin{proof}
Fix $B>0$ and intersect the event $\mathcal E_B^{\mathrm{cnt}}$ with the event in Lemma~\ref{prop:basic-event}.
Then
\begin{equation*}
\frac{\|\bm{Q}\bm{M}_{\bm{\Sigma}}\bm{Q}\|_{\mathrm{op}}}{nL_K}
\le C_B\left(\frac{\nu}{\sqrt nL_K}+\frac{\rho}{L_K}\right)
=: \varepsilon_{n,B}.
\end{equation*}
The intersection of the count and noise events has probability $1-O(n^{-B})$, and $\varepsilon_{n,B}\to0$ by \eqref{eq:direct-consequences} and \eqref{eq:rate-nu-small}.
For $\bm{x}\in\mathbb{R}^r$,
\begin{equation*}
\bm{x}^{\mathsf{T}}\bm{V}^{\mathsf{T}}\widehat{\bm{A}}\bm{V}\bm{x}
=\bm{x}^{\mathsf{T}}(\bm{\Lambda}+\bm{F})\bm{x}
=(\bm{D}\bm{x})^{\mathsf{T}}
\{\bm{I}_r+\bm{D}^{-1}\bm{F}\bm{D}^{-1}\}(\bm{D}\bm{x}).
\end{equation*}
By Lemma~\ref{lem:signal-block-scaled}, choose a deterministic integer $N_B$ such that $C_B\epsilon_n/\sqrt n\le1/2$ for every integer $n\ge N_B$.
On $\mathcal E_B^{\mathrm{cnt}}\cap\mathcal E_B^{\mathrm{noise}}$,
\begin{equation*}
\bm{x}^{\mathsf{T}}\bm{V}^{\mathsf{T}}\widehat{\bm{A}}\bm{V}\bm{x}
\ge\frac{\lambda_r}{2}\|\bm{x}\|_2^2.
\end{equation*}
In the Courant--Fischer variational characterization of $\widehat\lambda_r$ on $\bm{e}^\perp$, choose $\mathrm{Im}(\bm{V})$ as the $r$-dimensional comparison subspace.
The variational characterization gives
\begin{equation*}
        \widehat\lambda_r
        \ge\inf_{\bm{u}\in\mathrm{Im}(\bm{V}),\,\|\bm{u}\|_2=1}
        \bm{u}^{\mathsf{T}}\widehat{\bm{A}}\bm{u}
        \ge\frac12\lambda_r\ge cnL_K.
\end{equation*}
Since $\mathrm{Im}(\bm{Q})$ has codimension $r$ in $\bm{e}^\perp$, the Courant--Fischer variational characterization and \eqref{eq:QEQ} also give the $(r+1)$th eigenvalue bound
\begin{equation*}
\begin{aligned}
 &\widehat\lambda_{r+1}
\le\sup_{\bm{u}\in\mathrm{Im}(\bm{Q}),\,\|\bm{u}\|_2=1}
 \bm{u}^{\mathsf{T}}\widehat{\bm{A}}\bm{u}\\
 &\quad =\sup_{\bm{u}\in\mathrm{Im}(\bm{Q}),\,\|\bm{u}\|_2=1}
 \bm{u}^{\mathsf{T}}\bm{Q}\bm{M}_{\bm{\Sigma}}\bm{Q}\bm{u}\\
 &\quad \le\|\bm{Q}\bm{M}_{\bm{\Sigma}}\bm{Q}\|_{\mathrm{op}}
\le \varepsilon_{n,B}nL_K.
\end{aligned}
\end{equation*}
Increase $N_B$, if necessary, so that $\varepsilon_{n,B}\le c/2$ for every integer $n\ge N_B$.
Then $\widehat\lambda_r-\widehat\lambda_{r+1}\ge(c/2)nL_K$.
Take $c'=c/2$.
The eigenvalues $\widehat\lambda_k+\mathrm{tr}(\bm{\Sigma})$, $k\in\{1,\ldots,r\}$, of $\bm{Y}^{\mathsf{T}}\bm{Y}$ are positive, so the associated eigenvectors lie in $\bm{e}^\perp$ and span the leading eigenspace of $\bm{Y}^{\mathsf{T}}\bm{Y}$ on $\mathbb{R}^n$.
\end{proof}

\subsubsection{Uniform rowwise resolvent estimates}
\label{sec:rowwise-proof}
For related leave-one-out methods, see \citet{AbbeFanWangZhong2020,AbbeFanWang2022} and \citet{ZhangZhou2024}; for methods based on the Neumann series, see \citet{EldridgeBelkinWang2018}.
For brevity, write
\begin{equation*}
        \bm{M}=\bm{M}_{\bm{\Sigma}},
        \qquad
        \bm{Q}=\bm{J}_n-\bm{V}\bm{V}^{\mathsf{T}}=\bm{I}_n-\bm{\Pi}.
\end{equation*}
By Lemma~\ref{lem:eig-lower}, choose a fixed constant $c_0>0$, independent of $n$, sufficiently small that, for every fixed $B>0$, $\widehat\lambda_r\ge s_0$ with probability $1-O(n^{-B})$, where
\begin{equation*}
        s_0=c_0nL_K.
\end{equation*}
The following parameter sets depend only on the labels:
\begin{align*}
 &\mathcal A
=\left\{
\bm{a}\in\mathrm{span}(\bm{U}):
\|\bm{\Sigma}^{1/2}\bm{a}\|_2
\le C_0\sqrt{\frac{\rho}{nL_K}}
\right\},\\
 &\mathcal V
=\left\{
\bm{v}\in\mathrm{span}(\bm{W}):
\|\bm{v}\|_2\le\frac{C_0}{nL_K},
\quad
\|\bm{v}\|_\infty\le\frac{C_0}{n^{3/2}L_K}
\right\},
\end{align*}
where $C_0$ is a sufficiently large constant independent of $n$, chosen once and kept unchanged throughout Section~\ref{sec:rowwise-proof}.
For the joint parameter range, write
\begin{equation*}
 \mathcal T=[s_0,\infty)\times\mathcal A\times\mathcal V.
\end{equation*}
Since $\bm{G}^{\mathsf{T}}$ annihilates $\mathrm{span}(\bm{U})\cap\mathrm{Ker}(\bm{\Sigma}^{1/2})$, parameterize $\mathcal A$ on the quotient, of dimension at most $r$.

\begin{lemma}[Uniform unprojected rowwise resolvent]
\label{lem:unprojected-resolvent}
Under \eqref{eq:main-condition}, for every fixed $B>0$, with probability $1-O(n^{-B})$,
\begin{equation*}
\sup_{(s,\bm{a},\bm{v})\in\mathcal T}
\left\|
\left(\bm{I}_n-\frac1s\bm{M}\right)^{-1}
\{\bm{G}^{\mathsf{T}}\bm{a}+\bm{M}\bm{v}\}
\right\|_\infty
\le C_B\frac{\epsilon_n}{\sqrt n}.
\end{equation*}
The right-hand side $C_B\epsilon_n/\sqrt n$ is $o(n^{-1/2})$.
\end{lemma}

\begin{proof}
The claim is immediate if $\bm{\Sigma}=\bm0$, since $\bm{G}=\bm{M}=\bm0$.
For $\bm{\Sigma}\ne\bm0$, fix $B>0$ and a realization of the labels satisfying $\mathcal E_B^{\mathrm{cnt}}$.
Deterministic bounds use Lemma~\ref{prop:basic-event} and Lemmas~\ref{lem:loo-basic-event} and~\ref{lem:loo-ri-norm}.
Conditional tail bounds use only $\mathcal F_i$-measurable leave-one-out events.
First integrate each conditional failure probability with respect to the distribution of the variables generating $\mathcal F_i$, with the labels fixed.
Then intersect the leave-one-out success events with the global concentration events.

By \eqref{eq:E-Mop} and Lemma~\ref{lem:rate-consequences}, there is a deterministic integer $N_B$ such that the following bounds hold for every $n\ge N_B$ on the concentration event:
\begin{align}
 &\sup_{s\ge s_0}\|\bm{M}/s\|_{\mathrm{op}} \le\frac14,\label{eq:M-resolvent-op}\\
 &\sup_{s\ge s_0}\left\|\left(\bm{I}_n-\frac{\bm{M}}{s}\right)^{-1}\right\|_{\mathrm{op}} \le2.\notag
\end{align}
For $s\ge s_0$, $\bm{a}\in\mathcal A$, and $\bm{v}\in\mathcal V$, put
\begin{equation}
 \bm{w}=\bm{G}^{\mathsf{T}}\bm{a}+\bm{M}\bm{v},
 \qquad
 \bm{z}=\left(\bm{I}_n-\frac{\bm{M}}{s}\right)^{-1}\bm{w}.
 \label{eq:resolvent-input}
\end{equation}
Equation~\eqref{eq:resolvent-input} gives
\begin{equation}
 (s\bm{I}_n-\bm{M})\bm{z}=s\bm{w}.
 \label{eq:resolvent-equation}
\end{equation}

To bound $\bm{w}=\bm{G}^{\mathsf{T}}\bm{a}+\bm{M}\bm{v}$, work on the quotient of $\mathrm{span}(\bm{U})$ by $\mathrm{span}(\bm{U})\cap\mathrm{Ker}(\bm{\Sigma}^{1/2})$.
Choose $\bm{\Sigma}$-orthonormal representatives $\bm{u}_1,\ldots,\bm{u}_{d_{\bm{\Sigma}}}$ with $d_{\bm{\Sigma}}\le r$.
Conditional on the labels, Lemma~\ref{lem:subg-tools}(i) and a union bound over at most $nr$ choices give
\begin{equation*}
 \max_{i\le n}\sum_{\ell=1}^{d_{\bm{\Sigma}}}
             |\bm{g}_i^{\mathsf{T}}\bm{u}_\ell|^2\le C_B\log n
\end{equation*}
with probability $1-O(n^{-B})$.
The sum is zero if $d_{\bm{\Sigma}}=0$.
The radius of $\mathcal A$ and the Cauchy--Schwarz inequality give
\begin{equation*}
 \sup_{\bm{a}\in\mathcal A}\|\bm{G}^{\mathsf{T}}\bm{a}\|_\infty
 \le C_B\sqrt{\frac{\rho\log n}{nL_K}}.
\end{equation*}
For $\bm{v}=\bm{W}\bm{\beta}$, \eqref{eq:E-MV-row} gives $\|\bm{M}\bm{v}\|_\infty\le C_Ba_n/(nL_K)$ since $\|\bm{\beta}\|_2\le C_0/(nL_K)$.
Hence
\begin{equation}
 \sup_{\bm{a}\in\mathcal A,\,\bm{v}\in\mathcal V}
 \|\bm{w}\|_\infty\le C_B\frac{\epsilon_n}{\sqrt n}.
 \label{eq:w-infty}
\end{equation}
Similarly, \eqref{eq:E-Mop} and \eqref{eq:E-GU-metric} yield
\begin{equation}
 \sup_{\bm{a}\in\mathcal A,\,\bm{v}\in\mathcal V}\|\bm{w}\|_2
 \le C_B\left\{\sqrt{\frac{\rho}{L_K}}
       +\frac{\nu}{\sqrt nL_K}+\frac{\rho}{L_K}\right\}
 =C_B\eta_n.
 \label{eq:w-two}
\end{equation}

For the coordinatewise Schur complement, write
\begin{equation*}
 d_i=\|\bm{g}_i\|_2^2-\mathrm{tr}(\bm{\Sigma}),\qquad
 \bm{r}_i=\bm{G}_{-i}^{\mathsf{T}}\bm{g}_i,\qquad
 \bm{R}_i(s)=(s\bm{I}_{n-1}-\bm{M}^{(i)})^{-1}.
\end{equation*}
Since $\bm{M}^{(i)}$ is a principal submatrix of $\bm{M}$,
\begin{equation}
 \|\bm{R}_i(s)\|_{\mathrm{op}}\le\frac2s
 \qquad(i\le n,\ s\ge s_0).
 \label{eq:Ri-op}
\end{equation}
Separating coordinate $i$ in \eqref{eq:resolvent-equation} gives the two equations
\begin{align}
 &(s-d_i)z_i-\bm{r}_i^{\mathsf{T}}\bm{z}_{-i} =sw_i,\label{eq:schur-coordinate-equations}\\
 &-\bm{r}_i z_i+(s\bm{I}_{n-1}-\bm{M}^{(i)})\bm{z}_{-i} =s\bm{w}_{-i}.\label{eq:schur-complement-coordinate}
\end{align}
Substituting $\bm{z}_{-i}=\bm{R}_i(s)(s\bm{w}_{-i}+\bm{r}_i z_i)$ from \eqref{eq:schur-complement-coordinate} into \eqref{eq:schur-coordinate-equations} gives
\begin{equation*}
 z_i=\frac{s\{w_i+\bm{r}_i^{\mathsf{T}}\bm{R}_i(s)\bm{w}_{-i}\}}
             {s-d_i-\bm{r}_i^{\mathsf{T}}\bm{R}_i(s)\bm{r}_i}.
\end{equation*}
Moreover, \eqref{eq:M-resolvent-op} gives $s\bm{I}_n-\bm{M}\succeq(3s/4)\bm{I}_n$.
The block inversion formula therefore yields
\begin{equation}
 s-d_i-\bm{r}_i^{\mathsf{T}}\bm{R}_i(s)\bm{r}_i
 =\{\bm{e}_i^{\mathsf{T}}(s\bm{I}_n-\bm{M})^{-1}\bm{e}_i\}^{-1}
 \ge\frac{3s}{4}\ge\frac{s}{2}.
 \label{eq:denominator-lower}
\end{equation}
For the denominator in \eqref{eq:denominator-lower}, the inverse-order inequality gives
\begin{equation*}
 \bm0\prec(s\bm{I}_n-\bm{M})^{-1}\preceq\frac4{3s}\bm{I}_n,
 \qquad
 0<\bm{e}_i^{\mathsf{T}}(s\bm{I}_n-\bm{M})^{-1}\bm{e}_i\le\frac4{3s}.
\end{equation*}
Taking the reciprocal of the scalar diagonal entry proves \eqref{eq:denominator-lower}.

Separate the dependence on $\bm{g}_i$ by writing
\begin{align}
 &\bm{w}_{-i} =\bm{w}_{-i}^{(i)}+v_i\bm{r}_i,\label{eq:loo-input-decomposition}\\
 &\bm{w}_{-i}^{(i)} =\bm{G}_{-i}^{\mathsf{T}}\bm{a}+\bm{M}^{(i)}\bm{v}_{-i}.\label{eq:loo-input-independent}
\end{align}
For fixed parameters, $\bm{w}_{-i}^{(i)}$ in \eqref{eq:loo-input-independent} is $\mathcal F_i$-measurable.
Combining \eqref{eq:schur-coordinate-equations} and \eqref{eq:schur-complement-coordinate} with \eqref{eq:loo-input-decomposition} gives
\begin{equation}
 z_i=\frac{s\{w_i+\bm{r}_i^{\mathsf{T}}\bm{R}_i(s)\bm{w}_{-i}^{(i)}
                       +v_i\bm{r}_i^{\mathsf{T}}\bm{R}_i(s)\bm{r}_i\}}
             {s-d_i-\bm{r}_i^{\mathsf{T}}\bm{R}_i(s)\bm{r}_i}.
 \label{eq:schur-decomposed}
\end{equation}
By the $\ell_\infty$ restriction in $\mathcal V$, \eqref{eq:ri-norm}, and \eqref{eq:Ri-op},
\begin{equation}
 \sup_{i,s,\bm{v}}|v_i\bm{r}_i^{\mathsf{T}}\bm{R}_i(s)\bm{r}_i|
 \le\frac{C}{n^{3/2}L_K}
 \frac{C_Bn(\nu+\rho\sqrt{\log n})^2}{s_0}
 \le C_Bc_n.
 \label{eq:loo-correction}
\end{equation}

To bound $\bm{r}_i^{\mathsf{T}}\bm{R}_i(s)\bm{w}_{-i}^{(i)}$, fix $i$, condition on $\mathcal F_i$, and work on $\mathcal E_{i,B}^{\mathrm{loo}}$.
Set $\bm{H}_i=\bm{M}^{(i)}/s_0$.
By \eqref{eq:loo-Mop} and Lemma~\ref{lem:rate-consequences}, increase $N_B$ so that $\|\bm{H}_i\|_{\mathrm{op}}\le1/4$ on $\mathcal E_{i,B}^{\mathrm{loo}}$ for every integer $n\ge N_B$.
Define
\begin{align}
 &\bm{T}_a =C_0\sqrt{\frac{\rho}{nL_K}}(\bm{u}_1,\ldots,\bm{u}_{d_{\bm{\Sigma}}}),\label{eq:loo-parameter-maps}\\
 &\bm{T}_v =\frac{C_0}{nL_K}\bm{W}.\label{eq:loo-parameter-map-v}
\end{align}
The matrices $\bm{T}_a$ and $\bm{T}_v$ in \eqref{eq:loo-parameter-maps} and \eqref{eq:loo-parameter-map-v} depend only on the labels.
The normalizations are chosen so that
\begin{align}
 &\|\bm{\Sigma}^{1/2}\bm{T}_a\bm{\alpha}\|_2 =C_0\sqrt{\frac{\rho}{nL_K}}\|\bm{\alpha}\|_2,\label{eq:loo-parameter-norms}\\
 &\|\bm{T}_v\bm{\beta}\|_2 =\frac{C_0}{nL_K}\|\bm{\beta}\|_2.\label{eq:loo-parameter-norm-v}
\end{align}
The identities in \eqref{eq:loo-parameter-norms} and \eqref{eq:loo-parameter-norm-v} follow from $\bm{u}_\ell^{\mathsf{T}}\bm{\Sigma}\bm{u}_m=\mathbb I_{\{\ell=m\}}$ and $\bm{W}^{\mathsf{T}}\bm{W}=\bm{I}_K$, respectively.
Concatenating $\bm{\alpha}$ and $\bm{\beta}$ gives $\bm{u}=(\bm{\alpha}^{\mathsf{T}},\bm{\beta}^{\mathsf{T}})^{\mathsf{T}}$ with $\|\bm{u}\|_2\le\sqrt2$.
The coordinate bound on $\bm{v}$ is retained in \eqref{eq:loo-correction}; enlarging the coefficient set only for $\bm{r}_i^{\mathsf{T}}\bm{R}_i(s)\bm{w}_{-i}^{(i)}$ preserves the upper bound.
Every $\bm{a}\in\mathcal A$ has a representative $\bm{T}_a\bm{\alpha}$ with $\|\bm{\alpha}\|_2\le1$, and every $\bm{v}\in\mathcal V$ equals $\bm{T}_v\bm{\beta}$ with $\|\bm{\beta}\|_2\le1$.
The kernel part of $\bm{a}$ is annihilated by $\bm{G}_{-i}^{\mathsf{T}}$.
Dropping the coordinate constraint on $\mathcal V$ enlarges the input set to
\begin{equation}
 \bm{w}_{-i}^{(i)}=\bm{W}_i\bm{u},\qquad
 \bm{W}_i=\bigl(\bm{G}_{-i}^{\mathsf{T}}\bm{T}_a,\,
                    \bm{M}^{(i)}(\bm{T}_v)_{-i}\bigr),\qquad
 \|\bm{u}\|_2\le\sqrt2.
 \label{eq:loo-input-matrix}
\end{equation}
The two blocks of $\bm{W}_i$ satisfy, on the leave-one-out event,
\begin{align}
 &\|\bm{G}_{-i}^{\mathsf{T}}\bm{T}_a\|_{\mathrm{op}} \le C_BC_0\sqrt n\sqrt{\frac{\rho}{nL_K}}=C_BC_0\sqrt{\frac{\rho}{L_K}},\label{eq:loo-input-block-bounds}\\
 &\|\bm{M}^{(i)}(\bm{T}_v)_{-i}\|_{\mathrm{op}} \le C_B(\nu\sqrt n+\rho n)\frac{C_0}{nL_K}=C_BC_0\left(\frac{\nu}{\sqrt nL_K}+\frac{\rho}{L_K}\right).\label{eq:loo-input-block-bound-M}
\end{align}
In \eqref{eq:loo-input-block-bounds} and \eqref{eq:loo-input-block-bound-M}, the $\bm{G}_{-i}^{\mathsf{T}}\bm{T}_a$ bound uses \eqref{eq:loo-GU-metric}; the $\bm{M}^{(i)}(\bm{T}_v)_{-i}$ bound uses \eqref{eq:loo-Mop} and the fact that row deletion cannot increase the operator norm.
The triangle inequality for the two block operators gives $\|\bm{W}_i\|_{\mathrm{op}}\le C_BC_0\eta_n$.
Absorb the fixed factor $C_0$ into $C_B$ for the remaining estimates.
Here $(\bm{T}_v)_{-i}$ is the matrix obtained by deleting the $i$th row of $\bm{T}_v$.
The $\mathcal F_i$-measurable matrix $\bm{W}_i$ has $d=d_{\bm{\Sigma}}+K\le r+K$ columns and $\|\bm{W}_i\|_{\mathrm{op}}\le C_B\eta_n$ by \eqref{eq:loo-Mop} and \eqref{eq:loo-GU-metric}.
Thus, for $j\ge0$,
\begin{equation}
 \|\bm{\Sigma}^{1/2}\bm{G}_{-i}\bm{H}_i^j\bm{W}_i\|_{\mathrm{op}}
 \le C_B(\nu+\rho\sqrt n)\eta_n4^{-j}.
 \label{eq:loo-coefficient-bound}
\end{equation}

For the conditional concentration step, observe that
\begin{equation*}
 \bm{g}_i^{\mathsf{T}}\bm{G}_{-i}\bm{H}_i^j\bm{W}_i
 =\bm{\xi}_i^{\mathsf{T}}\bigl(\bm{\Sigma}^{1/2}\bm{G}_{-i}\bm{H}_i^j\bm{W}_i\bigr).
\end{equation*}
The independence of $\bm{g}_i$ from $\mathcal F_i$ transfers the unconditional linear-form bound in Lemma~\ref{lem:subg-tools}(i) to each $\mathcal F_i$-measurable column $\bm{h}$ of $\bm{G}_{-i}\bm{H}_i^j\bm{W}_i$: for $t>0$,
\begin{equation}
 \mathbb{P}\bigl(|\bm{g}_i^{\mathsf{T}}\bm{h}|>t\mid\mathcal F_i\bigr)
 \le 2\exp\!\left(-\frac{c_\kappa t^2}
 {\|\bm{\Sigma}^{1/2}\bm{h}\|_2^2}\right).
 \label{eq:loo-conditional-linear}
\end{equation}
When $\bm{\Sigma}^{1/2}\bm{h}=\bm0$, the linear form $\bm{g}_i^{\mathsf{T}}\bm{h}$ vanishes almost surely.
The left-hand side of \eqref{eq:loo-conditional-linear} is therefore zero, and the right-hand side is defined to be zero.
No assertion about the conditional law of the full vector $\bm{\xi}_i$ is needed.
By the Cauchy--Schwarz inequality,
\begin{equation}
 \sup_{\|\bm{u}\|_2\le\sqrt2}
 |\bm{g}_i^{\mathsf{T}}\bm{G}_{-i}\bm{H}_i^j\bm{W}_i\bm{u}|
 \le\sqrt2\,\|\bm{g}_i^{\mathsf{T}}\bm{G}_{-i}\bm{H}_i^j\bm{W}_i\|_2.
 \label{eq:loo-coefficient-uniform}
\end{equation}
Thus \eqref{eq:loo-coefficient-uniform} controls every coefficient vector on one event, without a union bound over an uncountable set.
More precisely, multiplying the bound in \eqref{eq:loo-coefficient-bound} by $A_B\sqrt{\log n+j+1}$ gives a union bound over the columns of at most
\begin{equation*}
 2(r+K)\exp\{-c_\kappa A_B^2(\log n+j+1)\}
 \le n^{-B-2}2^{-j-1}.
\end{equation*}
One fixed $A_B$ suffices for all $n\ge2$ and $j\ge0$: choose $c_\kappa A_B^2\ge\max\{B+2,\log(4(r+K))\}$.
The factor $2^{-j-1}$ makes the sum over all $j$ finite; the union bound does not require independence between the events for different $j$.
Conditional on $\mathcal F_i$, apply \eqref{eq:loo-conditional-linear} to the at most $r+K$ columns with thresholds proportional to $\sqrt{\log n+j+1}$.
For each fixed $j$, the conditional failure probability is at most $n^{-B-2}2^{-j-1}$.
A union bound over $j\ge0$ gives, with conditional probability at least $1-n^{-B-2}$,
\begin{equation}
 \|\bm{g}_i^{\mathsf{T}}\bm{G}_{-i}\bm{H}_i^j\bm{W}_i\|_2
 \le C_B(\nu+\rho\sqrt n)\eta_n4^{-j}
              \sqrt{\log n+j+1}
 \qquad(j\ge0).
 \label{eq:loo-column-bound}
\end{equation}
The Neumann series
\begin{equation}
 \bm{R}_i(s)=\frac1{s_0}\sum_{j=0}^{\infty}
                 (s_0/s)^{j+1}\bm{H}_i^j
 \label{eq:loo-neumann-series}
\end{equation}
converges uniformly in operator norm for $s\ge s_0$.
Indeed, for every integer $m\ge0$,
\begin{align*}
 &\sup_{s\ge s_0}
 \left\|\bm{R}_i(s)-\frac1{s_0}\sum_{j=0}^{m}
       (s_0/s)^{j+1}\bm{H}_i^j\right\|_{\mathrm{op}}
 \le\frac1{s_0}\sum_{j=m+1}^{\infty}4^{-j}
 =\frac{4^{-m}}{3s_0}.
\end{align*}
For $\bm{w}_{-i}^{(i)}=\bm{W}_i\bm{u}$ as in \eqref{eq:loo-input-matrix},
\begin{align*}
 &\bm{r}_i^{\mathsf{T}}\bm{R}_i(s)\bm{w}_{-i}^{(i)}
 =\frac1{s_0}\sum_{j=0}^{\infty}(s_0/s)^{j+1}
   \bm{g}_i^{\mathsf{T}}\bm{G}_{-i}\bm{H}_i^j\bm{W}_i\bm{u}.
\end{align*}
Combining \eqref{eq:loo-input-matrix}, \eqref{eq:loo-coefficient-uniform}, \eqref{eq:loo-column-bound}, and \eqref{eq:loo-neumann-series} gives
\begin{align*}
 &\sup_{s\ge s_0,\,\bm{a}\in\mathcal A,\,\bm{v}\in\mathcal V}
 |\bm{r}_i^{\mathsf{T}}\bm{R}_i(s)\bm{w}_{-i}^{(i)}|
 \le\frac{C_B(\nu+\rho\sqrt n)\eta_n}{s_0}
       \sum_{j=0}^{\infty}4^{-j}\sqrt{\log n+j+1}\\
 &\quad \le C_Bb_n\eta_n.
\end{align*}
For $n\ge3$, $\sqrt{\log n+j+1}\le\sqrt{\log n}\sqrt{j+2}$ and $\sum_{j\ge0}4^{-j}\sqrt{j+2}<\infty$.
Also, $0<s_0/s\le1$ for every $s\ge s_0$.
The bounds on $s_0/s$ and the geometric sum make the estimate of $\bm{r}_i^{\mathsf{T}}\bm{R}_i(s)\bm{w}_{-i}^{(i)}$ uniform over $s\ge s_0$, including data-dependent choices of $s$.

Integrating, taking a union bound over $i$, and applying Lemma~\ref{lem:loo-basic-event} gives
\begin{equation}
 \sup_{i\in\{1,\ldots,n\},\,(s,\bm{a},\bm{v})\in\mathcal T}
 |\bm{r}_i^{\mathsf{T}}\bm{R}_i(s)\bm{w}_{-i}^{(i)}|
 \le C_Bb_n\eta_n
 \label{eq:main-loo-term}
\end{equation}
with probability $1-O(n^{-B})$.
Uniformity in \eqref{eq:main-loo-term} allows substitution of random parameters in $\mathcal T$.
More explicitly, let $\mathcal B_i$ be the event that \eqref{eq:loo-column-bound} holds for every $j\ge0$, and write $Z=(Z_1,\ldots,Z_n)$.
Since $\mathcal E_{i,B}^{\mathrm{loo}}\in\mathcal F_i$, the tower property gives
\begin{equation*}
 \mathbb{P}(\mathcal E_{i,B}^{\mathrm{loo}}\cap\mathcal B_i^{\mathrm{c}}\mid Z)
 =\mathbb{E}\!\left[\mathbb I_{\mathcal E_{i,B}^{\mathrm{loo}}}
       \mathbb{P}(\mathcal B_i^{\mathrm{c}}\mid\mathcal F_i)\,\middle|\,Z\right]
 \le n^{-B-2}.
\end{equation*}
Summing over $i$ costs at most $n^{-B-1}$.
Adding the probability that some leave-one-out event fails, controlled by Lemma~\ref{lem:loo-basic-event}, gives failure probability $O(n^{-B})$ for \eqref{eq:main-loo-term}.

Finally, \eqref{eq:denominator-lower} and \eqref{eq:schur-decomposed} imply
\begin{equation*}
 |z_i|\le2|w_i|+2|\bm{r}_i^{\mathsf{T}}\bm{R}_i(s)\bm{w}_{-i}^{(i)}|
                   +2|v_i\bm{r}_i^{\mathsf{T}}\bm{R}_i(s)\bm{r}_i|.
\end{equation*}
Intersecting the events and combining \eqref{eq:w-infty}, \eqref{eq:loo-correction}, and \eqref{eq:main-loo-term} gives, uniformly in $s$, $\bm{a}$, and $\bm{v}$,
\begin{equation*}
 \|\bm{z}\|_\infty
 \le C_B\left\{\frac{\epsilon_n}{\sqrt n}+b_n\eta_n+c_n\right\}
 \le C_B\frac{\epsilon_n+5\epsilon_n^2}{\sqrt n}
 \le 6C_B\frac{\epsilon_n}{\sqrt n}.
\end{equation*}
Equation~\eqref{eq:rates-small} controls $b_n\eta_n+c_n$; increase $N_B$ so that $\epsilon_n\le1$ for every integer $n\ge N_B$.
Absorbing the factor $6$ into $C_B$ completes the proof.
\end{proof}

\begin{lemma}[Forcing in the membership space]
\label{lem:group-forcing}
Assume \eqref{eq:main-condition}.
For each sample size $n$, let $\eta\in[0,1]$ be a deterministic constant; the value of $\eta$ may depend on $n$.
Define $\mathcal W_\eta=\{\bm{q}\in\mathrm{span}(\bm{W}):\|\bm{q}\|_2\le\eta\}$.
For every fixed $B>0$, there is a deterministic integer $N_B$ such that, for every integer $n\ge N_B$, with probability $1-O(n^{-B})$,
\begin{equation*}
\sup_{s\ge s_0,\,\bm{q}\in\mathcal W_\eta}
\left\|
\left(\bm{I}_n-\frac1s\bm{M}\right)^{-1}\bm{q}
\right\|_\infty
\le C\frac{\eta}{\sqrt n}.
\end{equation*}
The right-hand side $C\eta/\sqrt n$ is $o(n^{-1/2})$ if $\eta=o(1)$.
\end{lemma}

\begin{proof}
If $\eta=0$, the claim is immediate.
For $\eta>0$, set $\bm{R}_s=(\bm{I}_n-\bm{M}/s)^{-1}$.
Then
\begin{equation*}
 \bm{R}_s\bm{q}=\bm{q}+\bm{R}_s\bm{M}(\bm{q}/s).
\end{equation*}
On the event $\mathcal E_B^{\mathrm{cnt}}$, $\|\bm{W}\|_{2\to\infty}\le C/\sqrt n$ gives $\|\bm{q}\|_\infty\le C\eta/\sqrt n$.
Since $s\ge c_0nL_K$, large enough $C_0$ ensures $\bm{q}/(\eta s)\in\mathcal V$ for $\bm{q}\in\mathrm{span}(\bm{W})$ with $\|\bm{q}\|_2\le\eta$.
Explicitly, for $\bm{q}\in\mathcal W_\eta$ and $\eta>0$,
\begin{equation*}
 \left\|\frac{\bm{q}}{\eta s}\right\|_2\le\frac1{c_0nL_K},\qquad
 \left\|\frac{\bm{q}}{\eta s}\right\|_\infty
 \le\frac{C}{c_0n^{3/2}L_K}.
\end{equation*}
The $\ell_2$ and $\ell_\infty$ estimates verify both constraints of $\mathcal V$ with one choice of $C_0$, uniformly in $\eta$ and $s$.
Linearity supplies the factor $\eta$ when applying Lemma~\ref{lem:unprojected-resolvent} to $\bm{M}(\bm{q}/s)$.
Apply Lemma~\ref{lem:unprojected-resolvent} with $\bm{a}=\bm0$, uniformly in $s$ and $\bm{q}$, to obtain
\begin{equation*}
 \|\bm{R}_s\bm{q}\|_\infty
 \le\frac{C\eta}{\sqrt n}
       +C_B\frac{\eta\epsilon_n}{\sqrt n}.
\end{equation*}
Since $\epsilon_n\to0$, choose a deterministic integer $N_B$ such that $C_B\epsilon_n\le1$ for every integer $n\ge N_B$.
Thus $\|\bm{R}_s\bm{q}\|_\infty\le(C+1)\eta/\sqrt n$, and renaming $C+1$ as $C$ completes the proof.
\end{proof}

\begin{lemma}[Projected uniform rowwise resolvent]
\label{thm:projected-resolvent}
Under \eqref{eq:main-condition}, for every fixed $B>0$, with probability $1-O(n^{-B})$,
\begin{equation*}
\sup_{(s,\bm{a},\bm{v})\in\mathcal T}
\left\|
\left(\bm{I}_n-\frac1s\bm{Q}\bm{M}\bm{Q}\right)^{-1}
\bm{Q}\{\bm{G}^{\mathsf{T}}\bm{a}+\bm{M}\bm{v}\}
\right\|_\infty
\le C_B\frac{\epsilon_n}{\sqrt n}.
\end{equation*}
The right-hand side $C_B\epsilon_n/\sqrt n$ is $o(n^{-1/2})$.
\end{lemma}

\begin{proof}
Fix $B>0$.
Intersect the events of probability $1-O(n^{-B})$ in Lemmas~\ref{lem:unprojected-resolvent} and~\ref{lem:group-forcing} with the concentration event giving \eqref{eq:M-resolvent-op} and \eqref{eq:w-two}.
The intersection has probability $1-O(n^{-B})$.
Put $\bm{w}=\bm{G}^{\mathsf{T}}\bm{a}+\bm{M}\bm{v}$ and
\begin{equation*}
\bm{h}=
\left(\bm{I}_n-\frac1s\bm{Q}\bm{M}\bm{Q}\right)^{-1}\bm{Q}\bm{w}.
\end{equation*}
By \eqref{eq:M-resolvent-op}, $\|\bm{Q}\bm{M}\bm{Q}/s\|_{\mathrm{op}}\le1/4$, so
\begin{equation}
        \|\bm{h}\|_2\le2\|\bm{w}\|_2\le C_B\eta_n
        \label{eq:h-two}
\end{equation}
by \eqref{eq:w-two}.
Also,
\begin{equation}
        \bm{h}=\bm{Q}\bm{w}+\frac1s\bm{Q}\bm{M}\bm{Q}\bm{h}.
 \label{eq:projected-resolvent-equation}
\end{equation}
The right-hand side of \eqref{eq:projected-resolvent-equation} lies in $\mathrm{Im}(\bm{Q})$, so $\bm{h}=\bm{Q}\bm{h}$.
Thus
\begin{equation*}
\bm{h}=\bm{Q}\left(\bm{w}+\frac1s\bm{M}\bm{h}\right)
=\bm{w}-\bm{q}+\frac1s\bm{M}\bm{h},
\end{equation*}
where
\begin{equation*}
        \bm{q}=\bm{\Pi}\left(\bm{w}+\frac1s\bm{M}\bm{h}\right)
        \in\mathrm{span}(\bm{W}).
\end{equation*}
Consequently,
\begin{equation}
        \bm{h}=
        \left(\bm{I}_n-\frac1s\bm{M}\right)^{-1}(\bm{w}-\bm{q}).
        \label{eq:projected-to-unprojected}
\end{equation}
Moreover, by \eqref{eq:w-two}, \eqref{eq:h-two}, and $\|\bm{M}/s\|_{\mathrm{op}}\le1/4$,
\begin{equation}
        \|\bm{q}\|_2
        \le\|\bm{w}\|_2+\frac1s\|\bm{M}\|_{\mathrm{op}}\|\bm{h}\|_2
        \le C_B\eta_n.
        \label{eq:q-correction-norm}
\end{equation}
Lemma~\ref{lem:unprojected-resolvent} bounds the $\bm{w}$ term in \eqref{eq:projected-to-unprojected} by $C_B\epsilon_n/\sqrt n$.
Choose a deterministic integer $N_B$ such that $C_B\eta_n\le1$ for every integer $n\ge N_B$.
Apply Lemma~\ref{lem:group-forcing} with deterministic radius $\eta=C_B\eta_n$ to \eqref{eq:q-correction-norm}.
Uniformity in Lemma~\ref{lem:group-forcing} allows the random $\bm{q}$, yielding
\begin{equation*}
 \|\bm{h}\|_\infty
 \le C_B\frac{\epsilon_n+\eta_n}{\sqrt n}
 \le 3C_B\frac{\epsilon_n}{\sqrt n}.
\end{equation*}
The bound $\eta_n\le2\epsilon_n$ in \eqref{eq:rates-epsilon} gives the factor $3$; absorbing $3$ into $C_B$ proves the lemma.
\end{proof}

\subsubsection{Control of the orthogonal component}

For the next two proofs, choose orthonormal eigenvectors
\begin{equation*}
 \widehat{\bm{V}}=(\widehat{\bm{v}}_1,\ldots,\widehat{\bm{v}}_r)\in\mathbb{R}^{n\times r}
\end{equation*}
of the restriction $\widehat{\bm{A}}|_{\bm{e}^\perp}$ of $\widehat{\bm{A}}$ to $\bm{e}^\perp$, associated with $\widehat\lambda_1\ge\cdots\ge\widehat\lambda_r$.
Lemma~\ref{lem:eig-lower} gives separation with probability $1-O(n^{-B})$ for every fixed $B>0$.
Basis invariance is proved in Lemma~\ref{lem:QhatV-row}.

\begin{lemma}[Admissibility of the random parameters]
\label{lem:random-params}
Assume \eqref{eq:main-condition}.
For $k\in\{1,\ldots,r\}$, put $s_k=\widehat\lambda_k$ and $\bm{o}_k=\bm{V}^{\mathsf{T}}\widehat{\bm{v}}_k$.
On the event that all groups are nonempty and $s_k>0$, define
\begin{equation*}
        \bm{a}_k=\frac{\bm{U}\bm{D}\bm{o}_k}{s_k},
        \qquad
        \bm{v}_k=\frac{\bm{V}\bm{o}_k}{s_k};
\end{equation*}
on the complement of $\{\min_k n_k>0,\ s_k>0\}$, set $\bm{a}_k=\bm0\in\mathbb{R}^p$ and $\bm{v}_k=\bm0\in\mathbb{R}^n$.
For every fixed $B>0$, with probability $1-O(n^{-B})$, simultaneously for $k\in\{1,\ldots,r\}$,
\begin{equation*}
        s_k\ge s_0,
        \qquad
        \bm{a}_k\in\mathcal A,
        \qquad
        \bm{v}_k\in\mathcal V.
\end{equation*}
\end{lemma}

\begin{proof}
Fix $B>0$ and intersect the event $\mathcal E_B^{\mathrm{cnt}}$ with the events in Lemmas~\ref{lem:eig-lower} and~\ref{lem:signal-block-scaled}.
With probability $1-O(n^{-B})$, the eigenvalue bound and Proposition~\ref{prop:population-structure} apply.
Since $\|\bm{o}_k\|_2\le1$,
\begin{equation*}
        \|\bm{v}_k\|_2
        \le\frac{1}{s_0}\le\frac{C}{nL_K},
\end{equation*}
and Proposition~\ref{prop:population-structure} gives
\begin{equation*}
        \|\bm{v}_k\|_\infty
        \le\frac{\|\bm{V}\|_{2\to\infty}\|\bm{o}_k\|_2}{s_0}
        \le\frac{C}{n^{3/2}L_K}.
\end{equation*}
Thus $\bm{v}_k\in\mathcal V$ for large enough $C_0$.

For $\bm{a}_k$, put $\bm{h}_k=\bm{Q}\widehat{\bm{v}}_k$.
Since $\widehat{\bm{v}}_k\in\bm{e}^\perp$, $\widehat{\bm{v}}_k=\bm{V}\bm{o}_k+\bm{h}_k$ gives the block equations
\begin{align}
 &(\bm{\Lambda}+\bm{F})\bm{o}_k+\bm{C}^{\mathsf{T}}\bm{h}_k
=s_k\bm{o}_k,\label{eq:random-param-first-block}\\
 &\bm{C}\bm{o}_k+\bm{Q}\bm{M}\bm{Q}\bm{h}_k
=s_k\bm{h}_k,\notag
\end{align}
where $\bm{F}=\bm{V}^{\mathsf{T}}\bm{E}\bm{V}$ and $\bm{C}=\bm{Q}\bm{E}\bm{V}$.
By \eqref{eq:compressed-noise-bound} and \eqref{eq:signal-eigenvalue-lower}, there is a deterministic integer $N_B$ such that, on the event where \eqref{eq:compressed-noise-bound} and \eqref{eq:signal-eigenvalue-lower} hold, $s_k>\|\bm{Q}\bm{M}\bm{Q}\|_{\mathrm{op}}$ for every integer $n\ge N_B$, so
\begin{equation*}
        \bm{h}_k=(s_k\bm{I}_n-\bm{Q}\bm{M}\bm{Q})^{-1}\bm{C}\bm{o}_k.
\end{equation*}
Substitution into \eqref{eq:random-param-first-block} yields
\begin{equation*}
\left\{
\bm{\Lambda}+\bm{F}+
\bm{C}^{\mathsf{T}}(s_k\bm{I}_n-\bm{Q}\bm{M}\bm{Q})^{-1}\bm{C}
\right\}\bm{o}_k=s_k\bm{o}_k.
\end{equation*}
Multiplying by $\bm{o}_k^{\mathsf{T}}$ gives
\begin{equation}
\begin{split}
\|\bm{D}\bm{o}_k\|_2^2
+\bm{o}_k^{\mathsf{T}}\bm{F}\bm{o}_k
+\bm{o}_k^{\mathsf{T}}\bm{C}^{\mathsf{T}}
(s_k\bm{I}_n-\bm{Q}\bm{M}\bm{Q})^{-1}\bm{C}\bm{o}_k
=s_k\|\bm{o}_k\|_2^2.
\end{split}
\label{eq:random-param-energy}
\end{equation}
Since $s_k>\|\bm{Q}\bm{M}\bm{Q}\|_{\mathrm{op}}$, the matrix $(s_k\bm{I}_n-\bm{Q}\bm{M}\bm{Q})^{-1}$ is positive definite, so the third term on the left-hand side of \eqref{eq:random-param-energy} is nonnegative.
Lemma~\ref{lem:signal-block-scaled} gives
\begin{equation*}
        |\bm{o}_k^{\mathsf{T}}\bm{F}\bm{o}_k|
        \le\|\bm{D}^{-1}\bm{F}\bm{D}^{-1}\|_{\mathrm{op}}
             \|\bm{D}\bm{o}_k\|_2^2
        \le C_B\frac{\epsilon_n}{\sqrt n}\|\bm{D}\bm{o}_k\|_2^2.
\end{equation*}
Since $\epsilon_n/\sqrt n\to0$, increase $N_B$ so that $C_B\epsilon_n/\sqrt n\le1/2$ for every integer $n\ge N_B$.
Let $\delta=\|\bm{D}^{-1}\bm{F}\bm{D}^{-1}\|_{\mathrm{op}}$, so that $\delta\le1/2$ for every integer $n\ge N_B$.
Equation~\eqref{eq:random-param-energy}, together with the nonnegativity of the third term on the left-hand side, yields
\begin{equation*}
 (1-\delta)\|\bm{D}\bm{o}_k\|_2^2
 \le s_k\|\bm{o}_k\|_2^2
 \le s_k.
\end{equation*}
Consequently,
\begin{equation*}
 \|\bm{D}\bm{o}_k\|_2^2\le2s_k.
\end{equation*}
A direct estimate using $\|\bm{D}\|_{\mathrm{op}}=\sqrt{\lambda_1}$ would introduce the strongest signal direction.
Instead, the bound $\|\bm{D}\bm{o}_k\|_2^2\le2s_k$ uses the eigenvalue $s_k$ associated with the selected sample direction.
The bound holds pointwise, without any independence assumption on $\bm{o}_k$.
Finally,
\begin{equation*}
\begin{aligned}
\|\bm{\Sigma}^{1/2}\bm{a}_k\|_2
&\le\sqrt\rho\frac{\|\bm{D}\bm{o}_k\|_2}{s_k}
 \le C\sqrt{\frac{\rho}{s_k}}
 \le C\sqrt{\frac{\rho}{nL_K}},
\end{aligned}
\end{equation*}
so $\bm{a}_k\in\mathcal A$ for large enough $C_0$.
\end{proof}

\begin{lemma}[Rowwise control of the orthogonal component]
\label{lem:QhatV-row}
Under \eqref{eq:main-condition}, for every fixed $B>0$, with probability $1-O(n^{-B})$,
\begin{equation*}
        \|\bm{Q}\widehat{\bm{V}}\|_{2\to\infty}
        \le C_B\frac{\epsilon_n}{\sqrt n}.
\end{equation*}
Consequently, $\|\bm{Q}\widehat{\bm{V}}\|_{2\to\infty}=o_p(n^{-1/2})$.
The bound holds for any orthonormal basis of the leading sample eigenspace.
\end{lemma}

\begin{proof}
Projecting the eigenvalue equation for $\widehat{\bm{v}}_k$ onto $\mathrm{Im}(\bm{Q})$ and using \eqref{eq:QEQ} gives
\begin{equation*}
        s_k\bm{Q}\widehat{\bm{v}}_k
        =\bm{Q}\bm{E}\bm{V}\bm{o}_k+
          \bm{Q}\bm{M}\bm{Q}\widehat{\bm{v}}_k.
\end{equation*}
Therefore
\begin{equation}
\bm{Q}\widehat{\bm{v}}_k
=\left(\bm{I}_n-\frac1{s_k}\bm{Q}\bm{M}\bm{Q}\right)^{-1}
 \bm{Q}\bm{E}\bm{V}\frac{\bm{o}_k}{s_k}.
\label{eq:Qhatv-resolvent}
\end{equation}
Because $\bm{Q}\bm{Y}_{\bm{\mu}}^{\mathsf{T}}=0$, $\bm{Y}_{\bm{\mu}}\bm{V}=\bm{U}\bm{D}$, $\bm{Q}\bm{J}_n=\bm{Q}$, and $\bm{J}_n\bm{V}=\bm{V}$,
\begin{equation*}
\begin{aligned}
 &\bm{Q}\bm{E}\bm{V}\frac{\bm{o}_k}{s_k}
=\bm{Q}\bm{G}^{\mathsf{T}}\bm{U}\bm{D}\frac{\bm{o}_k}{s_k}
  +\bm{Q}\bm{M}\bm{V}\frac{\bm{o}_k}{s_k}\\
 &\quad =\bm{Q}\{\bm{G}^{\mathsf{T}}\bm{a}_k+\bm{M}\bm{v}_k\}.
\end{aligned}
\end{equation*}
By Lemma~\ref{lem:random-params}, $(s_k,\bm{a}_k,\bm{v}_k)$ lies in the parameter set of Lemma~\ref{thm:projected-resolvent}.
Lemma~\ref{thm:projected-resolvent} and \eqref{eq:Qhatv-resolvent} give
\begin{equation*}
        \|\bm{Q}\widehat{\bm{v}}_k\|_\infty
        \le C_B\frac{\epsilon_n}{\sqrt n},
        \qquad k\in\{1,\ldots,r\}.
\end{equation*}
Combining the $r$ columns rowwise yields
\begin{equation*}
        \|\bm{Q}\widehat{\bm{V}}\|_{2\to\infty}
        \le\sqrt r\max_{1\le k\le r}
             \|\bm{Q}\widehat{\bm{v}}_k\|_\infty
        \le C_B\frac{\epsilon_n}{\sqrt n}.
\end{equation*}
For an orthogonal $\bm{S}\in\mathbb{R}^{r\times r}$,
\begin{equation*}
 \|\bm{Q}(\widehat{\bm{V}}\bm{S})\|_{2\to\infty}
 =\|\bm{Q}\widehat{\bm{V}}\|_{2\to\infty}.
\end{equation*}
The rowwise bound is therefore basis-invariant, also on the signal side since $\bm{Q}=\bm{J}_n-\bm{P}$.
The stochastic order follows from $\epsilon_n\to0$.
\end{proof}

\subsubsection{Orthogonal alignment and sign recovery}
The following alignment bound in the $\|\cdot\|_{2\to\infty}$ norm is obtained from \citet{CapeTangPriebe2019}.

\begin{lemma}[Orthogonal alignment]
\label{lem:procrustes}
Let $\bm{V},\widehat{\bm{V}}\in\mathbb{R}^{n\times r}$ have orthonormal columns, and set
\begin{equation*}
 \bm{H}=(\bm{I}_n-\bm{V}\bm{V}^{\mathsf{T}})\widehat{\bm{V}},\qquad
 \bm{O}=\bm{V}^{\mathsf{T}}\widehat{\bm{V}}.
\end{equation*}
If $\|\bm{H}\|_{\mathrm{op}}<1$, then $\bm{O}$ is nonsingular and has the unique orthogonal polar factor $\bm{R}=\bm{O}(\bm{O}^{\mathsf{T}}\bm{O})^{-1/2}$.
Moreover,
\begin{equation}
\|\widehat{\bm{V}}-\bm{V}\bm{R}\|_{2\to\infty}
 \le\|\bm{H}\|_{2\to\infty}
       +\|\bm{V}\|_{2\to\infty}\|\bm{H}\|_{\mathrm{op}}^2.
 \label{eq:deterministic-alignment}
\end{equation}
\end{lemma}

\begin{proof}
Orthonormality gives
\begin{equation*}
 \bm{O}^{\mathsf{T}}\bm{O}
 =\widehat{\bm{V}}^{\mathsf{T}}\bm{V}\bm{V}^{\mathsf{T}}\widehat{\bm{V}}
 =\bm{I}_r-\bm{H}^{\mathsf{T}}\bm{H}.
\end{equation*}
Thus $\bm{O}$ is nonsingular if $\|\bm{H}\|_{\mathrm{op}}<1$, with singular values $s_1,\ldots,s_r$ in $(0,1]$.
For a singular value decomposition $\bm{O}=\bm{U}_1\mathrm{diag}(s_1,\ldots,s_r)\bm{U}_2^{\mathsf{T}}$, the polar factor is $\bm{R}=\bm{U}_1\bm{U}_2^{\mathsf{T}}=\bm{O}(\bm{O}^{\mathsf{T}}\bm{O})^{-1/2}$.
Hence
\begin{equation*}
\begin{aligned}
 &\|\bm{O}-\bm{R}\|_{\mathrm{op}}
 =\max_j(1-s_j)
 \le\max_j(1-s_j^2)\\
 &\quad =\|\bm{I}_r-\bm{O}^{\mathsf{T}}\bm{O}\|_{\mathrm{op}}
 =\|\bm{H}\|_{\mathrm{op}}^2.
\end{aligned}
\end{equation*}
For each row, $\widehat{\bm{V}}-\bm{V}\bm{R}=\bm{H}+\bm{V}(\bm{O}-\bm{R})$ and the triangle inequality imply
\begin{equation*}
 \|\widehat V_{i*}-(\bm{V}\bm{R})_{i*}\|_2
 \le\|H_{i*}\|_2+\|V_{i*}\|_2\|\bm{O}-\bm{R}\|_{\mathrm{op}}.
\end{equation*}
Taking the maximum over $i$ proves \eqref{eq:deterministic-alignment}.
The polar factor depends continuously, and hence measurably, on nonsingular $\bm{O}$.
\end{proof}

\begin{proof}[Proof of Theorem~\ref{thm:rowwise}]
Let $\bm{V}$ and $\widehat{\bm{V}}$ be the arbitrary orthonormal bases from Section~\ref{sec:model}, and put $\bm{H}=(\bm{I}_n-\bm{V}\bm{V}^{\mathsf{T}})\widehat{\bm{V}}$.
The shift $\mathrm{tr}(\bm{\Sigma})\bm{J}_n$ preserves eigenspaces on $\bm{e}^\perp$; on the event in Lemma~\ref{lem:eig-lower},
\begin{equation*}
 \bm{J}_n\widehat{\bm{V}}=\widehat{\bm{V}},
 \qquad
 \bm{H}=(\bm{J}_n-\bm{V}\bm{V}^{\mathsf{T}})\widehat{\bm{V}}=\bm{Q}\widehat{\bm{V}}.
\end{equation*}
For fixed $B>0$, Lemma~\ref{lem:QhatV-row} and Proposition~\ref{prop:population-structure} give, with probability $1-O(n^{-B})$,
\begin{equation*}
 \|\bm{H}\|_{2\to\infty}\le C_B\frac{\epsilon_n}{\sqrt n},\qquad
 \|\bm{V}\|_{2\to\infty}\le\frac{C}{\sqrt n},\qquad
 \|\bm{H}\|_{\mathrm{op}}\le\|\bm{H}\|_F
        \le\sqrt n\,\|\bm{H}\|_{2\to\infty}\le C_B\epsilon_n.
\end{equation*}
Choose a deterministic integer $N_B$ such that
\begin{equation*}
 C_B\epsilon_n<1,\qquad \epsilon_n\le1\qquad(n\ge N_B).
\end{equation*}
By Lemma~\ref{lem:procrustes}, $\bm{O}=\bm{V}^{\mathsf{T}}\widehat{\bm{V}}$ is nonsingular and has the unique orthogonal polar factor $\bm{R}$, with
\begin{equation*}
 \|\widehat{\bm{V}}-\bm{V}\bm{R}\|_{2\to\infty}
 \le C_B\frac{\epsilon_n+\epsilon_n^2}{\sqrt n}
 \le 2C_B\frac{\epsilon_n}{\sqrt n}.
\end{equation*}
Absorb the factor $2$ into $C_B$; since $\epsilon_n\to0$, $\sqrt n\,\|\widehat{\bm{V}}-\bm{V}\bm{R}\|_{2\to\infty}\xrightarrow{p}0$.
\end{proof}

\begin{proof}[Proof of Theorem~\ref{thm:sign-recovery}]
Fix $B>0$ and intersect the event $\mathcal E_B^{\mathrm{cnt}}$ with the event in Theorem~\ref{thm:rowwise}.
Let $\bm{R}$ be the orthogonal polar factor of $\bm{V}^{\mathsf{T}}\widehat{\bm{V}}$ and put $\bm{\Delta}=\widehat{\bm{V}}-\bm{V}\bm{R}$.
Since $\bm{R}\bm{R}^{\mathsf{T}}=\bm{I}_r$,
\begin{equation*}
        \widehat{\bm{P}}-\bm{P}
        =\bm{V}\bm{R}\bm{\Delta}^{\mathsf{T}}
         +\bm{\Delta}\bm{R}^{\mathsf{T}}\bm{V}^{\mathsf{T}}
         +\bm{\Delta}\bm{\Delta}^{\mathsf{T}}.
\end{equation*}
For any two $n\times r$ matrices $\bm{B},\bm{C}$, the Cauchy--Schwarz inequality applied entrywise gives
\begin{equation*}
        \|\bm{B}\bm{C}^{\mathsf{T}}\|_{\max}
        \le\|\bm{B}\|_{2\to\infty}\|\bm{C}\|_{2\to\infty}.
\end{equation*}
Consequently, on the intersection of the event $\mathcal E_B^{\mathrm{cnt}}$ and the event in Theorem~\ref{thm:rowwise},
\begin{align}
 &\|\widehat{\bm{P}}-\bm{P}\|_{\max}
\le2\|\bm{V}\|_{2\to\infty}\|\bm{\Delta}\|_{2\to\infty}
    +\|\bm{\Delta}\|_{2\to\infty}^2\notag\\
 &\quad \le\frac{C_B}{n}(\epsilon_n+\epsilon_n^2)
 \le 2C_B\frac{\epsilon_n}{n}\label{eq:projector-entrywise-proof}
\end{align}
for every integer $n\ge N_B$, after choosing a deterministic integer $N_B$ so that $\epsilon_n\le1$, by Proposition~\ref{prop:population-structure} and \eqref{eq:rowwise-rate}.
Absorb the factor $2$ into $C_B$.
The convergence $\epsilon_n\to0$ gives the stated stochastic order.

Proposition~\ref{prop:population-structure} gives $\min_{i,j}|P_{ij}|\ge c/n$.
Increase $N_B$ so that $C_B\epsilon_n<c/2$ for every integer $n\ge N_B$.
Equation~\eqref{eq:projector-entrywise-proof} then gives $\|\widehat{\bm{P}}-\bm{P}\|_{\max}<c/(2n)$ on the intersection event, so all signs agree.
The signs of $\widehat{\bm{P}}$ determine all within- and between-component relations and hence the true partition up to relabeling.
Integrating the noise bounds and adding the count failure probability gives $O(n^{-B})$ under the joint law.
If $\bm{\Sigma}=\bm0$, then $\widehat{\bm{P}}=\bm{P}$ on the event that all groups are nonempty; the complementary event has exponentially small probability.
\end{proof}

\begin{proof}[Proof of Corollary~\ref{cor:effective-rank}]
Since $\nu=\rho\sqrt{r_2(\bm{\Sigma})}$,
\begin{equation*}
\begin{aligned}
\rho\log n+\frac{a_n}{\sqrt n}
&=\rho\log n
 +\rho\sqrt{\frac{r_2(\bm{\Sigma})\log n}{n}}
 +\rho\frac{\log n}{\sqrt n}.
\end{aligned}
\end{equation*}
Since $\rho\log n/\sqrt n\le\rho\log n$ for every integer $n\ge2$, the claim follows.
If $\bm{\Sigma}=\sigma^2\bm{I}_p$, then $\rho=\sigma^2$ and $r_2(\bm{\Sigma})=p$.
\end{proof}

\subsection{Proofs for the Gaussian cutoff}
\label{sec:gaussian-proofs}
Assume the Gaussian model with fixed mean shape given by \eqref{eq:gauss-shape} and \eqref{eq:gauss-model}.
The sufficient condition \eqref{eq:main-condition} is not imposed in Subsection~\ref{sec:gaussian-proofs}.
Constants may depend on $K$, $\bm{\pi}$, and $\bm{D}_0$.
Put $d_{\max}=\lambda_{\max}(\bm{D}_0)$ and $q_{\max}=\max_k\|\bm{q}_k\|_2$.

\subsubsection{Coordinates and the deterministic sign boundary}
\begin{lemma}[Weighted simplex coordinates]
\label{lem:gauss-coordinates}
Equations~\eqref{eq:gauss-q-mean}--\eqref{eq:gauss-q-geometry} hold.
Let $\bm{Q}_n\in\mathbb{R}^{n\times r}$ be the matrix whose $i$th row is $\bm{q}_{Z_i}^{\mathsf{T}}$, $i\in\{1,\ldots,n\}$.
Define
\begin{equation*}
 \overline{\bm{q}}_n=\frac1n\sum_i\bm{q}_{Z_i},\qquad
 \bm{S}_n=n^{-1}\bm{Q}_n^{\mathsf{T}}\bm{J}_n\bm{Q}_n.
\end{equation*}
Then $\|\bm{S}_n-\bm{I}_r\|_{\mathrm{op}}=O_p(n^{-1/2})$.
Define the event
\begin{equation*}
 \mathcal E_n^{\mathrm G}
 =\{n_k\ge n\pi_{\min}/2\text{ for every }k\in\{1,\ldots,K\}\}.
\end{equation*}
Then $\mathbb{P}(\mathcal E_n^{\mathrm G})\to1$.
On $\mathcal E_n^{\mathrm G}$, $\bm{V}_\circ=n^{-1/2}\bm{J}_n\bm{Q}_n\bm{S}_n^{-1/2}$ is an orthonormal basis of the oracle signal space.
On $\mathcal E_n^{\mathrm G}$, $\|\bm{V}_\circ\|_{2\to\infty}=O(n^{-1/2})$ and $\max_i\|\sqrt n(V_\circ)_{i*}-\bm{q}_{Z_i}^{\mathsf{T}}\|_2=O_p(n^{-1/2})$.
\end{lemma}
\begin{proof}
The identities in \eqref{eq:gauss-q-mean} and \eqref{eq:gauss-q-covariance} follow from \eqref{eq:gauss-shape}.
Let $\bm{W}=(\sqrt{\pi_1}\bm{q}_1,\ldots,\sqrt{\pi_K}\bm{q}_K)$ and $\bm{a}=(\sqrt{\pi_1},\ldots,\sqrt{\pi_K})^{\mathsf{T}}$.
Since $r=K-1$,
\begin{equation*}
 \bm{W}\bm{W}^{\mathsf{T}}=\bm{I}_r,\qquad \bm{W}\bm{a}=\bm0,\qquad \|\bm{a}\|_2=1,
 \qquad \bm{W}^{\mathsf{T}}\bm{W}=\bm{I}_K-\bm{a}\bm{a}^{\mathsf{T}}.
\end{equation*}
Dividing the $(k,\ell)$ entry of $\bm{W}^{\mathsf{T}}\bm{W}$ by $\sqrt{\pi_k\pi_\ell}$ gives the inner-product identity in \eqref{eq:gauss-q-geometry} and positive signed margins.

Independence of the labels, $\mathbb{E}\bm{q}_{Z_i}=\bm0$, and $\mathbb{E}[\bm{q}_{Z_i}\bm{q}_{Z_i}^{\mathsf{T}}]=\bm{I}_r$ give
\begin{equation*}
 \mathbb{E}\|\overline{\bm{q}}_n\|_2^2=\frac rn,\qquad
 \mathbb{E}\left\|n^{-1}\bm{Q}_n^{\mathsf{T}}\bm{Q}_n-\bm{I}_r\right\|_F^2
 =\frac{\sum_k\pi_k\|\bm{q}_k\|_2^4-r}{n}.
\end{equation*}
The Markov inequality yields
\begin{equation*}
 n^{-1}\bm{Q}_n^{\mathsf{T}}\bm{Q}_n=\bm{I}_r+O_p(n^{-1/2}),\qquad
 \overline{\bm{q}}_n=O_p(n^{-1/2}).
\end{equation*}
Since $\bm{S}_n=n^{-1}\bm{Q}_n^{\mathsf{T}}\bm{Q}_n-\overline{\bm{q}}_n\overline{\bm{q}}_n^{\mathsf{T}}$,
\begin{equation*}
 \|\bm{S}_n-\bm{I}_r\|_{\mathrm{op}}
 \le\|n^{-1}\bm{Q}_n^{\mathsf{T}}\bm{Q}_n-\bm{I}_r\|_F+\|\overline{\bm{q}}_n\|_2^2
 =O_p(n^{-1/2})+O_p(n^{-1})=O_p(n^{-1/2}).
\end{equation*}
For fixed $B>0$, choose a deterministic integer $N_B$ by Lemma~\ref{lem:count-concentration} so that $\mathcal E_B^{\mathrm{cnt}}\subseteq\mathcal E_n^{\mathrm G}$ for $n\ge N_B$; then $\mathbb{P}(\mathcal E_n^{\mathrm G})\ge1-O(n^{-B})\to1$.
On $\mathcal E_n^{\mathrm G}$, affine independence gives $\mathrm{rank}(\bm{J}_n\bm{Q}_n)=r$, and
\begin{equation*}
\begin{aligned}
 &\bm{Y}_{\bm{\mu}}=\sqrt{L_n}\bm{U}_n\bm{D}_0^{1/2}\bm{Q}_n^{\mathsf{T}}\bm{J}_n,
 \qquad \mathrm{Im}(\bm{Y}_{\bm{\mu}}^{\mathsf{T}})=\mathrm{Im}(\bm{J}_n\bm{Q}_n),\\
 &\bm{V}_\circ^{\mathsf{T}}\bm{V}_\circ
 =\bm{S}_n^{-1/2}(n^{-1}\bm{Q}_n^{\mathsf{T}}\bm{J}_n\bm{Q}_n)\bm{S}_n^{-1/2}
 =\bm{I}_r.
\end{aligned}
\end{equation*}
Thus $\bm{V}_\circ\bm{V}_\circ^{\mathsf{T}}=\bm{P}$, and \eqref{eq:oracle-projection-entries} gives
\begin{equation*}
 \|\bm{V}_\circ\|_{2\to\infty}^2
 =\max_iP_{ii}=\max_k\left(\frac1{n_k}-\frac1n\right)
 \le\frac{2}{n\pi_{\min}}.
\end{equation*}
With probability tending to one, $\|\bm{S}_n-\bm{I}_r\|_{\mathrm{op}}\le1/2$, and the scalar mean value theorem applied to the eigenvalues gives
\begin{equation*}
\|\bm{S}_n^{-1/2}-\bm{I}_r\|_{\mathrm{op}}
\le C\|\bm{S}_n-\bm{I}_r\|_{\mathrm{op}}
=O_p(n^{-1/2}).
\end{equation*}
Finally,
\begin{equation*}
 \sqrt n(V_\circ)_{i*}-\bm{q}_{Z_i}^{\mathsf{T}}
 =\bm{q}_{Z_i}^{\mathsf{T}}(\bm{S}_n^{-1/2}-\bm{I}_r)
       -\overline{\bm{q}}_n^{\mathsf{T}}\bm{S}_n^{-1/2}.
\end{equation*}
Hence
\begin{equation*}
\begin{aligned}
 &\max_i\|\sqrt n(V_\circ)_{i*}-\bm{q}_{Z_i}^{\mathsf{T}}\|_2
 \le q_{\max}\|\bm{S}_n^{-1/2}-\bm{I}_r\|_{\mathrm{op}}\\
 &\quad+\|\overline{\bm{q}}_n\|_2\|\bm{S}_n^{-1/2}\|_{\mathrm{op}}
 =O_p(n^{-1/2}),
\end{aligned}
\end{equation*}
without a union bound over the rows.
\end{proof}

\begin{lemma}[Properties of the sign boundary functional]
\label{lem:gauss-geometry}
For every $\bm{W}\succ\bm0$, $\mathfrak d(\bm{W})$ is finite and strictly positive, and the infimum in \eqref{eq:gauss-d-functional} is attained.
Let $R\ge0$.
If $0<R^2<\mathfrak d(\bm{W})$, all pairwise inner products of points in the ellipsoids
\begin{equation*}
 \mathcal C_k(\bm{W},R)
 =\{\bm{q}_k+\bm{W}^{1/2}\bm{u}:\|\bm{u}\|_2\le R\}
\end{equation*}
have the prescribed signs, with a uniform strict margin.
If $R^2>\mathfrak d(\bm{W})$, some interior pair has a strictly incorrect sign.
Within a component, the pair can be chosen in disjoint neighborhoods.

For fixed $n$, $t\mapsto\mathfrak d(\bm{\Omega}_n(t))$ is continuous and strictly increasing from zero to infinity.
For $a\ge1$,
\begin{equation}
 a\,\mathfrak d(\bm{\Omega}_n(t))
 \le\mathfrak d(\bm{\Omega}_n(at))
 \le a^2\mathfrak d(\bm{\Omega}_n(t)).
 \label{eq:gauss-d-scaling}
\end{equation}
Consequently \eqref{eq:gauss-critical-equation} has a unique positive solution for $n\ge K\ge2$.
\end{lemma}
\begin{proof}
Each $\mathcal D_{k\ell}$ is closed and contains $(-\bm{q}_k,\bm0)$, while the continuous objective satisfies
\begin{equation*}
 \max\{\bm{x}^{\mathsf{T}}\bm{W}^{-1}\bm{x},\ \bm{y}^{\mathsf{T}}\bm{W}^{-1}\bm{y}\}
 \ge \lambda_{\min}(\bm{W}^{-1})
       \max\{\|\bm{x}\|_2^2,\|\bm{y}\|_2^2\}.
\end{equation*}
Coercivity gives attainment; \eqref{eq:gauss-q-geometry} excludes $\bm{x}=\bm{y}=\bm0$, so $0<\mathfrak d(\bm{W})<\infty$.
Also, $(a\bm{W})^{-1}=a^{-1}\bm{W}^{-1}$ gives $\mathfrak d(a\bm{W})=a^{-1}\mathfrak d(\bm{W})$ for $a>0$.
For $0<\delta<1$, inverse comparison under $(1-\delta)\bm{W}\preceq\bm{W}'\preceq(1+\delta)\bm{W}$ gives
\begin{equation*}
 \frac{\mathfrak d(\bm{W})}{1+\delta}
 \le\mathfrak d(\bm{W}')
 \le\frac{\mathfrak d(\bm{W})}{1-\delta}.
\end{equation*}
The comparison of quadratic forms proves continuity.
For $R^2<\mathfrak d(\bm{W})$, \eqref{eq:gauss-d-functional} and compactness give
\begin{equation*}
 \min_{1\le k\le\ell\le K}
 \min_{\substack{\bm{a}\in\mathcal C_k(\bm{W},R)\\\bm{b}\in\mathcal C_\ell(\bm{W},R)}}
 s_{k\ell}\bm{a}^{\mathsf{T}}\bm{b}>0.
\end{equation*}
For $R^2>\mathfrak d(\bm{W})$, put $\bm{a}=\bm{q}_k+\bm{x}$, $\bm{b}=\bm{q}_\ell+\bm{y}$, and $s=s_{k\ell}$ for a minimizing pair; $\bm{a}$ and $\bm{b}$ are interior.
If $\bm{a}^{\mathsf{T}}\bm{b}=0$, choose $\delta>0$ sufficiently small to keep the perturbed points interior:
\begin{equation*}
\begin{aligned}
 &s(\bm{a}-s\delta\bm{b})^{\mathsf{T}}\bm{b}=-\delta\|\bm{b}\|_2^2<0
 \qquad (\bm{b}\ne\bm0),\\
 &s\bm{a}^{\mathsf{T}}(\bm{b}-s\delta\bm{a})=-\delta\|\bm{a}\|_2^2<0
 \qquad (\bm{b}=\bm0,\ \bm{a}\ne\bm0),\\
 &s(\delta\bm{u})^{\mathsf{T}}(-s\delta\bm{u})=-\delta^2<0
 \qquad (\bm{a}=\bm{b}=\bm0,\ \|\bm{u}\|_2=1).
\end{aligned}
\end{equation*}
If $k=\ell$, strict negativity ensures distinct points, and continuity gives disjoint neighborhoods preserving the sign.

For $a\ge1$, \eqref{eq:gauss-Omega} gives
\begin{equation*}
 a^{-2}\bm{\Omega}_n(t)\preceq\bm{\Omega}_n(at)
                         \preceq a^{-1}\bm{\Omega}_n(t).
\end{equation*}
Taking inverses and infima gives \eqref{eq:gauss-d-scaling} and strict increase for $a>1$.
With $v_n(t)=t^{-1}+\gamma_n t^{-2}$, normalization of $\bm{D}_0$ gives
\begin{align}
 &d_{\max}^{-2}v_n(t)\bm{I}_r \preceq\bm{\Omega}_n(t)\preceq v_n(t)\bm{I}_r,\label{eq:gauss-comparability}\\
 &\frac{\mathfrak d(\bm{I}_r)}{v_n(t)} \le\mathfrak d(\bm{\Omega}_n(t))\le\frac{d_{\max}^{2}\mathfrak d(\bm{I}_r)}{v_n(t)}.\label{eq:gauss-functional-comparability}
\end{align}
By \eqref{eq:gauss-functional-comparability}, for fixed $n$,
\begin{equation*}
\begin{aligned}
 &v_n(t)\to\infty,\qquad
 \mathfrak d(\bm{\Omega}_n(t))\to0
 \qquad (t\to0),\\
 &v_n(t)\to0,\qquad
 \mathfrak d(\bm{\Omega}_n(t))\to\infty
 \qquad (t\to\infty).
\end{aligned}
\end{equation*}
Continuity, strict increase, and $2\log n>0$ give the unique solution of \eqref{eq:gauss-critical-equation}.
\end{proof}

\subsubsection{Gaussian rotation invariance and the residual covariance}
Rescale by $\sigma_n$ and write $t=t_n$, $\gamma=p/n$.
Rescaling leaves $\widehat{\bm{P}}$ unchanged.
In the remainder of Subsection~\ref{sec:gaussian-proofs}, we use the event $\mathcal E_n^{\mathrm G}$ defined in Lemma~\ref{lem:gauss-coordinates}.
On $\mathcal E_n^{\mathrm G}$, put $\bm{Q}=\bm{J}_n-\bm{V}_\circ\bm{V}_\circ^{\mathsf{T}}$ and define
\begin{equation*}
 \bm{A}_\circ=(\bm{V}_\circ^{\mathsf{T}}\widehat{\bm{P}}\bm{V}_\circ)^{1/2},
 \qquad \widetilde{\bm{V}}
       =\widehat{\bm{P}}\bm{V}_\circ\bm{A}_\circ^{-1},
 \qquad \bm{H}_\circ=\bm{Q}\widetilde{\bm{V}},
 \qquad \bm{B}_\circ=(\bm{I}_r-\bm{A}_\circ^2)^{1/2}.
\end{equation*}
Whenever $\bm{A}_\circ$ is invertible, $\widehat{\bm{P}}^{\,2}=\widehat{\bm{P}}$ gives
\begin{equation*}
 \widetilde{\bm{V}}^{\mathsf{T}}\widetilde{\bm{V}}
 =\bm{A}_\circ^{-1}\bm{V}_\circ^{\mathsf{T}}\widehat{\bm{P}}\bm{V}_\circ\bm{A}_\circ^{-1}
 =\bm{I}_r,\qquad
 \bm{V}_\circ^{\mathsf{T}}\widetilde{\bm{V}}=\bm{A}_\circ.
\end{equation*}
Thus $\widetilde{\bm{V}}$ spans the image of $\widehat{\bm{P}}$, and $\bm{V}_\circ\bm{A}_\circ$ is the orthogonal projection of $\widetilde{\bm{V}}$ onto the oracle space.
Subtracting $\bm{V}_\circ\bm{A}_\circ$ leaves $\bm{H}_\circ$, whose Gram matrix is $\bm{I}_r-\bm{A}_\circ^2$ by orthogonality.
Consequently,
\begin{align}
 &\widetilde{\bm{V}} =\bm{V}_\circ\bm{A}_\circ+\bm{H}_\circ,\label{eq:gauss-canonical-identities}\\
 &\widetilde{\bm{V}}^{\mathsf{T}}\widetilde{\bm{V}} =\bm{I}_r,\label{eq:gauss-canonical-orthonormality}\\
 &\bm{H}_\circ^{\mathsf{T}}\bm{H}_\circ =\bm{B}_\circ^2.\label{eq:gauss-canonical-residual-gram}
\end{align}
For every integer $n\ge2r+1$, $\bm{A}_\circ$ and $\bm{B}_\circ$ are positive definite almost surely conditional on $\mathcal E_n^{\mathrm G}$.
Conditional on the labels, the centered data, expressed in an orthonormal basis of $\bm{e}^\perp$, form a $p\times(n-1)$ matrix with an everywhere positive Gaussian density; the corresponding law is therefore mutually absolutely continuous with the central Gaussian law.
Under the central law, right orthogonal invariance makes the rank-$r$ projection uniform on the Grassmannian of $r$-dimensional subspaces of $\bm{e}^\perp$.
For $n-1\ge2r$, the principal angles between the random subspace and the fixed oracle signal space lie in $(0,\pi/2)$ almost surely, since the corresponding Gaussian coordinate blocks have full column rank.
Under the central law, represent the uniform subspace in oracle coordinates by the columns of $\binom{\bm{G}_1}{\bm{G}_2}\bm{T}_*$, where $\bm{G}_1\in\mathbb{R}^{r\times r}$ and $\bm{G}_2\in\mathbb{R}^{(n-r-1)\times r}$ are independent matrices with i.i.d. standard Gaussian entries and $\bm{T}_*=(\bm{G}_1^{\mathsf{T}}\bm{G}_1+\bm{G}_2^{\mathsf{T}}\bm{G}_2)^{-1/2}$.
Both blocks have full column rank almost surely, so
\begin{equation*}
\begin{aligned}
 &\bm{G}_1\bm{T}_*^2\bm{G}_1^{\mathsf{T}}\succ\bm0,\\
 &\bm{I}_r-\bm{T}_*\bm{G}_1^{\mathsf{T}}\bm{G}_1\bm{T}_*
   =\bm{T}_*\bm{G}_2^{\mathsf{T}}\bm{G}_2\bm{T}_*\succ\bm0.
\end{aligned}
\end{equation*}
The matrix $\bm{G}_1\bm{T}_*^2\bm{G}_1^{\mathsf{T}}$ represents $\bm{A}_\circ^2$ under the central law and has the same eigenvalues as $\bm{T}_*\bm{G}_1^{\mathsf{T}}\bm{G}_1\bm{T}_*$, proving $\bm0\prec\bm{A}_\circ^2\prec\bm{I}_r$.
Absolute continuity transfers the positive-definiteness properties to the present model.
Under the Gaussian model, the eigenvalue $\lambda_r(\bm{Y}^{\mathsf{T}}\bm{Y})$ is simple almost surely.
Define auxiliary quantities arbitrarily on $(\mathcal E_n^{\mathrm G})^{\mathrm c}$.
For related rotation-invariance and Gaussian frame arguments, see \citet{Jiang2006} and \citet{Paul2007}.

\begin{lemma}[Conditional Gaussian representation]
\label{lem:gauss-Haar}
For every integer $n\ge2r+1$, conditional on a label realization in $\mathcal E_n^{\mathrm G}$ and on $\bm{A}_\circ$, the matrix $\bm{H}_\circ\bm{B}_\circ^{-1}$ is uniform among orthonormal $r$-frames in $\mathrm{Im}(\bm{Q})$.
The same joint law of the labels and $\widehat{\bm{P}}$ admits a realization in which, on $\mathcal E_n^{\mathrm G}$,
\begin{equation*}
 \widetilde{\bm{V}}
 =\bm{V}_\circ\bm{A}_\circ
  +\bm{Q}\bm{Z}_*(\bm{Z}_*^{\mathsf{T}}\bm{Q}\bm{Z}_*)^{-1/2}\bm{B}_\circ,
\end{equation*}
where $\bm{Z}_*\in\mathbb{R}^{n\times r}$ has i.i.d. standard Gaussian entries and is independent of the labels and $\bm{A}_\circ$.
If row $i$ of $\bm{Z}_*$ is $\bm{z}_i^{\mathsf{T}}$, then
\begin{equation}
 \max_i\|\sqrt n\,\widetilde V_{i*}
       -\bm{q}_{Z_i}^{\mathsf{T}}\bm{A}_\circ-\bm{z}_i^{\mathsf{T}}\bm{B}_\circ\|_2
       =O_p\!\left(\sqrt{\frac{\log n}{n}}\right).
 \label{eq:gauss-row-Haar}
\end{equation}
The representation holds at every signal level.
\end{lemma}
\begin{proof}
Fix $Z=(Z_1,\ldots,Z_n)$ in $\mathcal E_n^{\mathrm G}$ and set
\begin{equation*}
 \mathcal R_Z=\{\bm{R}\in\mathcal{O}(n):\bm{R}\bm{e}=\bm{e},\ \bm{R}\bm{V}_\circ=\bm{V}_\circ\},
 \qquad \bm{F}_*=\bm{H}_\circ\bm{B}_\circ^{-1}.
\end{equation*}
By \eqref{eq:gauss-canonical-residual-gram}, $\bm{F}_*^{\mathsf{T}}\bm{F}_*=\bm{I}_r$ and $\bm{Q}\bm{F}_*=\bm{F}_*$.
Right multiplication of the data by $\bm{R}^{\mathsf{T}}$, $\bm{R}\in\mathcal R_Z$, preserves the conditional law of the data and gives
\begin{equation*}
 (\bm{A}_\circ,\bm{F}_*)\ \overset{\mathrm d}=\ 
 (\bm{A}_\circ,\bm{R}\bm{F}_*)\qquad\text{conditionally on }Z.
\end{equation*}
Let $\mu_Z$ be normalized Haar measure on $\mathcal R_Z$ and $\nu_Z$ the uniform probability measure on the frame space.
For bounded Borel functions $f$ and $g$, transitivity on the frame space and invariance give
\begin{equation*}
\begin{aligned}
 &\mathbb{E}[g(\bm{A}_\circ)f(\bm{F}_*)\mid Z]
 =\mathbb{E}\!\left[
 g(\bm{A}_\circ)\int_{\mathcal R_Z}f(\bm{R}\bm{F}_*)\,d\mu_Z(\bm{R})
 \,\middle|\,Z\right]\\
 &\quad =\mathbb{E}[g(\bm{A}_\circ)\mid Z]\int f\,d\nu_Z.
\end{aligned}
\end{equation*}
Thus the conditional law of $\bm{F}_*$ given $(Z,\bm{A}_\circ)$ is $\nu_Z$.
For an auxiliary matrix $\bm{Z}_*$ with i.i.d. standard Gaussian entries, independent of $(Z,\bm{A}_\circ)$,
\begin{equation*}
 \bm{Q}\bm{Z}_*(\bm{Z}_*^{\mathsf{T}}\bm{Q}\bm{Z}_*)^{-1/2}
\end{equation*}
has conditional law $\nu_Z$ by the Gaussian polar construction in $\mathrm{Im}(\bm{Q})$; see \citet{JauchHoffDunson2021}.
Retain the original joint law of $(Z,\bm{A}_\circ)$ and, on $\mathcal E_n^{\mathrm G}$, replace $\bm{F}_*$ by the Gaussian polar factor, whose conditional law given $(Z,\bm{A}_\circ)$ is $\nu_Z$.
By \eqref{eq:gauss-canonical-identities} and \eqref{eq:gauss-canonical-orthonormality}, the replacement preserves the conditional law of $\widehat{\bm{P}}$ given $(Z,\bm{A}_\circ)$ on $\mathcal E_n^{\mathrm G}$.
Retaining the original construction on $(\mathcal E_n^{\mathrm G})^{\mathrm c}$ therefore gives a realization with the original joint law of $(Z,\widehat{\bm{P}})$.

We establish the following stochastic orders uniformly over label realizations in $\mathcal E_n^{\mathrm G}$:
\begin{align}
 &\|\bm{Z}_*^{\mathsf{T}}\bm{Q}\bm{Z}_*/n-\bm{I}_r\|_{\mathrm{op}} =O_p(n^{-1/2}),\label{eq:gauss-frame-moments}\\
 &\|(\bm{e},\bm{V}_\circ)^{\mathsf{T}}\bm{Z}_*\|_{\mathrm{op}} =O_p(1),\label{eq:gauss-frame-projection-moment}\\
 &\max_i\|\bm{z}_i\|_2 =O_p(\sqrt{\log n}).\label{eq:gauss-frame-row-maximum}
\end{align}
For \eqref{eq:gauss-frame-moments}, $\mathrm{rank}(\bm{Q})=n-r-1$ and Gaussian rotation invariance give, conditionally on $Z$,
\begin{equation*}
 \bm{Z}_*^{\mathsf{T}}\bm{Q}\bm{Z}_*
 \ \overset{\mathrm d}=\ 
 \sum_{j=1}^{n-r-1}\bm{w}_j\bm{w}_j^{\mathsf{T}},
\end{equation*}
where $\bm{w}_1,\ldots,\bm{w}_{n-r-1}$ are independent $N(\bm0,\bm{I}_r)$ vectors.
For $a,b\in\{1,\ldots,r\}$,
\begin{equation*}
 \begin{aligned}
 &\mathbb{E}[(\bm{Z}_*^{\mathsf{T}}\bm{Q}\bm{Z}_*)_{ab}\mid Z]
   =(n-r-1)\mathbb I_{\{a=b\}},\\
 &\operatorname{Var}[(\bm{Z}_*^{\mathsf{T}}\bm{Q}\bm{Z}_*)_{ab}\mid Z]
   =(n-r-1)(1+\mathbb I_{\{a=b\}}).\end{aligned}
\end{equation*}
Consequently,
\begin{equation*}
 \mathbb{E}\!\left[\left\|\bm{Z}_*^{\mathsf{T}}\bm{Q}\bm{Z}_*/n-\bm{I}_r\right\|_F^2\,\middle|\,Z\right]
 =\frac{(n-r-1)r(r+1)+r(r+1)^2}{n^2}
 =\frac{r(r+1)}n.
\end{equation*}
The Frobenius norm dominates the operator norm, so Markov's inequality gives, for every $M>0$,
\begin{equation*}
 \mathbb{P}\!\left\{\sqrt n\left\|\bm{Z}_*^{\mathsf{T}}\bm{Q}\bm{Z}_*/n-\bm{I}_r\right\|_{\mathrm{op}}>M\,\middle|\,Z\right\}
 \le\frac{r(r+1)}{M^2}.
\end{equation*}
Letting $M\to\infty$ verifies the uniform $O_p(n^{-1/2})$ assertion in \eqref{eq:gauss-frame-moments}.
Conditional on $Z$, the $(r+1)\times r$ matrix $(\bm{e},\bm{V}_\circ)^{\mathsf{T}}\bm{Z}_*$ has i.i.d. standard Gaussian entries, so its finite second moment gives \eqref{eq:gauss-frame-projection-moment}.
A Gaussian tail bound and a union bound over the $nr$ entries of $\bm{Z}_*$ give \eqref{eq:gauss-frame-row-maximum}.
Also,
\begin{equation*}
 \bm{Q}\bm{Z}_*-\bm{Z}_*
 =-\bm{e}(\bm{e}^{\mathsf{T}}\bm{Z}_*)-\bm{V}_\circ(\bm{V}_\circ^{\mathsf{T}}\bm{Z}_*)
\end{equation*}
and Lemma~\ref{lem:gauss-coordinates}, \eqref{eq:gauss-frame-moments} and \eqref{eq:gauss-frame-projection-moment}, and the scalar mean value theorem applied to eigenvalues give
\begin{equation*}
\begin{aligned}
 &\|\bm{Q}\bm{Z}_*-\bm{Z}_*\|_{2\to\infty}
 \le(n^{-1/2}+\|\bm{V}_\circ\|_{2\to\infty})
        \|(\bm{e},\bm{V}_\circ)^{\mathsf{T}}\bm{Z}_*\|_{\mathrm{op}}
 =O_p(n^{-1/2}),\\
 &\|(\bm{Z}_*^{\mathsf{T}}\bm{Q}\bm{Z}_*/n)^{-1/2}-\bm{I}_r\|_{\mathrm{op}}
 =O_p(n^{-1/2}).
\end{aligned}
\end{equation*}
Combining the bounds for the rows and the inverse square root with \eqref{eq:gauss-frame-row-maximum} yields
\begin{equation*}
 \|\sqrt n\bm{Q}\bm{Z}_*(\bm{Z}_*^{\mathsf{T}}\bm{Q}\bm{Z}_*)^{-1/2}
                         -\bm{Z}_*\|_{2\to\infty}
 =O_p(\sqrt{\log n/n}).
\end{equation*}
Use $\|\bm{A}_\circ\|_{\mathrm{op}},\|\bm{B}_\circ\|_{\mathrm{op}}\le1$ and the row bound in Lemma~\ref{lem:gauss-coordinates} to obtain \eqref{eq:gauss-row-Haar}.
\end{proof}

Complete $\bm{U}_n$ to an orthonormal basis $(\bm{U}_n,\bm{U}_\perp)$ of $\mathbb{R}^p$ and $\bm{V}_\circ$ to $(\bm{V}_\circ,\bm{V}_\perp)$ on $\bm{e}^\perp$.
Put $N=n-r-1$.
Then $\bm{V}_\perp\in\mathbb{R}^{n\times N}$ has orthonormal columns spanning $\mathrm{Im}(\bm{Q})$.
In the bases $(\bm{U}_n,\bm{U}_\perp)$ and $(\bm{V}_\circ,\bm{V}_\perp)$, the data divided by $\sigma_n$ form a $p\times(n-1)$ matrix
\begin{equation*}
 \mathcal Y=\begin{pmatrix}
       \bm{R}_n+\bm{G}_{11}&\bm{G}_{12}\\
       \bm{G}_{21}&\bm{G}_{22}
       \end{pmatrix},
 \qquad
 \bm{R}_n=\sqrt{nt}\,\bm{D}_0^{1/2}\bm{S}_n^{1/2}.
\end{equation*}
Given the labels, the noise blocks $\bm{G}_{11}$, $\bm{G}_{12}$, $\bm{G}_{21}$, and $\bm{G}_{22}$ are independent Gaussian matrices with i.i.d. standard entries and sizes $r\times r$, $r\times N$, $(p-r)\times r$, and $(p-r)\times N$, respectively.
The bottom blocks are empty if $p=r$.
To check independence despite sample centering, let $\bm{O}_{\mathrm L}$ be the completed $p\times p$ left orthogonal basis and $\bm{O}_{\mathrm R}=(\bm{V}_\circ,\bm{V}_\perp)$.
Conditional on the labels, $\bm{O}_{\mathrm L}$ and $\bm{O}_{\mathrm R}$ are deterministic and $\bm{J}_n\bm{O}_{\mathrm R}=\bm{O}_{\mathrm R}$.
For $\bm{X}=(\bm{\xi}_1,\ldots,\bm{\xi}_n)$, the transformed noise $\bm{Z}=\bm{O}_{\mathrm L}^{\mathsf{T}}\bm{X}\bm{O}_{\mathrm R}$ is jointly Gaussian with
\begin{equation*}
 \operatorname{Cov}(Z_{ab},Z_{cd}\mid Z_1,\ldots,Z_n)
 =(\bm{O}_{\mathrm L}^{\mathsf{T}}\bm{O}_{\mathrm L})_{ac}
  (\bm{O}_{\mathrm R}^{\mathsf{T}}\bm{O}_{\mathrm R})_{bd}
 =\mathbb I_{\{a=c\}}\mathbb I_{\{b=d\}}.
\end{equation*}
The transformed noise entries are therefore independent standard Gaussians, and disjoint blocks are independent.
Independence is used in the specified orthonormal coordinates, not for the original centered columns.

Although $\bm{R}_n$ need not be symmetric,
\begin{equation*}
 \bm{R}_n^{\mathsf{T}}\bm{R}_n=nt\,\bm{D}_{0,n},\qquad
 \bm{D}_{0,n}:=\bm{S}_n^{1/2}\bm{D}_0\bm{S}_n^{1/2}
                 =\bm{D}_0+O_p(n^{-1/2}).
\end{equation*}
The shifted Gram matrix is
\begin{equation}
 \mathcal Y^{\mathsf{T}}\mathcal Y-p\bm{I}_{n-1}
 =\begin{pmatrix}nt\bm{D}_{0,n}+\bm{F}_\circ&\bm{C}_\circ^{\mathsf{T}}\\
                  \bm{C}_\circ&\bm{T}_\circ\end{pmatrix},
 \label{eq:gauss-Gram-blocks}
\end{equation}
where
\begin{align*}
 &\bm{F}_\circ=\bm{R}_n^{\mathsf{T}}\bm{G}_{11}+\bm{G}_{11}^{\mathsf{T}}\bm{R}_n
          +\bm{G}_{11}^{\mathsf{T}}\bm{G}_{11}+\bm{G}_{21}^{\mathsf{T}}\bm{G}_{21}-p\bm{I}_r,\\
 &\bm{C}_\circ=\bm{G}_{12}^{\mathsf{T}}(\bm{R}_n+\bm{G}_{11})
                                      +\bm{G}_{22}^{\mathsf{T}}\bm{G}_{21},\\
 &\bm{T}_\circ=\bm{G}_{12}^{\mathsf{T}}\bm{G}_{12}
                                      +\bm{G}_{22}^{\mathsf{T}}\bm{G}_{22}-p\bm{I}_N.
\end{align*}

On $\mathcal E_n^{\mathrm G}$, express the canonical basis in the $(\bm{V}_\circ,\bm{V}_\perp)$ coordinates as
\begin{equation*}
 \bm{H}_\perp=\bm{V}_\perp^{\mathsf{T}}\bm{H}_\circ,\qquad
 \widetilde{\bm{V}}_{\mathrm{blk}}
   =\begin{pmatrix}\bm{A}_\circ\\\bm{H}_\perp\end{pmatrix},\qquad
 \bm{H}_\perp^{\mathsf{T}}\bm{H}_\perp=\bm{B}_\circ^2.
\end{equation*}
The columns of $\widetilde{\bm{V}}_{\mathrm{blk}}$ are an orthonormal basis of the leading $r$-dimensional invariant subspace of $\mathcal Y^{\mathsf{T}}\mathcal Y-p\bm{I}_{n-1}$.
The restriction to $\mathrm{Im}(\widetilde{\bm{V}}_{\mathrm{blk}})$ is represented by
\begin{equation*}
 \bm{J}_\circ
 =\widetilde{\bm{V}}_{\mathrm{blk}}^{\mathsf{T}}
   (\mathcal Y^{\mathsf{T}}\mathcal Y-p\bm{I}_{n-1})
   \widetilde{\bm{V}}_{\mathrm{blk}}.
\end{equation*}
Thus $\bm{J}_\circ$ is symmetric with the leading $r$ eigenvalues.
Invariance gives
\begin{align}
 &(nt\bm{D}_{0,n}+\bm{F}_\circ)\bm{A}_\circ+\bm{C}_\circ^{\mathsf{T}}\bm{H}_\perp =\bm{A}_\circ\bm{J}_\circ,\label{eq:gauss-canonical-block-equations}\\
 &\bm{C}_\circ\bm{A}_\circ+\bm{T}_\circ\bm{H}_\perp =\bm{H}_\perp\bm{J}_\circ.\label{eq:gauss-canonical-lower-block}
\end{align}
The identities in \eqref{eq:gauss-canonical-block-equations} and \eqref{eq:gauss-canonical-lower-block} hold at every signal level, without eigenspace consistency.

\begin{lemma}[Gaussian block estimates]
\label{lem:gauss-blocks}
The blocks in \eqref{eq:gauss-Gram-blocks} satisfy
\begin{equation*}
 \|\bm{F}_\circ\|_{\mathrm{op}}=O_p(\sqrt{nt}+\sqrt p+1),\qquad
 \|\bm{T}_\circ\|_{\mathrm{op}}\le C(\sqrt{np}+n),
\end{equation*}
where the bound for $\|\bm{T}_\circ\|_{\mathrm{op}}$ holds with probability tending to one.
If $nt+p\to\infty$, then
\begin{equation}
 \|\bm{C}_\circ^{\mathsf{T}}\bm{C}_\circ-n(nt\bm{D}_0+p\bm{I}_r)\|_{\mathrm{op}}
                  =o_p\{n(nt+p)\}.
 \label{eq:gauss-C-energy}
\end{equation}
If $t+\gamma$ is bounded, then
\begin{align}
 &\|\bm{F}_\circ\|_{\mathrm{op}}/n =o_p(1),\label{eq:gauss-small-blocks}\\
 &\|\bm{C}_\circ\|_{\mathrm{op}}^2/n^2 \le d_{\max}t+\gamma+o_p(1).\label{eq:gauss-small-cross-block}
\end{align}
\end{lemma}
\begin{proof}
For fixed $r$, second moments give $\|\bm{G}_{11}\|_{\mathrm{op}}=O_p(1)$ and $\|\bm{G}_{21}^{\mathsf{T}}\bm{G}_{21}-(p-r)\bm{I}_r\|_{\mathrm{op}}=O_p(\sqrt p+1)$.
Since $\|\bm{R}_n\|_{\mathrm{op}}=O_p(\sqrt{nt})$, the bound for $\|\bm{F}_\circ\|_{\mathrm{op}}$ follows.
For a fixed unit vector $\bm{u}\in\mathbb{R}^N$, $\bm{u}^{\mathsf{T}}\bm{T}_\circ\bm{u}+p\sim\chi_p^2$, where $\chi_p^2$ denotes the chi-squared distribution with $p$ degrees of freedom.
Let $X_p\sim\chi_p^2$.
The moment generating function of $X_p$ is $\mathbb{E}[e^{sX_p}]=(1-2s)^{-p/2}$ for $s<1/2$.
The Chernoff inequality gives the following bound for $x>0$ \citep{LaurentMassart2000}:
\begin{equation}
\mathbb{P}\{|X_p-p|>2\sqrt{px}+2x\}\le 2e^{-x}.
\label{eq:chi-square-tail}
\end{equation}
Applying \eqref{eq:chi-square-tail} with $x=4N$ on a $1/4$-net of at most $9^N$ points, together with Lemma~\ref{lem:net-reconstruction}, gives a failure probability of at most $2\exp\{-(4-\log9)N\}$, yielding the bound for $\bm{T}_\circ$.

Given $\bm{G}_{11},\bm{G}_{21}$ and the labels, the $N$ rows of $\bm{C}_\circ$ are independent centered Gaussian vectors with covariance
\begin{equation}
 \bm{K}_n=(\bm{R}_n+\bm{G}_{11})^{\mathsf{T}}(\bm{R}_n+\bm{G}_{11})
                                      +\bm{G}_{21}^{\mathsf{T}}\bm{G}_{21}.
 \label{eq:gauss-conditional-covariance}
\end{equation}
The covariance in \eqref{eq:gauss-conditional-covariance} satisfies
\begin{equation}
 \bm{K}_n=nt\bm{D}_0+p\bm{I}_r
           +O_p(nt/\sqrt n+\sqrt{nt}+\sqrt p+1).
\label{eq:gauss-covariance-expansion}
\end{equation}
To make the moment bound explicit, condition on $\bm{G}_{11},\bm{G}_{21},Z$ and let $\bm{Z}_0\in\mathbb{R}^{N\times r}$ have independent standard Gaussian entries.
Set $\bm{T}_N=N^{-1}\bm{Z}_0^{\mathsf{T}}\bm{Z}_0$.
Conditional on $\bm{G}_{11},\bm{G}_{21},Z$, the matrix $N^{-1}\bm{C}_\circ^{\mathsf{T}}\bm{C}_\circ$ has the law of $\bm{K}_n^{1/2}\bm{T}_N\bm{K}_n^{1/2}$, and
\begin{equation*}
 \mathbb{E}\|\bm{T}_N-\bm{I}_r\|_F^2=\frac{r(r+1)}{N},\qquad
 \|\bm{K}_n^{1/2}(\bm{T}_N-\bm{I}_r)\bm{K}_n^{1/2}\|_{\mathrm{op}}
 \le\|\bm{K}_n\|_{\mathrm{op}}\|\bm{T}_N-\bm{I}_r\|_F.
\end{equation*}
Markov's inequality gives $\|\bm{T}_N-\bm{I}_r\|_F=O_p(N^{-1/2})$ with a bound independent of the conditioned covariance $\bm{K}_n$.
Integrating the conditional bound gives
\begin{equation}
 \frac{\|N^{-1}\bm{C}_\circ^{\mathsf{T}}\bm{C}_\circ-\bm{K}_n\|_{\mathrm{op}}}{\|\bm{K}_n\|_{\mathrm{op}}}
                  =O_p(N^{-1/2}).
\label{eq:gauss-Wishart-error}
\end{equation}
When $\|\bm{K}_n\|_{\mathrm{op}}=0$, the conditional Gaussian representation gives $\bm{C}_\circ=\bm0$ almost surely, and we define the ratio to be zero.
If $nt+p\to\infty$, \eqref{eq:gauss-covariance-expansion} implies $\|\bm{K}_n\|_{\mathrm{op}}=O_p(nt+p)$ and $\|\bm{K}_n-nt\bm{D}_0-p\bm{I}_r\|_{\mathrm{op}}=o_p(nt+p)$.
Multiplying the normalized error in \eqref{eq:gauss-Wishart-error} by $\|\bm{K}_n\|_{\mathrm{op}}$ gives $\|N^{-1}\bm{C}_\circ^{\mathsf{T}}\bm{C}_\circ-\bm{K}_n\|_{\mathrm{op}}=o_p(nt+p)$, and $N/n\to1$ gives \eqref{eq:gauss-C-energy}.
For bounded $t+\gamma$, \eqref{eq:gauss-covariance-expansion} gives $\bm{K}_n/n=t\bm{D}_0+\gamma\bm{I}_r+o_p(1)$ without requiring $nt+p$ to diverge; combining \eqref{eq:gauss-Wishart-error} with the bound for $\bm{F}_\circ$ proves \eqref{eq:gauss-small-blocks} and \eqref{eq:gauss-small-cross-block}.
\end{proof}

Define
\begin{equation*}
 v_n=t_n^{-1}+\gamma_nt_n^{-2},\qquad
 \iota_n=v_n^{-1},\qquad \theta_n=v_n\log n.
\end{equation*}

\begin{lemma}[Residual covariance with a nonspherical signal shape]
\label{lem:gauss-residual}
If $\iota_n\to\infty$, then
\begin{align}
 &\|\bm{B}_\circ^2-\bm{\Omega}_n(t_n)\|_{\mathrm{op}} =o_p(v_n),\label{eq:gauss-residual-cov}\\
 &\|\bm{A}_\circ-\bm{I}_r\|_{\mathrm{op}} =O_p(v_n).\label{eq:gauss-alignment-bound}
\end{align}
If also $\theta_n=O(1)$, the representation of Lemma~\ref{lem:gauss-Haar} can be chosen so that
\begin{equation}
\max_i\|\sqrt n\,\widetilde V_{i*}
       -\bm{q}_{Z_i}^{\mathsf{T}}-\bm{z}_i^{\mathsf{T}}\bm{\Omega}_n(t_n)^{1/2}\|_2=o_p(1).
 \label{eq:gauss-critical-coupling}
\end{equation}
\end{lemma}
\begin{proof}
Since $v_n\to0$, we have $t_n\to\infty$ and $\sqrt{\gamma_n}/t_n\to0$.
Lemma~\ref{lem:gauss-blocks} gives
\begin{equation*}
 \|\bm{F}_\circ\|_{\mathrm{op}}/(nt)=o_p(1),\qquad
 \|\bm{T}_\circ\|_{\mathrm{op}}/(nt)=o_p(1),\qquad
 \|\bm{C}_\circ\|_{\mathrm{op}}/(nt)=O_p(\sqrt{v_n}).
\end{equation*}
Since $v_n\to0$ implies $nt+p\to\infty$, \eqref{eq:gauss-C-energy} applies.
In particular,
\begin{equation*}
 \frac{\sqrt{np}+n}{nt}=\frac{\sqrt\gamma+1}{t}\to0,
 \qquad
 \frac{\|\bm{C}_\circ\|_{\mathrm{op}}^2}{n^2t^2}
 =O_p\!\left(\frac{t+\gamma}{t^2}\right)=O_p(v_n).
\end{equation*}
Also $\|\bm{F}_\circ\|_{\mathrm{op}}/(nt)=O_p((nt)^{-1/2}+\sqrt\gamma/(\sqrt n\,t)+(nt)^{-1})=o_p(1)$.
The estimates for $\bm{F}_\circ,\bm{T}_\circ,\bm{C}_\circ$ control all blocks relative to the smallest eigenvalue of the fixed matrix $\bm{D}_0$.

After division by $nt$, \eqref{eq:gauss-Gram-blocks} is an $o_p(1)$ perturbation of $\mathrm{diag}(\bm{D}_0,\bm0)$.
The Weyl inequality separates the leading $r$ eigenvalues, which lie between $cnt$ and $Cnt$ before scaling.
For $j\in\{1,\ldots,r\}$, let
\begin{equation*}
 s_j=\lambda_j(\mathcal Y^{\mathsf{T}}\mathcal Y-p\bm{I}_{n-1}),
 \qquad
 \begin{pmatrix}\bm{o}_j\\\bm{h}_j\end{pmatrix}
 \in\mathbb{R}^r\times\mathbb{R}^N
\end{equation*}
denote the $j$th largest eigenvalue and an associated unit eigenvector, respectively.
The lower block of the eigenvalue equation for \eqref{eq:gauss-Gram-blocks} is
\begin{equation}
 (s_j\bm{I}_N-\bm{T}_\circ)\bm{h}_j=\bm{C}_\circ\bm{o}_j.
 \label{eq:gauss-eigenvector-lower-block}
\end{equation}
Since $s_j\ge cnt$ and $\|\bm{T}_\circ\|_{\mathrm{op}}=o_p(nt)$, the coefficient matrix in \eqref{eq:gauss-eigenvector-lower-block} has inverse norm $O_p((nt)^{-1})$, giving $\|\bm{h}_j\|_2=O_p(\sqrt{v_n})$.
Combining the columns and rotating gives $\|\bm{H}_\circ\|_{\mathrm{op}}=O_p(\sqrt{v_n})$.
By \eqref{eq:gauss-canonical-residual-gram}, $\bm0\preceq\bm{A}_\circ\preceq\bm{I}_r$ and
\begin{equation*}
 \|\bm{A}_\circ-\bm{I}_r\|_{\mathrm{op}}
 \le \|\bm{I}_r-\bm{A}_\circ^2\|_{\mathrm{op}}
 =\|\bm{H}_\circ\|_{\mathrm{op}}^2=O_p(v_n).
\end{equation*}

Since $\bm{A}_\circ\xrightarrow{p}\bm{I}_r$, the equation in \eqref{eq:gauss-canonical-block-equations} gives
\begin{equation*}
 \frac{\bm{J}_\circ}{nt}
 =\bm{A}_\circ^{-1}
  \left\{\left(\bm{D}_{0,n}+\frac{\bm{F}_\circ}{nt}\right)\bm{A}_\circ
             +\frac{\bm{C}_\circ^{\mathsf{T}}\bm{H}_\perp}{nt}\right\}
 =\bm{D}_0+o_p(1),
\end{equation*}
where $\|\bm{C}_\circ^{\mathsf{T}}\bm{H}_\perp\|_{\mathrm{op}}/(nt)=O_p(v_n)$.
For the similarity term $\bm{A}_\circ^{-1}\bm{D}_0\bm{A}_\circ$, $\bm{A}_\circ\xrightarrow{p}\bm{I}_r$ implies $\|\bm{A}_\circ^{-1}\|_{\mathrm{op}}=O_p(1)$, and
\begin{equation}
 \bm{A}_\circ^{-1}\bm{D}_0\bm{A}_\circ-\bm{D}_0
 =\bm{A}_\circ^{-1}\{\bm{D}_0(\bm{A}_\circ-\bm{I}_r)
                  -(\bm{A}_\circ-\bm{I}_r)\bm{D}_0\}.
 \label{eq:gauss-similarity-error}
\end{equation}
The operator norm of the right-hand side of \eqref{eq:gauss-similarity-error} is at most $2\|\bm{A}_\circ^{-1}\|_{\mathrm{op}}\|\bm{D}_0\|_{\mathrm{op}}\|\bm{A}_\circ-\bm{I}_r\|_{\mathrm{op}}=O_p(v_n)$.
The errors $\bm{D}_{0,n}-\bm{D}_0$, $\bm{F}_\circ/(nt)$, and $\bm{C}_\circ^{\mathsf{T}}\bm{H}_\perp/(nt)$ are all $o_p(1)$.
Thus $\bm{J}_\circ/(nt)=\bm{D}_0+o_p(1)$ does not rely on commutation of the matrices.

Set $\bm{E}_J=\bm{J}_\circ-nt\bm{D}_0$, so $\|\bm{E}_J\|_{\mathrm{op}}/(nt)=o_p(1)$.
The equation in \eqref{eq:gauss-canonical-lower-block} gives
\begin{equation*}
 \bm{H}_\perp-\frac1{nt}\bm{C}_\circ\bm{D}_0^{-1}
 =\frac1{nt}\{\bm{C}_\circ(\bm{A}_\circ-\bm{I}_r)
       +\bm{T}_\circ\bm{H}_\perp-\bm{H}_\perp\bm{E}_J\}\bm{D}_0^{-1}.
\end{equation*}
Submultiplicativity of the operator norm and the preceding estimates give
\begin{align*}
 &\frac{\|\bm{C}_\circ(\bm{A}_\circ-\bm{I}_r)\|_{\mathrm{op}}}{nt}
 \le \frac{\|\bm{C}_\circ\|_{\mathrm{op}}}{nt}
           \|\bm{A}_\circ-\bm{I}_r\|_{\mathrm{op}}
   =O_p(v_n^{3/2}),\\
 &\frac{\|\bm{T}_\circ\bm{H}_\perp\|_{\mathrm{op}}}{nt}
 \le \frac{\|\bm{T}_\circ\|_{\mathrm{op}}}{nt}\|\bm{H}_\perp\|_{\mathrm{op}}
   =o_p(\sqrt{v_n}),\\
 &\frac{\|\bm{H}_\perp\bm{E}_J\|_{\mathrm{op}}}{nt}
 \le \|\bm{H}_\perp\|_{\mathrm{op}}\frac{\|\bm{E}_J\|_{\mathrm{op}}}{nt}
   =o_p(\sqrt{v_n}).
\end{align*}
The bound $O_p(v_n^{3/2})$ is also $o_p(\sqrt{v_n})$ since $v_n\to0$.
Put $\bm{H}_0=(nt)^{-1}\bm{C}_\circ\bm{D}_0^{-1}$.
Since $\|\bm{D}_0^{-1}\|_{\mathrm{op}}=1$,
\begin{equation*}
 \|\bm{H}_0\|_{\mathrm{op}}=O_p(\sqrt{v_n}),\qquad
 \|\bm{H}_\perp-\bm{H}_0\|_{\mathrm{op}}=o_p(\sqrt{v_n}).
\end{equation*}
Consequently,
\begin{align*}
 &\|\bm{H}_\perp^{\mathsf{T}}\bm{H}_\perp-\bm{H}_0^{\mathsf{T}}\bm{H}_0\|_{\mathrm{op}}
 \le (\|\bm{H}_\perp\|_{\mathrm{op}}+\|\bm{H}_0\|_{\mathrm{op}})
                           \|\bm{H}_\perp-\bm{H}_0\|_{\mathrm{op}}\\
 &\quad =o_p(v_n).
\end{align*}
Using $\bm{B}_\circ^2=\bm{H}_\perp^{\mathsf{T}}\bm{H}_\perp$ and \eqref{eq:gauss-C-energy} now yields
\begin{align*}
 &\bm{B}_\circ^2
 =\frac1{n^2t^2}\bm{D}_0^{-1}\bm{C}_\circ^{\mathsf{T}}\bm{C}_\circ\bm{D}_0^{-1}
                                                +o_p(v_n)\\
 &\quad =\frac1t\bm{D}_0^{-1}+\frac{p}{nt^2}\bm{D}_0^{-2}+o_p(v_n)
  =\bm{\Omega}_n(t_n)+o_p(v_n),
\end{align*}
since $n(nt+p)/(n^2t^2)=v_n$.
No commutation with $\bm{D}_0$ is used.
Together with the bound on $\|\bm{A}_\circ-\bm{I}_r\|_{\mathrm{op}}$, the covariance expansion proves \eqref{eq:gauss-residual-cov} and \eqref{eq:gauss-alignment-bound}.

For the rowwise approximation, set
\begin{equation*}
 \bm{X}_n=\frac{\bm{B}_\circ}{\sqrt{v_n}},\qquad
 \bm{W}_n=\left(\frac{\bm{\Omega}_n(t_n)}{v_n}\right)^{1/2}.
\end{equation*}
The covariance bound in \eqref{eq:gauss-residual-cov} gives $\|\bm{X}_n^2-\bm{W}_n^2\|_{\mathrm{op}}=o_p(1)$.
By \eqref{eq:gauss-comparability}, the eigenvalues of $\bm{W}_n^2$ lie in $[d_{\max}^{-2},1]$.
Thus, with probability tending to one, $\bm{X}_n,\bm{W}_n\succeq c\bm{I}_r$ for fixed $c>0$.
Set $\bm{\Delta}_n=\bm{X}_n-\bm{W}_n$.
The Sylvester equation
\begin{equation}
 \bm{X}_n\bm{\Delta}_n+\bm{\Delta}_n\bm{W}_n=\bm{X}_n^2-\bm{W}_n^2
 \label{eq:gauss-sylvester}
\end{equation}
has the integral solution \citep{BhatiaRosenthal1997}
\begin{equation}
 \bm{\Delta}_n=\int_0^\infty
 e^{-u\bm{X}_n}(\bm{X}_n^2-\bm{W}_n^2)e^{-u\bm{W}_n}\,du.
 \label{eq:gauss-sylvester-integral}
\end{equation}
For $\bm{F}=\bm{X}_n^2-\bm{W}_n^2$,
\begin{equation*}
 -\frac{d}{du}\bigl(e^{-u\bm{X}_n}\bm{F} e^{-u\bm{W}_n}\bigr)
 =\bm{X}_n e^{-u\bm{X}_n}\bm{F} e^{-u\bm{W}_n}
  +e^{-u\bm{X}_n}\bm{F} e^{-u\bm{W}_n}\bm{W}_n.
\end{equation*}
On $\{\bm{X}_n,\bm{W}_n\succeq c\bm{I}_r\}$, $\|e^{-u\bm{X}_n}\|_{\mathrm{op}}\|e^{-u\bm{W}_n}\|_{\mathrm{op}}\le e^{-2cu}$.
Integration verifies that \eqref{eq:gauss-sylvester-integral} solves \eqref{eq:gauss-sylvester}, and uniqueness follows from $\lambda_i(\bm{X}_n)+\lambda_j(\bm{W}_n)\ge2c$ in the eigenbases of $\bm{X}_n$ and $\bm{W}_n$.
Equation~\eqref{eq:gauss-sylvester-integral} gives
\begin{equation*}
 \|\bm{\Delta}_n\|_{\mathrm{op}}
 \le\frac{1}{2c}\|\bm{X}_n^2-\bm{W}_n^2\|_{\mathrm{op}}=o_p(1).
\end{equation*}
Rescaling gives
\begin{equation*}
 \|\bm{B}_\circ-\bm{\Omega}_n(t_n)^{1/2}\|_{\mathrm{op}}=o_p(\sqrt{v_n}).
\end{equation*}
Apply \eqref{eq:gauss-row-Haar}, $\|\bm{q}_k\|_2\le q_{\max}$, and $\max_i\|\bm{z}_i\|_2=O_p(\sqrt{\log n})$ to obtain
\begin{align*}
 &\max_i\|\sqrt n\,\widetilde V_{i*}
              -\bm{q}_{Z_i}^{\mathsf{T}}-\bm{z}_i^{\mathsf{T}}\bm{\Omega}_n(t_n)^{1/2}\|_2\\
 &\quad\le
 \max_i\|\sqrt n\,\widetilde V_{i*}
              -\bm{q}_{Z_i}^{\mathsf{T}}\bm{A}_\circ-\bm{z}_i^{\mathsf{T}}\bm{B}_\circ\|_2
       +q_{\max}\|\bm{A}_\circ-\bm{I}_r\|_{\mathrm{op}}\\
 &\qquad\quad+
       \max_i\|\bm{z}_i\|_2\,
                 \|\bm{B}_\circ-\bm{\Omega}_n(t_n)^{1/2}\|_{\mathrm{op}}\\
 &\quad=O_p\!\left(\sqrt{\frac{\log n}{n}}\right)
         +O_p(v_n)+o_p(\sqrt{v_n\log n})=o_p(1),
\end{align*}
where $\theta_n=v_n\log n=O(1)$ makes the rowwise error $o_p(1)$, proving \eqref{eq:gauss-critical-coupling}.
\end{proof}

\subsubsection{Extreme observations and signals below the critical scale}
Let $\overline B_r(\bm0,R)=\{\bm{x}\in\mathbb{R}^r:\|\bm{x}\|_2\le R\}$ denote the closed Euclidean ball.
For nonempty compact sets, $d_{\mathrm H}$ denotes Hausdorff distance, the maximum of the two directed distances.
For general sample-cloud limits, see \citet{BalkemaNolde2010}.

\begin{lemma}[Hausdorff convergence of Gaussian sample clouds]
\label{lem:gauss-extreme-cloud}
Let $\bm{z}_1,\ldots,\bm{z}_n$ be independent $N(\bm0,\bm{I}_r)$ vectors, independent of the latent labels, with $r,K,\pi_1,\ldots,\pi_K$ fixed.
Then
\begin{equation*}
 \max_{1\le k\le K}d_{\mathrm H}\!\left(
  \{\bm{z}_i/\sqrt{\log n}:Z_i=k\},
                 \overline B_r(\bm0,\sqrt2)\right)\xrightarrow{p}0.
\end{equation*}
If the point set of a component is empty, we interpret the Hausdorff distance of the empty point set to $\overline B_r(\bm0,\sqrt2)$ as $+\infty$.
Every fixed nonempty open subset of the open ball contains a rescaled point from each group with probability tending to one.
The containment assertion holds simultaneously for any fixed finite collection of such open subsets.
\end{lemma}
\begin{proof}
For $a>\sqrt2$, radial integration gives
\begin{equation}
 \mathbb{P}\{\max_i\|\bm{z}_i\|_2>a\sqrt{\log n}\}
 \le C_r n\{1+a\sqrt{\log n}\}^{r}\,n^{-a^2/2}\to0.
\label{eq:gauss-cloud-outer-bound}
\end{equation}
For a ball $B(\bm{u},\delta)$ with $\|\bm{u}\|_2+\delta<\sqrt2$, integrating the Gaussian density gives
\begin{equation}
 \mathbb{P}\{\bm{z}_i/\sqrt{\log n}\in B(\bm{u},\delta)\}
 \ge c_{r,\delta}(\log n)^{r/2}
            n^{-(\|\bm{u}\|_2+\delta)^2/2}.
\label{eq:gauss-cloud-ball-probability}
\end{equation}
Write $\beta_n=\mathbb{P}\{\bm{z}_i/\sqrt{\log n}\in B(\bm{u},\delta)\}$.
By \eqref{eq:gauss-cloud-ball-probability}, the expected number of rescaled observations in $B(\bm{u},\delta)$ from group $k$ is $n\pi_k\beta_n$ and tends to infinity.
Independence gives
\begin{equation}
 \begin{aligned}
 &\mathbb{P}\{\text{there is no }i\text{ such that }Z_i=k
          \text{ and }\bm{z}_i/\sqrt{\log n}\in B(\bm{u},\delta)\}\\
 &\qquad=(1-\pi_k\beta_n)^n
       \le \exp(-n\pi_k\beta_n)\to0.
 \end{aligned}
\label{eq:gauss-cloud-empty-ball}
\end{equation}

Fix $\varepsilon>0$ and choose $\bm{u}_1,\ldots,\bm{u}_m$ and $0<\delta<\varepsilon/2$ such that
\begin{equation*}
 \overline B_r(\bm0,\sqrt2)
 \subseteq\bigcup_{j=1}^m B(\bm{u}_j,\varepsilon/2),
 \qquad
 \max_j\|\bm{u}_j\|_2+\delta<\sqrt2.
\end{equation*}
Equation~\eqref{eq:gauss-cloud-empty-ball} and a finite union bound place a point from every group in each $B(\bm{u}_j,\delta)$ with probability tending to one.
The finite cover gives a directed distance below $\varepsilon$ from the closed ball to each group cloud; \eqref{eq:gauss-cloud-outer-bound} with $a=\sqrt2+\varepsilon$ controls the reverse distance.
Each prescribed nonempty open subset contains an interior ball, so \eqref{eq:gauss-cloud-empty-ball} and a finite union bound give a point from every group in every prescribed subset with probability tending to one.
\end{proof}

\begin{lemma}[Bounded effective signal implies nonvanishing residual]
\label{lem:gauss-bounded-I}
For every fixed $M<\infty$, along every sequence with $\iota_n\le M$, there is $c_M>0$ such that
\begin{equation*}
 \mathbb{P}\{\|\bm{B}_\circ\|_{\mathrm{op}}\ge c_M\}\to1.
\end{equation*}
\end{lemma}
\begin{proof}
Put $h=\|\bm{H}_\perp\|_{\mathrm{op}}=\|\bm{B}_\circ\|_{\mathrm{op}}$, so that $\sigma_{\min}(\bm{A}_\circ)=\sqrt{1-h^2}$.
First suppose $t+\gamma\ge\epsilon_0$ for small fixed $\epsilon_0>0$.
Then $nt+p\to\infty$.
By \eqref{eq:gauss-C-energy} and $\bm{D}_0\succeq\bm{I}_r$, with probability tending to one,
\begin{equation*}
 \|\bm{C}_\circ\|_{\mathrm{op}}\ge c n\sqrt{t+\gamma},\qquad
 \|\bm{J}_\circ\|_{\mathrm{op}}+\|\bm{T}_\circ\|_{\mathrm{op}}
                                  \le Cn(t+1+\sqrt\gamma).
\end{equation*}
The bound on $\|\bm{J}_\circ\|_{\mathrm{op}}+\|\bm{T}_\circ\|_{\mathrm{op}}$ follows by bounding the full matrix in \eqref{eq:gauss-Gram-blocks}, using $\sqrt{t+\gamma}\le t+1+\sqrt\gamma$ and $\|\bm{F}_\circ\|_{\mathrm{op}}/\{n(t+1+\sqrt\gamma)\}=o_p(1)$.
If $h\le1/2$, then $\sigma_{\min}(\bm{A}_\circ)\ge\sqrt{3}/2$.
The equation in \eqref{eq:gauss-canonical-lower-block} gives
\begin{equation*}
\begin{aligned}
 &\frac{\sqrt{3}}{2}\|\bm{C}_\circ\|_{\mathrm{op}}
 \le\sigma_{\min}(\bm{A}_\circ)\|\bm{C}_\circ\|_{\mathrm{op}}
 \le\|\bm{C}_\circ\bm{A}_\circ\|_{\mathrm{op}}\\
 &\quad =\|\bm{H}_\perp\bm{J}_\circ-\bm{T}_\circ\bm{H}_\perp\|_{\mathrm{op}}
 \le h(\|\bm{J}_\circ\|_{\mathrm{op}}+\|\bm{T}_\circ\|_{\mathrm{op}}).
\end{aligned}
\end{equation*}
Combining the block-norm bounds yields
\begin{equation*}
 h\ge c\frac{\sqrt{t+\gamma}}{t+1+\sqrt\gamma}.
\end{equation*}
Putting $u=\sqrt{t+\gamma}$ and using $\iota_n=t^2/(t+\gamma)\le M$ gives
\begin{equation*}
 t\le\sqrt M\,u,\qquad
 \sqrt\gamma\le u,\qquad
 1\le u/\sqrt{\epsilon_0},
\end{equation*}
and hence $h\ge c_{M,\epsilon_0}>0$.
The case $h>1/2$ is immediate.

Now suppose $t+\gamma<\epsilon_0$, and let $\bm{G}_*$ be the $p\times(n-1)$ standard Gaussian noise matrix in $\mathcal Y$.
Write $s_j(\cdot)$ for singular values in nonincreasing order.
For every unit $\bm{u}\in\mathbb{R}^p$, $\|\bm{G}_*^{\mathsf{T}}\bm{u}\|_2^2\sim\chi_{n-1}^2$.
Applying \eqref{eq:chi-square-tail} with $p$ replaced by $n-1$ and $x=a(n-1)$, for a fixed sufficiently small $a>0$, to a $1/4$-net of the unit sphere of cardinality at most $9^p$, and choosing $\epsilon_0<a/(2\log 9)$, gives
\begin{equation*}
 \left\|
 \frac{\bm{G}_*\bm{G}_*^{\mathsf{T}}}{n-1}-\bm{I}_p
 \right\|_{\mathrm{op}}
 \le 4(\sqrt a+a)<0.1
\end{equation*}
with probability tending to one.
Hence
\begin{equation*}
 s_p(\bm{G}_*)\ge\sqrt{0.9(n-1)}
             \ge0.9\sqrt n
 \qquad(n\ge10).
\end{equation*}
Intersect with $\{\|\bm{S}_n\|_{\mathrm{op}}\le2\}$, whose probability tends to one by Lemma~\ref{lem:gauss-coordinates}.
On $\{\|\bm{S}_n\|_{\mathrm{op}}\le2\}$,
\begin{equation*}
 \|\mathcal Y-\bm{G}_*\|_{\mathrm{op}}=\|\bm{R}_n\|_{\mathrm{op}}
 \le\sqrt{2ntd_{\max}}.
\end{equation*}
Choose $\epsilon_0<1/2$.
Since $\gamma=p/n<\epsilon_0$, we have $r\le p<n-1$ for every integer $n\ge2$ when $t+\gamma<\epsilon_0$.
Singular value perturbation gives
\begin{equation*}
 s_p(\mathcal Y)
 \ge s_p(\bm{G}_*)-\|\mathcal Y-\bm{G}_*\|_{\mathrm{op}}
 \ge\sqrt n\{0.9-\sqrt{2d_{\max}\epsilon_0}\}.
\end{equation*}
Choose $\epsilon_0$ so that $0.9-\sqrt{2d_{\max}\epsilon_0}>0$ and $(0.9-\sqrt{2d_{\max}\epsilon_0})^2-\epsilon_0\ge1/2$.
The eigenvalues of $\bm{J}_\circ$ are the leading $r$ eigenvalues of $\mathcal Y^{\mathsf{T}}\mathcal Y-p\bm{I}_{n-1}$, so
\begin{equation*}
\begin{aligned}
 &\lambda_{\min}(\bm{J}_\circ)
 =s_r(\mathcal Y)^2-p\ge s_p(\mathcal Y)^2-p\\
 &\quad \ge n\{(0.9-\sqrt{2d_{\max}\epsilon_0})^2-\epsilon_0\}
 \ge n/2.
\end{aligned}
\end{equation*}
Also, \eqref{eq:gauss-small-blocks} and \eqref{eq:gauss-small-cross-block} give
\begin{equation*}
 \|nt\bm{D}_{0,n}+\bm{F}_\circ\|_{\mathrm{op}}\le3d_{\max}\epsilon_0 n,
 \qquad
 \|\bm{C}_\circ\|_{\mathrm{op}}\le2\sqrt{d_{\max}\epsilon_0}\,n
\end{equation*}
with probability tending to one.
If $h\le1/2$, the top block equation in \eqref{eq:gauss-canonical-block-equations} would imply
\begin{equation*}
 \frac{\sqrt3}{4}n
 \le\|\bm{A}_\circ\bm{J}_\circ\|_{\mathrm{op}}
 \le\{3d_{\max}\epsilon_0+\sqrt{d_{\max}\epsilon_0}\}\,n,
\end{equation*}
a contradiction for small enough $\epsilon_0$.
Thus $h>1/2$, completing the proof.
\end{proof}

\begin{lemma}[A large residual produces an incorrect between-component sign]
\label{lem:gauss-large-residual}
If $\|\bm{B}_\circ\|_{\mathrm{op}}^2\log n\xrightarrow{p}\infty$, then
\begin{equation}
 \mathbb{P}\{\exists i<j:Z_i\ne Z_j,\ \widehat P_{ij}>0\}\to1.
 \label{eq:gauss-false-cross}
\end{equation}
In particular, \eqref{eq:gauss-false-cross} holds if $\theta_n\to\infty$.
\end{lemma}
\begin{proof}
Work on $\mathcal E_n^{\mathrm G}$ and use Lemma~\ref{lem:gauss-Haar}, conditional on the labels and $\bm{A}_\circ$.
Put $b=\|\bm{B}_\circ\|_{\mathrm{op}}>0$ and choose measurably a unit eigenvector $\bm{u}$ of $\bm{B}_\circ$ with eigenvalue $b$.
The $\bm{z}_i$ remain independent standard Gaussians.
Fix $M_0>4(q_{\max}+1)$ and call $i$ selected if
\begin{equation*}
 \|\bm{z}_i-(M_0/b)\bm{u}\|_2\le1/2.
\end{equation*}
Integrating the Gaussian density gives conditional selection probability at least $c\exp(-C/b^2)$, uniformly for $0<b\le1$.
Every point in the selected ball has norm at most $M_0/b+1/2$.
The Gaussian probability of the selected ball is consequently at least
\begin{equation*}
 (2\pi)^{-r/2}\operatorname{vol}(\overline B_r(\bm0,1/2))
 \exp\!\left\{-\frac12(M_0/b+1/2)^2\right\}
 \ge c\exp(-C/b^2),
\end{equation*}
where $0<b\le1$ and $M_0$ is fixed.
The constants do not depend on the conditioned values of $b$ or $\bm{u}$.

Groups 1 and 2 each contain at least $n\pi_{\min}/2$ observations.
Let
\begin{equation*}
 \mathcal I_k=\{i\in\mathcal G_k:
       \|\bm{z}_i-(M_0/b)\bm{u}\|_2\le1/2\},\qquad k\in\{1,2\}.
\end{equation*}
Set $\mathcal I_1=\mathcal I_2=\varnothing$ on $(\mathcal E_n^{\mathrm G})^{\mathrm c}$.
Conditional independence therefore bounds $\mathbb{P}\{\mathcal I_1=\varnothing\text{ or }\mathcal I_2=\varnothing\mid Z,\bm{A}_\circ\}$ by
\begin{equation*}
 2\exp\{-c n\exp(-C/b^2)\}.
\end{equation*}
On $\{b^2\log n\ge2C\}$, $n\exp(-C/b^2)\ge\sqrt n$.
Integrating the conditional estimate and using $\mathbb{P}\{(\mathcal E_n^{\mathrm G})^{\mathrm c}\}=o(1)$ gives
\begin{equation}
\begin{aligned}
 &\mathbb{P}\{\mathcal I_1=\varnothing\text{ or }\mathcal I_2=\varnothing\}\\
 &\quad\le o(1)+\mathbb{P}\{\mathcal E_n^{\mathrm G}\cap\{b^2\log n<2C\}\}
          +2e^{-c\sqrt n}=o(1).
\end{aligned}
\label{eq:gauss-selected-pair-probability}
\end{equation}

Put $\bm{w}_i=\bm{A}_\circ\bm{q}_{Z_i}+\bm{B}_\circ\bm{z}_i$ and $R_0=q_{\max}+1/2$.
For every index $i$ satisfying $\|\bm{z}_i-(M_0/b)\bm{u}\|_2\le1/2$, we have $\|\bm{w}_i-M_0\bm{u}\|_2\le R_0$ and $\|\bm{w}_i\|_2\le M_0+R_0$.
For two selected points,
\begin{equation*}
 \bm{w}_i^{\mathsf{T}}\bm{w}_j
 \ge M_0^2-2M_0R_0-R_0^2=:m_0>0,
\end{equation*}
since $M_0>4R_0$.
By \eqref{eq:gauss-row-Haar},
\begin{equation*}
 \delta_n:=\max_i
 \|\sqrt n\,\widetilde V_{i*}^{\mathsf{T}}-\bm{w}_i\|_2=o_p(1).
\end{equation*}
On $\{\mathcal I_1\ne\varnothing,\ \mathcal I_2\ne\varnothing\}$, choose $i\in\mathcal I_1$ and $j\in\mathcal I_2$.
For every $i\in\mathcal I_1$ and $j\in\mathcal I_2$,
\begin{equation}
 n\widehat P_{ij}
 \ge m_0-2(M_0+R_0)\delta_n-\delta_n^2>0
\label{eq:gauss-selected-pair-error}
\end{equation}
with probability tending to one.
The uniform row bound in \eqref{eq:gauss-row-Haar} applies to the selected indices without requiring independence from the approximation error.
Reordering the indices if necessary proves \eqref{eq:gauss-false-cross}.

If $\theta_n\to\infty$, every subsequence has a further subsequence along which either $\iota_n\to\infty$ or the sequence $\iota_n$ is bounded.
If $\iota_n\to\infty$, \eqref{eq:gauss-comparability} and \eqref{eq:gauss-residual-cov} give $\|\bm{B}_\circ\|_{\mathrm{op}}^2\ge c v_n$ with probability tending to one, so $\|\bm{B}_\circ\|_{\mathrm{op}}^2\log n\xrightarrow{p}\infty$.
If $\iota_n$ is bounded, Lemma~\ref{lem:gauss-bounded-I} gives $\|\bm{B}_\circ\|_{\mathrm{op}}^2\log n\xrightarrow{p}\infty$.
The subsequence principle completes the proof.
\end{proof}

\subsubsection{Completion of the Gaussian cutoff and the regular-simplex formula}
\begin{proof}[Proof of Theorem~\ref{thm:gaussian-cutoff}]
Lemma~\ref{lem:gauss-geometry} gives existence and uniqueness of the critical root.
Set
\begin{equation*}
 \bm{E}_n=\bm{\Omega}_n(t_n)/v_n.
\end{equation*}
By \eqref{eq:gauss-comparability}, $\bm{E}_n$ lies in the compact set $\{\bm{E}:d_{\max}^{-2}\bm{I}_r\preceq\bm{E}\preceq\bm{I}_r\}$, and homogeneity gives
\begin{equation}
 \frac{\mathfrak d(\bm{\Omega}_n(t_n))}{2\log n}
                         =\frac{\mathfrak d(\bm{E}_n)}{2\theta_n}.
 \label{eq:gauss-compact-ratio}
\end{equation}
Moreover, $0<\mathfrak d(\bm{I}_r)\le\mathfrak d(\bm{E}_n)\le d_{\max}^{2}\mathfrak d(\bm{I}_r)<\infty$.

Whenever $\theta_n=O(1)$, $v_n\to0$, so Lemma~\ref{lem:gauss-residual} gives \eqref{eq:gauss-critical-coupling}.
Set $\bm{u}_i=\sqrt n\,\widetilde V_{i*}^{\mathsf{T}}$ and $\bm{w}_i=\bm{q}_{Z_i}+\bm{\Omega}_n(t_n)^{1/2}\bm{z}_i$.
Then
\begin{equation*}
 \delta_n:=\max_i\|\bm{u}_i-\bm{w}_i\|_2=o_p(1),
 \qquad
 M_n:=\max_i\|\bm{w}_i\|_2
 \le q_{\max}+\sqrt{v_n}\max_i\|\bm{z}_i\|_2=O_p(1).
\end{equation*}
Since $\widetilde{\bm{V}}\widetilde{\bm{V}}^{\mathsf{T}}=\widehat{\bm{P}}$, the Cauchy--Schwarz inequality gives
\begin{equation}
 \max_{i,j}|n\widehat P_{ij}-\bm{w}_i^{\mathsf{T}}\bm{w}_j|
 \le2M_n\delta_n+\delta_n^2=o_p(1).
 \label{eq:gauss-product-error}
\end{equation}

Suppose that
\begin{equation*}
 \liminf_{n\to\infty}
 \frac{\mathfrak d(\bm{\Omega}_n(t_n))}{2\log n}>1.
\end{equation*}
Then $\theta_n$ is bounded.
From any subsequence, extract a further subsequence with $\theta_n\to\theta\in[0,\infty)$ and $\bm{E}_n\to\bm{E}$.
If $\theta=0$, Gaussian maxima give $\max_i\|\bm{w}_i-\bm{q}_{Z_i}\|_2\xrightarrow{p}0$; \eqref{eq:gauss-q-geometry} and \eqref{eq:gauss-product-error} imply success.
For $\theta>0$, put
\begin{equation*}
 \mathcal C_k=\bm{q}_k+\bm{E}^{1/2}\overline B_r(\bm0,\sqrt{2\theta}).
\end{equation*}
On $\mathcal E_n^{\mathrm G}$, the triangle inequality for Hausdorff distance gives, for each $k$,
\begin{equation*}
\begin{aligned}
 &d_{\mathrm H}(\{\bm{w}_i:Z_i=k\},\mathcal C_k)\\
 &\quad\le\sqrt{\log n}\,\|\bm{\Omega}_n(t_n)^{1/2}\|_{\mathrm{op}}
 d_{\mathrm H}\!\left(\{\bm{z}_i/\sqrt{\log n}:Z_i=k\},
                       \overline B_r(\bm0,\sqrt2)\right)\\
 &\qquad+\sqrt2\,
 \|\sqrt{\log n}\,\bm{\Omega}_n(t_n)^{1/2}
                  -\sqrt\theta\,\bm{E}^{1/2}\|_{\mathrm{op}}.
\end{aligned}
\end{equation*}
Since $\sqrt{\log n}\,\bm{\Omega}_n(t_n)^{1/2}\to\sqrt\theta\,\bm{E}^{1/2}$, Lemma~\ref{lem:gauss-extreme-cloud}, \eqref{eq:gauss-compact-ratio}, and continuity give
\begin{equation*}
 h_n:=\max_k d_{\mathrm H}(\{\bm{w}_i:Z_i=k\},\mathcal C_k)\xrightarrow{p}0,
 \qquad 2\theta<\mathfrak d(\bm{E}).
\end{equation*}
By Lemma~\ref{lem:gauss-geometry},
\begin{equation*}
 m:=\min_{1\le k\le\ell\le K}\min_{\substack{\bm{x}\in\mathcal C_k\\\bm{y}\in\mathcal C_\ell}}
       s_{k\ell}\bm{x}^{\mathsf{T}}\bm{y}>0,
 \qquad
 R:=\max_k\sup_{\bm{x}\in\mathcal C_k}\|\bm{x}\|_2<\infty.
\end{equation*}
Approximating each $\bm{w}_i$ by a point in $\mathcal C_{Z_i}$ gives
\begin{equation}
 s_{Z_iZ_j}\bm{w}_i^{\mathsf{T}}\bm{w}_j\ge m-2Rh_n-h_n^2
 \qquad\text{for all }1\le i<j\le n.
\label{eq:gauss-cloud-inner-product-margin}
\end{equation}
Combining \eqref{eq:gauss-product-error} and \eqref{eq:gauss-cloud-inner-product-margin} gives
\begin{equation*}
 \min_{i<j}s_{Z_iZ_j}n\widehat P_{ij}
 \ge m-2Rh_n-h_n^2-2M_n\delta_n-\delta_n^2=m-o_p(1),
\end{equation*}
and hence $\mathbb{P}(\mathcal S_n)\to1$ along the extracted subsequence.

Suppose instead that
\begin{equation*}
 \limsup_{n\to\infty}
 \frac{\mathfrak d(\bm{\Omega}_n(t_n))}{2\log n}<1.
\end{equation*}
Then $\liminf_{n\to\infty}\theta_n>0$.
From any subsequence, extract a further subsequence with $\theta_n\to\theta\in(0,\infty]$ and $\bm{E}_n\to\bm{E}$ when $\theta<\infty$.
For finite $\theta$, $2\theta>\mathfrak d(\bm{E})$, so Lemma~\ref{lem:gauss-geometry} gives interior points $\bm{x}\in\mathcal C_k$, $\bm{y}\in\mathcal C_\ell$ with $s_{k\ell}\bm{x}^{\mathsf{T}}\bm{y}=-m<0$, and $\bm{x}\ne\bm{y}$ if $k=\ell$.
Choose $\delta>0$ so that
\begin{equation*}
 \delta(\|\bm{x}\|_2+\|\bm{y}\|_2)+\delta^2<m/2.
\end{equation*}
If $k=\ell$, also require $2\delta<\|\bm{x}-\bm{y}\|_2$.
By Lemma~\ref{lem:gauss-extreme-cloud}, with probability tending to one there are distinct $i,j$ satisfying
\begin{equation*}
 Z_i=k,\quad Z_j=\ell,\qquad
 \|\bm{w}_i-\bm{x}\|_2<\delta,\quad\|\bm{w}_j-\bm{y}\|_2<\delta.
\end{equation*}
Equation~\eqref{eq:gauss-product-error} gives
\begin{equation*}
\begin{aligned}
 &s_{k\ell}n\widehat P_{ij}
 \le-m+\delta(\|\bm{x}\|_2+\|\bm{y}\|_2)+\delta^2+2M_n\delta_n+\delta_n^2\\
 &\quad <-m/2+2M_n\delta_n+\delta_n^2=-m/2+o_p(1).
\end{aligned}
\end{equation*}
Thus $\mathbb{P}(\mathcal S_n^{\mathrm{c}})\to1$.
No independence between the chosen indices and the error in \eqref{eq:gauss-product-error} is needed.
For $\theta=\infty$, Lemma~\ref{lem:gauss-large-residual} also gives $\mathbb{P}(\mathcal S_n^{\mathrm{c}})\to1$.
The subsequence principle applied to the deterministic probabilities gives both full-sequence conclusions.

Finally, \eqref{eq:gauss-d-scaling} and monotonicity give
\begin{equation*}
 \mathfrak d(\bm{\Omega}_n(t_n))\ge(1+\varepsilon)\,2\log n
 \qquad\bigl(t_n\ge(1+\varepsilon)t_{\mathrm{crit},n}\bigr).
\end{equation*}
For the lower side, apply \eqref{eq:gauss-d-scaling} with $a=(1-\varepsilon)^{-1}$ at $t=(1-\varepsilon)t_{\mathrm{crit},n}$:
\begin{equation}
 \mathfrak d\bigl(\bm{\Omega}_n((1-\varepsilon)t_{\mathrm{crit},n})\bigr)
 \le(1-\varepsilon)\,2\log n.
\label{eq:gauss-subcritical-d-bound}
\end{equation}
Monotonicity extends \eqref{eq:gauss-subcritical-d-bound} to $t_n\le(1-\varepsilon)t_{\mathrm{crit},n}$, proving the cutoff since $L_K=\sigma_n^2t_n$.
Since $\mathbb{P}\{(\mathcal E_n^{\mathrm G})^{\mathrm c}\}=o(1)$, the conclusions hold under the joint law.
\end{proof}

\begin{proof}[Proof of Corollary~\ref{cor:gaussian-simplex}]
Here $\bm{q}_k=\bm{m}_k$, $\|\bm{q}_k\|_2^2=r$, and $\bm{q}_k^{\mathsf{T}}\bm{q}_\ell=-1$ for $k\ne\ell$.
Also $\bm{\Omega}_n(t)=v_n(t)\bm{I}_r$.
We compute
\begin{equation}
 \mathfrak d(\bm{I}_r)=R_*^2
   =\begin{cases}
      1,&r=1,\\
      \{r-\sqrt{r^2-1}\}/2,&r\ge2.
     \end{cases}
 \label{eq:gauss-simplex-radius}
\end{equation}
For $r=1$, the centers are $1$ and $-1$.
Radii below one preserve opposite signs, while larger intervals admit positive cross-products.

For $r\ge2$, distinct centers have angle $\alpha=\arccos(-1/r)\in(\pi/2,\pi)$.
For $\|\bm{x}-\bm{q}_k\|_2\le R<\sqrt r$, projection onto the ray through $\bm{x}$ gives
\begin{equation*}
 \bm{q}_k^{\mathsf{T}}\bm{x}\ge\sqrt r(\sqrt r-R)>0,
 \qquad
 r-\frac{(\bm{q}_k^{\mathsf{T}}\bm{x})^2}{\|\bm{x}\|_2^2}\le R^2.
\end{equation*}
Thus the angular deviation is at most $\arcsin(R/\sqrt r)$, and the first possible orthogonal pair satisfies $2\arcsin(R_{*}/\sqrt r)=\alpha-\pi/2$.
Since $\sin\alpha=\sqrt{r^2-1}/r$, the half-angle identity gives
\begin{equation*}
 R_*^2=r\sin^2\!\left(\frac{\alpha-\pi/2}{2}\right)
       =\frac r2(1-\sin\alpha)
       =\frac{r-\sqrt{r^2-1}}2.
\end{equation*}
For $R<R_*$, the between- and within-component angles are controlled by
\begin{equation*}
 \alpha-2\arcsin(R/\sqrt r)>\pi/2,
 \qquad
 2\arcsin(R/\sqrt r)<\alpha-\pi/2<\pi/2,
\end{equation*}
so all signs are correct.
For attainment, write the two centers as $(s,u)$ and $(s,-u)$ in the plane spanned by the centers, where
\begin{equation*}
 s=\sqrt{\frac{r-1}{2}},\qquad
 u=\sqrt{\frac{r+1}{2}}.
\end{equation*}
For $a=(s+u)/2$, the points $(a,a)$ and $(a,-a)$ are orthogonal, with common squared distance
\begin{equation*}
 (a-s)^2+(a-u)^2=(u-s)^{2}/2=R_*^2.
\end{equation*}
Increasing both first coordinates gives a positive inner product inside any balls of radius exceeding $R_*$.
Thus \eqref{eq:gauss-simplex-radius} is attained first by a between-component error, and
\begin{equation*}
 \frac{2}{R_*^2}
 =\begin{cases}
   2,&r=1,\\
   4(r+\sqrt{r^2-1}),&r\ge2
  \end{cases}
 =C_K.
\end{equation*}
Homogeneity gives $\mathfrak d(\bm{\Omega}_n(t))=R_*^2/v_n(t)$.
Equation~\eqref{eq:gauss-critical-equation} is therefore
\begin{equation}
 \frac{t^2}{t+\gamma_n}=C_K\log n,
 \qquad
 t^2-C_K(\log n)t-C_K\gamma_n\log n=0.
\label{eq:gauss-simplex-quadratic}
\end{equation}
The positive root of \eqref{eq:gauss-simplex-quadratic}, multiplied by $\sigma_n^2$, gives the cutoff and \eqref{eq:gauss-simplex-effective}.
Expansions of the quadratic root for $p/(n\log n)\to0$ and $p/(n\log n)\to\infty$ give the asymptotic forms of $L_{\mathrm{crit},n}$ stated in Remark~\ref{rem:simplex-cutoff-order}.
\end{proof}

\subsubsection{Proof of the optimal threshold theorem}

\begin{proof}[Proof of Theorem~\ref{thm:simplex-threshold}]
Write $a=a_K$, $b=b_K$, and $c=c_K$.
For $-1<\tau<r$, $R_K^2(\tau)$ is defined by \eqref{eq:simplex-threshold-radius}.
The boundary is nonempty, and
\begin{equation*}
 \min_{k,\ell}s_{k\ell}(\bm{q}_k^{\mathsf{T}}\bm{q}_\ell-\tau)
 =\min\{r-\tau,\tau+1\}>0.
\end{equation*}
As in Lemma~\ref{lem:gauss-geometry}, each constraint set is closed and the objective is continuous and coercive, so the infimum is attained and finite.
The positive signed margins at $\bm{x}=\bm{y}=\bm0$ imply $R_K(\tau)>0$.
For $0\le R<R_K(\tau)$, the definition excludes a nonpositive signed margin on the compact product of the two radius-$R$ balls, so
\begin{equation*}
 \min_{1\le k\le\ell\le K}
 \min_{\substack{\|\bm{x}\|_2\le R\\\|\bm{y}\|_2\le R}}
 s_{k\ell}\{(\bm{q}_k+\bm{x})^{\mathsf{T}}(\bm{q}_\ell+\bm{y})-\tau\}>0.
\end{equation*}
Every minimizer lies on the equality boundary: otherwise, simultaneously shrinking $\bm{x}$ and $\bm{y}$ toward zero reaches equality at a strictly smaller objective value.
For $R>R_K(\tau)$, let $(\bm{x},\bm{y})$ be a boundary minimizer for components $k,\ell$, and put $\bm{u}=\bm{q}_k+\bm{x}$ and $\bm{v}=\bm{q}_\ell+\bm{y}$.
The points $\bm{u}$ and $\bm{v}$ lie in the interiors of the radius-$R$ balls centered at $\bm{q}_k$ and $\bm{q}_\ell$, respectively.
If either point is nonzero, interchange the points and their component labels if necessary so that $\bm{v}\ne\bm0$.
In this case, for sufficiently small $\varepsilon>0$, the perturbed point $\bm{u}-\varepsilon s_{k\ell}\bm{v}$ stays in the same ball and satisfies
\begin{equation*}
 s_{k\ell}\{(\bm{u}-\varepsilon s_{k\ell}\bm{v})^{\mathsf{T}}\bm{v}-\tau\}
 =-\varepsilon\|\bm{v}\|_2^2<0
 \qquad(\bm{u}^{\mathsf{T}}\bm{v}=\tau,\ \bm{v}\ne\bm0).
\end{equation*}
If $\bm{u}=\bm{v}=\bm0$, then $\tau=0$ and Lemma~\ref{lem:gauss-geometry} applies; within-component points can be made distinct by a further small perturbation.

For $a^2\le\tau<c^2$ and $\max\{\|\bm{u}-\bm{q}_k\|_2,\|\bm{v}-\bm{q}_\ell\|_2\}\le R$, put $\bm{m}=(\bm{u}+\bm{v})/2$ and $\bm d=(\bm{u}-\bm{v})/2$.
If $k=\ell$ and $R\le c/2$, the parallelogram identity, with $s=\|\bm{m}-\bm{q}_k\|_2\le R$, gives
\begin{equation*}
\begin{aligned}
 &s^2+\|\bm d\|_2^2\le R^2,\\
 &\bm{u}^{\mathsf{T}}\bm{v}
 \ge(c-s)^2-(R^2-s^2)
 =(c-R)^2+2(R-s)(c-R-s)\ge(c-R)^2.
\end{aligned}
\end{equation*}
Equality at $\bm{u}=\bm{v}=(1-R/c)\bm{q}_k$ and $\tau\ge a^2\ge c^2/4$ give $R_{\mathrm w,K}(\tau)=c-\sqrt\tau$.
For $k\ne\ell$, choose orthonormal $\bm{e}_1,\bm{e}_2$ with $\bm{q}_k=a\bm{e}_1+b\bm{e}_2$ and $\bm{q}_\ell=a\bm{e}_1-b\bm{e}_2$.
For $y=\|\bm d\|_2$ and $\bm{u}^{\mathsf{T}}\bm{v}\ge\tau$,
\begin{equation*}
 R^2
 \ge\|\bm{m}-a\bm{e}_1\|_2^2+\|\bm d-b\bm{e}_2\|_2^2
 \ge(\sqrt{\tau+y^2}-a)^2+(b-y)^2=:f_\tau(y),
\end{equation*}
with equality at $\bm{m}=\sqrt{\tau+y^2}\,\bm{e}_1$, $\bm d=y\bm{e}_2$.
Since
\begin{equation*}
\begin{aligned}
 &f_\tau'(y)=2y\left(2-\frac a{\sqrt{\tau+y^2}}\right)-2b,
 \qquad
 f_\tau''(y)=4-\frac{2a\tau}{(\tau+y^2)^{3/2}}\ge2,\\
 &f_\tau'(0)<0<f_\tau'(b),
\end{aligned}
\end{equation*}
the minimizer $y_\tau\in(0,b)$ is unique and
\begin{equation*}
 R_K^2(\tau)=\min\{R_{\mathrm w,K}^2(\tau),R_{\mathrm b,K}^2(\tau)\}.
\end{equation*}
The radii are continuous, and implicit differentiation gives
\begin{equation*}
\begin{aligned}
 &(R_{\mathrm w,K}^2)'(\tau)=1-c/\sqrt\tau<0,
 \qquad
 (R_{\mathrm b,K}^2)'(\tau)=1-a/\sqrt{\tau+y_\tau^2}>0,\\
 &R_{\mathrm b,K}^2(a^2)<f_{a^2}(b)=(c-a)^2=R_{\mathrm w,K}^2(a^2),\\
 &R_{\mathrm w,K}^2(\tau)\to0,
 \qquad
 \liminf_{\tau\to c^2}R_{\mathrm b,K}^2(\tau)\ge(c-a)^2>0.
\end{aligned}
\end{equation*}
Thus $R_K^2$ is uniquely maximized at $\tau_K^\star$.
For $-1<\tau<a^2$, between-component error-set inclusion gives
\begin{equation*}
 R_K(\tau)\le R_{\mathrm b,K}(a^2)<R_K(\tau_K^\star)
 \qquad(-1<\tau<a^2).
\end{equation*}

Fix $-1<\tau<r$ and put $R=R_K(\tau)$, $C=2/R^2$, and $\theta_n=v_n\log n$.
Then
\begin{equation*}
 \bm{\Omega}_n(t_n)=v_n\bm{I}_r,
 \qquad
 \frac{\iota_n}{C\log n}=\frac{R^2}{2\theta_n}.
\end{equation*}
For bounded $\theta_n$, put $\bm{w}_i=\bm{q}_{Z_i}+\sqrt{v_n}\bm{z}_i$.
Equation~\eqref{eq:gauss-critical-coupling} and Lemma~\ref{lem:gauss-extreme-cloud} give a uniform $o_p(1)$ row error and $\max_i\|\bm{w}_i\|_2=O_p(1)$, respectively.
Thus \eqref{eq:gauss-product-error} yields
\begin{equation*}
\begin{aligned}
 &\max_{i,j}|(n\widehat P_{ij}-\tau)-(\bm{w}_i^{\mathsf{T}}\bm{w}_j-\tau)|\\
 &\qquad=\max_{i,j}|n\widehat P_{ij}-\bm{w}_i^{\mathsf{T}}\bm{w}_j|=o_p(1).
\end{aligned}
\end{equation*}
The strict-margin and disjoint-neighborhood arguments in the proof of Theorem~\ref{thm:gaussian-cutoff} therefore give
\begin{equation*}
\begin{aligned}
 &\mathbb{P}(\mathcal S_n(\tau))\to1
 \qquad \text{if }\limsup_{n\to\infty}2\theta_n<R^2,\\
 &\mathbb{P}(\mathcal S_n(\tau))\to0
 \qquad \text{if }\theta_n\to\theta\in(R^2/2,\infty).
\end{aligned}
\end{equation*}
For $\theta_n\to\infty$, the proof of Lemma~\ref{lem:gauss-large-residual} gives $\|\bm{B}_\circ\|_{\mathrm{op}}^2\log n\xrightarrow{p}\infty$.
The conditional selection argument in the proof of Lemma~\ref{lem:gauss-large-residual} applies with a fixed $M_0$ satisfying
\begin{equation*}
 R_0=q_{\max}+\tfrac12,
 \qquad m_0=M_0^2-2M_0R_0-R_0^2>\tau+1.
\end{equation*}
Equations~\eqref{eq:gauss-selected-pair-probability} and~\eqref{eq:gauss-selected-pair-error} give $n\widehat P_{ij}\ge m_0-o_p(1)>\tau$ for selected observations in different groups, with probability tending to one.
Under $\limsup_{n\to\infty}\iota_n/(C\log n)<1$, choose $\delta\in(0,1)$ such that $\theta_n\ge R^2/\{2(1-\delta)\}>R^2/2$ for all sufficiently large $n$.
Every subsequence has a further subsequence along which $\theta_n$ converges to a finite value in $[R^2/\{2(1-\delta)\},\infty)$ or tends to infinity.
Along each extracted subsequence, the corresponding failure argument applies, so the subsequence principle yields $\mathbb{P}(\mathcal S_n(\tau))\to0$ for the full sequence.
Lemma~\ref{lem:gauss-coordinates} and the subsequence principle give the cutoff under the joint law.
Take $\tau=\tau_K^\star$ and replace $C_K$ by $C_{K,\mathrm{opt}}^{\mathrm{thr}}$ in \eqref{eq:gauss-simplex-quadratic}.

For $\tau\notin(-1,r)$, work on $\mathcal E_n^{\mathrm G}$, where all $n_k\ge2$ for sufficiently large $n$.
For each group, positive semidefiniteness implies
\begin{equation*}
 2\sum_{\substack{i<j\\i,j\in\mathcal G_k}}\widehat P_{ij}
 \le\sum_{\substack{i<j\\i,j\in\mathcal G_k}}
       (\widehat P_{ii}+\widehat P_{jj})
 =(n_k-1)\sum_{i\in\mathcal G_k}\widehat P_{ii}.
\end{equation*}
On $\mathcal S_n(\tau)$ with $\tau\ge r$, positive semidefiniteness gives the contradiction
\begin{equation*}
 r=\mathrm{tr}(\widehat{\bm{P}})
 \ge\sum_{k=1}^K\frac{2}{n_k-1}
       \sum_{\substack{i<j\\i,j\in\mathcal G_k}}\widehat P_{ij}
 >\tau\ge r.
\end{equation*}
On $\mathcal S_n(\tau)$ with $\tau\le-1$, \eqref{eq:oracle-projection-entries} and $\widehat{\bm{P}}\mathbf{1}_n=\bm0$ give
\begin{equation*}
 r\ge\mathrm{tr}(\widehat{\bm{P}}\bm{P})
 =-\sum_{k=1}^K\frac1{n_k}
       \sum_{\substack{i\in\mathcal G_k\\j\notin\mathcal G_k}}
              \widehat P_{ij}
 \ge-\tau r\ge r.
\end{equation*}
The equality $\mathrm{tr}(\widehat{\bm{P}}\bm{P})=r$ forces
\begin{equation*}
 \tau=-1,\qquad
 \|\widehat{\bm{P}}-\bm{P}\|_F^2=2r-2\mathrm{tr}(\widehat{\bm{P}}\bm{P})=0,
 \qquad \bm{B}_\circ=\bm0,
\end{equation*}
contrary to $\bm{B}_\circ\succ\bm0$ established before Lemma~\ref{lem:gauss-Haar}.
Thus $\mathbb{P}(\mathcal S_n(\tau))\to0$ outside $(-1,r)$, completing the proof of global optimality.
For each fixed $K$, optimality means minimizing the asymptotic cutoff coefficient over common scalar thresholds $\tau$ independent of $n$ in the rule $n\widehat P_{ij}>\tau$.

At $\tau=\tau_K^\star$, put $x=\sqrt\tau$, $y=y_\tau$, $s=\sqrt{\tau+y^2}$, and $t=y/s$.
From $f_\tau'(y)=0$ and $R_{\mathrm w,K}^2(\tau)=R_{\mathrm b,K}^2(\tau)$,
\begin{equation*}
\begin{aligned}
 &b=y(2-a/s),\qquad
 (c-x)^2=(s-a)^2+(b-y)^2,\qquad cx=as+by-y^2,\\
 &s=(a+b/t)/2,\qquad
 2c\sqrt{1-t^2}=a(2-t^2)+bt.
\end{aligned}
\end{equation*}
With $t=2w/(1+w^2)$, $0<w<1$, and $h=b/(c+a)$,
\begin{equation*}
\begin{aligned}
 &\sqrt{1-t^2}=\frac{1-w^2}{1+w^2},
 \qquad\frac{c-a}{c+a}=\frac{b^2}{(c+a)^2}=h^2,\\
 &0=(c-a)-(c+a)w^4-bw(1+w^2)
 =-(c+a)(w^4+hw^3+hw-h^2).
\end{aligned}
\end{equation*}
Hence
\begin{equation*}
 P_h(w):=w^4+hw^3+hw-h^2=0,
 \qquad
 P_h'(w)>0\ (w\ge0),\qquad P_h(0)<0<P_h(h).
\end{equation*}
Hence $w=w_K$; using $a/c=(1-h^2)/(1+h^2)$, $b/c=2h/(1+h^2)$, and $P_h(w)=0$ gives
\begin{equation*}
 (1-h^2)w+h(1+w^2)-(1-w^4)(w+h)=wP_h(w)=0
\end{equation*}
and hence
\begin{equation*}
 \frac xc
 =\frac{(1-w^2)\{(1-h^2)w+h(1+w^2)\}}
 {2(1+h^2)w(1+w^2)}
 =\frac{(1-w^2)^2(w+h)}{2(1+h^2)w}.
\end{equation*}
Strict monotonicity of $P_h$ gives $w_K\to w_\infty$ as $h_K\to h_\infty=\sqrt2-1$.
Since $(c-x)^2=R_{\mathrm b,K}^2(\tau)>(x-a)^2$,
\begin{equation*}
\begin{aligned}[b]
 &\frac xc<\frac{1+a/c}{2}\le\frac{1+1/\sqrt2}{2}<1,
 \qquad
 \frac{\tau_K^\star}{K}=\frac{K-1}{K}\left(\frac xc\right)^2
 \to\alpha_\infty,\\
 &K C_{K,\mathrm{opt}}^{\mathrm{thr}}
 =\frac{2K}{(K-1)(1-x/c)^2}
 \to\frac{2}{(1-\sqrt{\alpha_\infty})^2}.
\end{aligned}
\qedhere
\end{equation*}
\end{proof}

\section*{Acknowledgements}
The author wishes to express his sincere gratitude to
Professor Koji Tsukuda and Yuichi Goto for their insightful comments.

During the preparation of this work the author used
ChatGPT (OpenAI) and Gemini (Google) in order to
improve the presentation.
After using these tools, the author reviewed and
edited the content as needed and takes full
responsibility for the content of the published article.


\begin{thebibliography}{99}

\bibitem[Abbe et al.(2020)]{AbbeFanWangZhong2020}
\textsc{Abbe, E., Fan, J., Wang, K. and Zhong, Y.} (2020).
Entrywise eigenvector analysis of random matrices with low expected rank.
\textit{Ann. Statist.} \textbf{48}, 1452--1474.

\bibitem[Abbe et al.(2022)]{AbbeFanWang2022}
\textsc{Abbe, E., Fan, J. and Wang, K.} (2022).
An $\ell_p$ theory of PCA and spectral clustering.
\textit{Ann. Statist.} \textbf{50}, 2359--2385.

\bibitem[Awasthi et al.(2015)]{AwasthiEtAl2015}
\textsc{Awasthi, P., Bandeira, A. S., Charikar, M., Krishnaswamy, R., Villar, S. and Ward, R.} (2015).
Relax, no need to round: integrality of clustering formulations.
In \textit{Proc. 6th Innovations Theoret. Comput. Sci.}, 191--200. ACM.

\bibitem[Azizyan et al.(2013)]{AzizyanSinghWasserman2013}
\textsc{Azizyan, M., Singh, A. and Wasserman, L.} (2013).
Minimax theory for high-dimensional Gaussian mixtures with sparse mean separation.
In \textit{Adv. Neural Inf. Process. Syst.} \textbf{26}, 2139--2147.

\bibitem[Balkema and Nolde(2010)]{BalkemaNolde2010}
\textsc{Balkema, G. and Nolde, N.} (2010).
Asymptotic independence for unimodal densities.
\textit{Adv. in Appl. Probab.} \textbf{42}, 411--432.

\bibitem[Bhatia and Rosenthal(1997)]{BhatiaRosenthal1997}
\textsc{Bhatia, R. and Rosenthal, P.} (1997).
How and why to solve the operator equation $AX-XB=Y$.
\textit{Bull. Lond. Math. Soc.} \textbf{29}, 1--21.

\bibitem[Cai and Zhang(2018)]{CaiZhang2018}
\textsc{Cai, T. T. and Zhang, A.} (2018).
Rate-optimal perturbation bounds for singular subspaces with applications to high-dimensional statistics.
\textit{Ann. Statist.} \textbf{46}, 60--89.

\bibitem[Cai et al.(2022)]{CaiHanZhang2022}
\textsc{Cai, T. T., Han, R. and Zhang, A. R.} (2022).
On the non-asymptotic concentration of heteroskedastic Wishart-type matrix.
\textit{Electron. J. Probab.} \textbf{27}, no.~29, 1--40.

\bibitem[Cape et al.(2019)]{CapeTangPriebe2019}
\textsc{Cape, J., Tang, M. and Priebe, C. E.} (2019).
The two-to-infinity norm and singular subspace geometry with applications to high-dimensional statistics.
\textit{Ann. Statist.} \textbf{47}, 2405--2439.

\bibitem[Chen and Yang(2021)]{ChenYang2021}
\textsc{Chen, X. and Yang, Y.} (2021).
Cutoff for exact recovery of Gaussian mixture models.
\textit{IEEE Trans. Inform. Theory} \textbf{67}, 4223--4238.

\bibitem[Chen and Zhang(2024)]{ChenZhang2024}
\textsc{Chen, X. and Zhang, A. Y.} (2024).
Achieving optimal clustering in Gaussian mixture models with anisotropic covariance structures.
In \textit{Adv. Neural Inf. Process. Syst.} \textbf{37}, 113698--113741.

\bibitem[Ding and He(2004)]{DingHe2004}
\textsc{Ding, C. H. Q. and He, X.} (2004).
K-means clustering via principal component analysis.
In \textit{Proc. 21st Int. Conf. Mach. Learn.}, 225--232.
ACM.

\bibitem[Eldridge et al.(2018)]{EldridgeBelkinWang2018}
\textsc{Eldridge, J., Belkin, M. and Wang, Y.} (2018).
Unperturbed: spectral analysis beyond Davis--Kahan.
In \textit{Proc. Mach. Learn. Res.} \textbf{83}, 321--358.

\bibitem[Fei and Chen(2018)]{FeiChen2018}
\textsc{Fei, Y. and Chen, Y.} (2018).
Hidden integrality of SDP relaxations for sub-Gaussian mixture models.
In \textit{Proc. Mach. Learn. Res.} \textbf{75}, 1931--1965.

\bibitem[Jauch et al.(2021)]{JauchHoffDunson2021}
\textsc{Jauch, M., Hoff, P. D. and Dunson, D. B.} (2021).
Monte Carlo simulation on the Stiefel manifold via polar expansion.
\textit{J. Comput. Graph. Statist.} \textbf{30}, 622--631.

\bibitem[Jiang(2006)]{Jiang2006}
\textsc{Jiang, T.} (2006).
How many entries of a typical orthogonal matrix can be approximated by independent normals?
\textit{Ann. Probab.} \textbf{34}, 1497--1529.

\bibitem[Kawamoto et al.(2025)]{KawamotoGotoTsukuda2025}
\textsc{Kawamoto, K., Goto, Y. and Tsukuda, K.} (2025).
Spectral clustering algorithm for the allometric extension model.
\textit{Statist. Papers} \textbf{66}, Art.~63.

\bibitem[Kawamoto et al.(2026)]{KawamotoGotoTsukuda2026}
\textsc{Kawamoto, K., Goto, Y. and Tsukuda, K.} (2026).
On spectral clustering under non-isotropic Gaussian mixture models.
\textit{Statist. Probab. Lett.} \textbf{239}, 110894.

\bibitem[Laurent and Massart(2000)]{LaurentMassart2000}
\textsc{Laurent, B. and Massart, P.} (2000).
Adaptive estimation of a quadratic functional by model selection.
\textit{Ann. Statist.} \textbf{28}, 1302--1338.

\bibitem[Little et al.(2023)]{LittleXieSun2023}
\textsc{Little, A., Xie, Y. and Sun, Q.} (2023).
An analysis of classical multidimensional scaling with applications to clustering.
\textit{Inf. Inference} \textbf{12}, 72--112.

\bibitem[L\"offler et al.(2021)]{LofflerZhangZhou2021}
\textsc{L\"offler, M., Zhang, A. Y. and Zhou, H. H.} (2021).
Optimality of spectral clustering in the Gaussian mixture model.
\textit{Ann. Statist.} \textbf{49}, 2506--2530.

\bibitem[Mixon et al.(2017)]{MixonVillarWard2017}
\textsc{Mixon, D. G., Villar, S. and Ward, R.} (2017).
Clustering subgaussian mixtures by semidefinite programming.
\textit{Inf. Inference} \textbf{6}, 389--415.

\bibitem[Ndaoud(2022)]{Ndaoud2022}
\textsc{Ndaoud, M.} (2022).
Sharp optimal recovery in the two component Gaussian mixture model.
\textit{Ann. Statist.} \textbf{50}, 2096--2126.

\bibitem[Paul(2007)]{Paul2007}
\textsc{Paul, D.} (2007).
Asymptotics of sample eigenstructure for a large dimensional spiked covariance model.
\textit{Statist. Sinica} \textbf{17}, 1617--1642.

\bibitem[Peng and Wei(2007)]{PengWei2007}
\textsc{Peng, J. and Wei, Y.} (2007).
Approximating K-means-type clustering via semidefinite programming.
\textit{SIAM J. Optim.} \textbf{18}, 186--205.

\bibitem[Royer(2017)]{Royer2017}
\textsc{Royer, M.} (2017).
Adaptive clustering through semidefinite programming.
In \textit{Adv. Neural Inf. Process. Syst.} \textbf{30}, 1795--1803.

\bibitem[Rudelson and Vershynin(2013)]{RudelsonVershynin2013}
\textsc{Rudelson, M. and Vershynin, R.} (2013).
Hanson--Wright inequality and sub-Gaussian concentration.
\textit{Electron. Commun. Probab.} \textbf{18}, no.~82, 1--9.

\bibitem[Vershynin(2018)]{Vershynin2018}
\textsc{Vershynin, R.} (2018).
\textit{High-Dimensional Probability: An Introduction with Applications in Data Science}.
Cambridge University Press.

\bibitem[Zha et al.(2001)]{ZhaHeDingGuSimon2001}
\textsc{Zha, H., He, X., Ding, C., Gu, M. and Simon, H. D.} (2001).
Spectral relaxation for k-means clustering.
In \textit{Adv. Neural Inf. Process. Syst.} \textbf{14}, 1057--1064.

\bibitem[Zhang and Zhou(2024)]{ZhangZhou2024}
\textsc{Zhang, A. Y. and Zhou, H. Y.} (2024).
Leave-one-out singular subspace perturbation analysis for spectral clustering.
\textit{Ann. Statist.} \textbf{52}, 2004--2033.

\end{thebibliography}
\end{document}